%% file: LinearConditionalMeanPaper.tex
\documentclass[a4paper, twoside]{scrartcl}
\input{Packages}

\input{Macros}

\usepackage{etex}
\usepackage[a4paper, top=88pt, bottom=88pt, left=72pt, right=72pt, headsep=16pt, footskip=28pt]{geometry}
\usepackage{hyperref}
\hypersetup{
	colorlinks,
	linkcolor=blue,
	citecolor=blue,
	urlcolor=blue,
}
\newcommand*{\doi}[1]{\bgroup\color{blue}\href{https://doi.org/#1}{doi:#1}\egroup}
\newcommand*{\email}[1]{\bgroup\color{blue}\href{mailto:#1}{#1}\egroup}
\renewcommand*{\url}[1]{\bgroup\color{blue}\href{#1}{#1}\egroup}
\usepackage{enumitem, moreenum}
\setlist{nosep}
\usepackage{mleftright} \mleftright
\renewcommand{\qedsymbol}{$\blacksquare$}
\renewenvironment{proof}[1][\proofname]{\noindent{\bfseries\sffamily #1.} }{\hfill\qedsymbol\medskip}
\usepackage[labelfont={sf,bf}]{caption}
\usepackage{scrlayer-scrpage, xhfill}
\automark[section]{section}
\setkomafont{pageheadfoot}{\normalcolor\sffamily}
\setkomafont{pagenumber}{\normalfont\normalsize\sffamily}
\clearpairofpagestyles
\let\oldtitle\title
\renewcommand{\title}[1]{\oldtitle{#1}\newcommand{\theshorttitle}{#1}}
\newcommand{\shorttitle}[1]{\renewcommand{\theshorttitle}{#1}}
\let\oldauthor\author
\renewcommand{\author}[1]{\oldauthor{#1}\newcommand{\theshortauthor}{#1}}
\newcommand{\shortauthor}[1]{\renewcommand{\theshortauthor}{#1}}
\cohead{\xrfill[0.525ex]{0.6pt}~\theshorttitle~\xrfill[0.525ex]{0.6pt}}
\cehead{\xrfill[0.525ex]{0.6pt}~\theshortauthor~\xrfill[0.525ex]{0.6pt}}
\newcommand{\theabstract}[1]{\par\bgroup\noindent\textbf{\textsf{Abstract.}} #1\egroup}
\newcommand{\thekeywords}[1]{\par\smallskip\bgroup\noindent\textbf{\textsf{Keywords.}}\newcommand{\and}{ $\bullet$ } #1\egroup}
\newcommand{\themsc}[1]{\par\smallskip\bgroup\noindent\textbf{\textsf{2010 Mathematics Subject Classification.}}\newcommand{\and}{ $\bullet$ } #1\egroup}

\newcommand*{\affilref}[1]{\ref{affiliation#1}}
\newcommand*{\affiliation}[3]{
	\footnotetext[#1]{\label{affiliation#2}#3}
}

\setlist{topsep=0.3ex, itemsep=0.3ex}

\title{The Linear Conditional Expectation\\in Hilbert Space}
\shorttitle{The Linear Conditional Expectation in Hilbert Space}
\author{%
	Ilja~Klebanov\textsuperscript{\affilref{ZIB}}%
	\and
	Bj\"orn~Sprungk\textsuperscript{\affilref{Freiberg}}%
	\and
	T.~J.~Sullivan\textsuperscript{\affilref{ZIB},\affilref{Warwick}}%
}
\shortauthor{I.~Klebanov, B.~Sprungk, and T.~J.~Sullivan}
\date{\today}

\begin{document}
\maketitle
\affiliation{1}{ZIB}{Zuse Institute Berlin, Takustra{\ss}e 7, 14195 Berlin, Germany (\email{klebanov@zib.de}, \email{sullivan@zib.de})}
\affiliation{2}{Freiberg}{Technische Universit{\"a}t Bergakademie Freiberg,
09596 Freiberg, Germany\newline (\email{bjoern.sprungk@math.tu-freiberg.de})}
\affiliation{3}{Warwick}{Mathematics Institute and School of Engineering, The University of Warwick, Coventry, CV4 7AL, United Kingdom (\email{t.j.sullivan@warwick.ac.uk})}

\begin{abstract}
	\theabstract{\input{./chunk-abstract.tex}}

	\thekeywords{\input{./chunk-keywords.tex}}
	\themsc{\input{./chunk-msc.tex}}
\end{abstract}


\input{Section_01_Introduction}
\input{Section_02_RelatedWork}
\input{Section_03_SetupAndNotation}
\input{Section_04_LinearConditionalMean}
\input{Section_05_CME_alternative}
\input{Section_06_GaussianConditioning}
\input{Section_07_ClosingRemarks}
\input{Section_08_Proofs}

\section*{Acknowledgements}
\addcontentsline{toc}{section}{Acknowledgements}

IK and TJS are supported in part by the Deut\-sche For\-schungs\-ge\-mein\-schaft (DFG) through project TrU-2 ``Demand modelling and control for e-commerce using RKHS transfer operator approaches'' of the Excellence Cluster ``MATH+ The Berlin Mathematics Research Centre'' (EXC-2046/1, project \href{https://gepris.dfg.de/gepris/projekt/390685689}{390685689}).
TJS is further supported by the DFG project \href{https://gepris.dfg.de/gepris/projekt/415980428}{415980428}.
BS has been supported by the DFG project \href{https://gepris.dfg.de/gepris/projekt/389483880}{389483880}.

\appendix
\input{Section_AppendixTechnicalResults}

\bibliographystyle{abbrvnat}
\bibliography{myBibliography}
\addcontentsline{toc}{section}{References}

\end{document}

%% file: Packages.tex
\usepackage{amsfonts,amssymb,amsmath,amsthm,amstext,amssymb,mathtools}
\usepackage{subcaption,float,color,xcolor,graphicx,enumitem}
\usepackage{dsfont}  
\usepackage[normalem]{ulem} 
\usepackage[authoryear,sort,round]{natbib}  
\usepackage{fullpage} 
\usepackage{hyperref,nicefrac}
\usepackage[noabbrev,capitalise,nosort,nameinlink]{cleveref}

\crefname{assumption}{Assumption}{Assumptions}
\crefname{openproblem}{Open Problem}{Open Problem}
\crefname{notation}{Notation}{Notation}

\usepackage{adjustbox}
\usepackage{tikz}
\usetikzlibrary{bayesnet,er,trees,shapes.symbols,mindmap,arrows,decorations,shapes.misc,shapes.arrows,chains,matrix,positioning,scopes,decorations.pathmorphing,patterns}
\usepackage{tikz-cd} 

\usetikzlibrary{decorations.markings}
\tikzset{
  negate/.style={
    decoration={
      markings,
      mark= at position 0.5 with {
        \node[transform shape] (tempnode) {$/$};
      },
    },
    postaction={decorate},
  },
}

\usepackage{dashbox}

\usepackage{chngcntr}

\usepackage{thmtools,thm-restate}

\usepackage{array}
\newcolumntype{x}[1]{>{\centering\arraybackslash\hspace{0pt}}p{#1}}
\usepackage{booktabs,colortbl,xcolor,xfrac}
\usepackage{multirow}
\usepackage{multicol}

\usepackage{multirow,hhline,graphicx,array}

%% file: Macros.tex
\newcommand*{\cX}{\mathcal{X}}
\newcommand*{\cY}{\mathcal{Y}}
\newcommand*{\cH}{\mathcal{H}}
\newcommand*{\cG}{\mathcal{G}}
\newcommand*{\cF}{\mathcal{F}}

\newcommand*{\cU}{\mathcal{U}}
\newcommand*{\cL}{\mathcal{L}}
\newcommand*{\cC}{\mathcal{C}}
\newcommand*{\cE}{\mathcal{E}}
\newcommand*{\uu}{\prescript{u}{} }
\newcommand*{\Lin}{\mathsf{L}}
\newcommand*{\ALin}{\mathsf{A}}
\newcommand*{\BLin}{\mathsf{L}}
\newcommand*{\ABLin}{\mathsf{A}}
\newcommand*{\HS}{\mathsf{L}_{2}}
\newcommand*{\AHS}{\mathsf{A}_{2}}

\newcommand*{\fm}{\mathfrak{m}}
\newcommand*{\ff}{\mathfrak{f}}
\newcommand*{\fh}{\mathfrak{h}}

\newcommand*{\cJ}{\mathcal{J}}

\newcommand*{\bR}{\mathbb{R}}
\newcommand*{\bN}{\mathbb{N}}
\newcommand*{\bZ}{\mathbb{Z}}
\newcommand*{\bE}{\mathbb{E}}
\newcommand*{\bV}{\mathbb{V}}
\newcommand*{\bP}{\mathbb{P}}

\newcommand*{\cB}{\mathcal{B}}
\newcommand*{\cN}{\mathcal{N}}

\newcommand*{\Cov}{\mathbb{C}\mathrm{ov}}

\makeatletter
\newcommand*\quark{\mathpalette\quark@{.5}}
\newcommand*\quark@[2]{\mathbin{\vcenter{\hbox{\scalebox{#2}{$\; \m@th#1\bullet \;$}}}}}
\makeatother

\makeatletter
\newcommand{\mylabel}[2]{#2\def\@currentlabel{#2}\label{#1}}
\makeatother

\newcommand*{\rd}{\mathrm{d}}

\DeclareMathOperator{\trace}{tr}
\DeclareMathOperator{\spn}{span}
\DeclareMathOperator{\supp}{supp}
\DeclareMathOperator{\Id}{Id}

\DeclareMathOperator*{\argmin}{arg\,min}
\DeclareMathOperator{\ran}{ran}

\DeclareMathOperator{\dom}{dom}

\usepackage{pifont}

\definecolor{darkgreen}{rgb}{0,0.4,0}

\newcommand*{\todo}[1]{\bgroup\color{red}TODO:~#1\egroup}
\newcommand*{\tjs}[1]{\bgroup\color{olive}TJS:~#1\egroup}
\newcommand*{\bs}[1]{\bgroup\color{orange}BS:~#1\egroup}
\newcommand*{\ik}[1]{\bgroup\color{cyan}IK:~#1\egroup}

\DeclarePairedDelimiter\abs{\lvert}{\rvert}%
\DeclarePairedDelimiter\mynorm{\lVert}{\rVert}%
\makeatletter
\let\oldabs\abs
\def\abs{\@ifstar{\oldabs}{\oldabs*}}
\let\oldnorm\mynorm
\def\mynorm{\@ifstar{\oldnorm}{\oldnorm*}}
\makeatother

\newcommand*{\defeq}{\coloneqq}

\theoremstyle{plain}
\newtheorem{theorem}{\sffamily Theorem}[section]
\newtheorem{proposition}[theorem]{\sffamily Proposition}
\newtheorem{lemma}[theorem]{\sffamily Lemma}
\newtheorem{corollary}[theorem]{\sffamily Corollary}
\theoremstyle{definition}
\newtheorem{definition}[theorem]{\sffamily Definition}

\newtheorem{notation}[theorem]{\sffamily Notation}
\newtheorem{example}[theorem]{\sffamily Example}

\newtheorem{remark}[theorem]{\sffamily Remark}

\newtheorem{assumption}[theorem]{\sffamily Assumption}

\newcommand{\absval}[1]{\lvert #1 \rvert}
\newcommand{\innerprod}[2]{\langle #1 , #2 \rangle}
\newcommand{\norm}[1]{\lVert #1 \rVert}

\newcommand{\bignorm}[1]{\bigl\Vert #1 \bigr\Vert}

\newcommand{\Set}[2]{\left\{ #1 \,\middle\vert\, #2 \right\}}

\numberwithin{equation}{section}
\numberwithin{figure}{section}
\numberwithin{table}{section}

\definecolor{verylightgray}{rgb}{0.85,0.85,0.85}

\newcommand\mlnode[1]{\fbox{\begin{tabular}{@{}c@{}}#1\end{tabular}}} 
\newcommand\dashedmlnode[1]{\dbox{\begin{tabular}{@{}c@{}}#1\end{tabular}}}
\newcommand\filledmlnode[1]{\fcolorbox{black}{verylightgray}{\begin{tabular}{@{}c@{}}#1\end{tabular}}}
\definecolor{boxgreen}{rgb}{0,0.7,0}
\newcommand\greenmlnode[1]{\fcolorbox{boxgreen}{white}{\begin{tabular}{@{}c@{}}#1\end{tabular}}}
\definecolor{boxred}{rgb}{0.7,0,0}

\newcommand{\proofinappendix}
{
}

%% file: chunk-abstract.tex
The \emph{linear conditional expectation} (LCE) provides a best linear (or rather, affine) estimate of the conditional expectation and hence plays an important r\^ole in approximate Bayesian inference, especially the \emph{Bayes linear} approach.
This article establishes the analytical properties of the LCE in an infinite-dimensional Hilbert space context.
In addition, working in the space of affine Hilbert--Schmidt operators, we establish a regularisation procedure for this LCE.
As an important application, we obtain a simple alternative derivation and intuitive justification of the \emph{conditional mean embedding} formula, a concept widely used in machine learning to perform the conditioning of random variables by embedding them into reproducing kernel Hilbert spaces.

%% file: chunk-keywords.tex
Bayes linear analysis%
\and
conditional mean embedding%
\and
reproducing kernel Hilbert space%
\and
linear conditional expectation%

%% file: chunk-msc.tex
46E22
\and
28C20
\and
62C10
\and
62J05
\and
62G05

%% file: Section_01_Introduction.tex

\section{Introduction}
\label{section:Introduction}

The crucial step in most inference problems is the approximation of the conditional expectation $\bE[U|V]$, where $U\in L^2(\Omega,\Sigma,\bP;\cG)$ and $V\in L^2(\Omega,\Sigma,\bP;\cH)$ are random variables over some probability space $(\Omega,\Sigma,\bP)$ taking values in some separable Hilbert spaces $\cG$ and $\cH$, respectively.
In Bayesian statistics, where it relates to the posterior mean, $\bE[U|V]$ is an important point estimator of the inferred parameter.
It is well known\footnote{\label{footnote:best_approximation}For $\bR$-valued random variables see e.g.\ \citet[Theorem~10.2.9]{dudley2002real}; the general case follows by choosing orthonormal bases.}
that $\bE[U|V]$ is the best approximation of $U$ by a $\sigma(V)$-measurable random variable within $L^2(\Omega,\sigma(V),\bP;\cG)$ (i.e.\ the orthogonal projection of $U$ onto $L^2(\Omega,\sigma(V),\bP;\cG)$),
\begin{equation}
	\label{equ:ConditionalExpectationMinimizesL2loss}
	\bE[U|V]
	=
	\argmin_{\tilde{U} \in L^2(\Omega,\sigma(V);\cG)} \norm{U - \tilde{U}}_{ L^2(\Omega,\Sigma,\bP;\cG)}
	=
	\argmin_{\tilde{U} \in L^2(\Omega,\sigma(V);\cG)} \bE \bigl[ \norm{U - \tilde{U}}_{\cG}^{2} \bigr].
\end{equation}
By the Doob--Dynkin representation \citep[Lemma~1.13]{kallenberg2006foundations}, the conditional expectation can therefore be rewritten in the form
\begin{equation}
	\label{equ:DoobDynkinAppliedToConditionalExpectation}
	\bE[U|V]
	=
	\gamma_{U|V} \circ V
	\quad
	\text{$\bP$-almost surely,}
\end{equation}
where $\gamma_{U|V}\colon \cH\to\cG$ is a measurable map which we will call the \emph{conditional expectation function} (CEF).
In the language of statistical learning theory (or statistical decision theory), $\gamma_{U|V}$ is called the \emph{regression function} and constitutes a Bayes predictor for the least squares error loss, i.e.\ the predictor with the minimal risk \citep[Section~2.4]{hastie2009elements}, which follows directly from \eqref{equ:ConditionalExpectationMinimizesL2loss}.

While computing $\gamma_{U|V}$, which is the main object of interest, is infeasible in most applications, various estimates can be constructed.
The most prominent approach is to approximate $\gamma_{U|V}$ within the class $\ABLin(\cH;\cG)$ of bounded affine operators\footnote{\label{footnote:TerminologyClashLinearAffine}We note here an unfortunate but seemingly unavoidable clash of terminology:
while the approximate conditional expectation \eqref{equ:CommonLCMformula} is usually called the \emph{linear} conditional expectation in the literature, it in fact corresponds to approximation using \emph{affine} operators.} from $\cH$ to $\cG$, since this provides an explicit formula for the \emph{linear conditional expectation function} (LCEF) $\gamma_{U|V}^{\ABLin}$ under appropriate conditions \citep[Lemma~4.1]{ernst2015analysis}:
\begin{equation}
	\label{equ:CommonLCMformula}
	\gamma_{U|V}^{\ABLin}(v)
	=
	\mu_U + C_{UV} C_{V}^{-1}(v - \mu_V),
\end{equation}
where $\mu_U$ and $\mu_V$ denote the means and $C_{UV}$ and $C_{V}$ denote the cross-covariance and covariance operators of $U$ and $V$, as defined in \Cref{section:GeneralSetupAndNotation}.

\begin{figure}[t]
	\centering
	\begin{subfigure}[b]{0.45\textwidth}
		\centering
		\includegraphics[width=\textwidth]{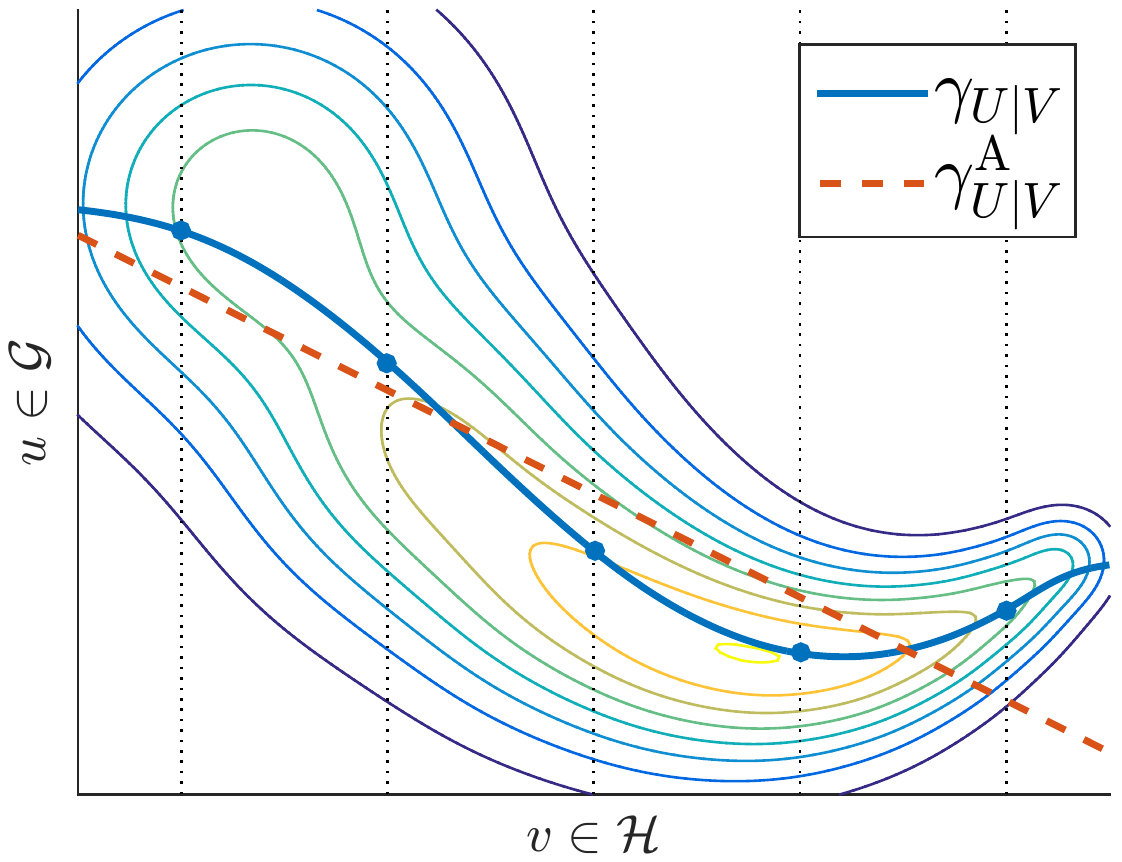}
	\end{subfigure}
	\hspace{1em}
	\begin{subfigure}[b]{0.45\textwidth}
		\centering
		\includegraphics[width=\textwidth]{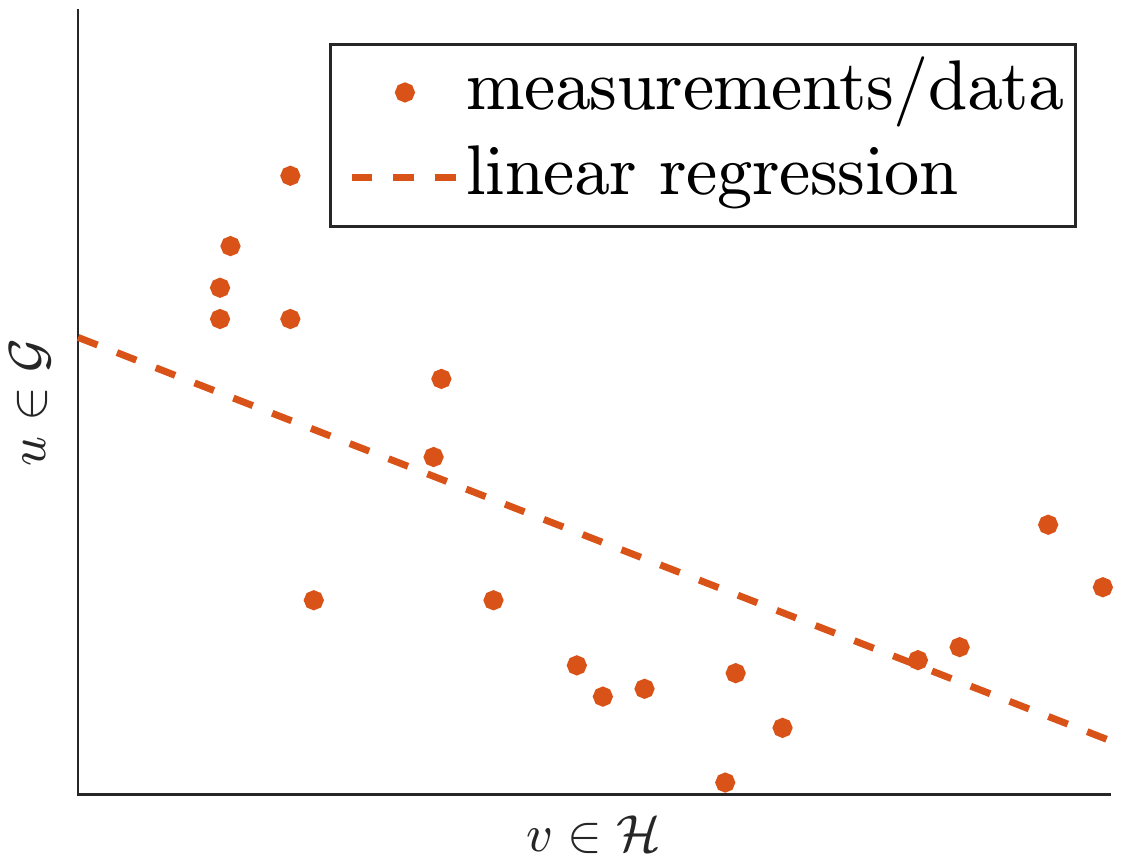}
	\end{subfigure}
	\caption{
		\emph{Left:} Comparison of the conditional expectation function (CEF) $\gamma_{U|V} \colon \cH \to \cG$ and the linear conditional expectation function (LCEF) $\gamma_{U|V}^{\ABLin} \in \ABLin (\cH; \cG)$. The contour plot shows the probability density $\rho_{VU}$ of $(V,U)$.
		\emph{Right:} For an empirical probability distribution (e.g.\ given by data), the LCEF coincides with the solution to the linear least squares regression problem.
	}
	\label{fig:VizualizationCMandLCM}
\end{figure}

While the \emph{linear conditional expectation} (LCE) $\bE^{\ABLin}[U|V]\defeq \gamma_{U|V}^{\ABLin} \circ V$ (also known as \emph{Bayes linear estimator} or \emph{adjusted expectation}) has been discussed extensively by Michael Goldstein and his collaborators in the framework of Bayes linear statistics mostly from an application point of view \citep{GoldsteinWooff2007}, a rigorous mathematical analysis of the LCE is yet to be established, especially for the case of infinite-dimensional $\cG$ and $\cH$.
This level of generality, which this article seeks to provide, yields not just a satisfying mathematical theory but is also necessary for the application of LCE-type methods to problems with high-dimensional unknowns or data, such as time series and functional data analysis.

The first contribution of this paper is to fill this theoretical gap by studying the properties of the LCE and generalising formula \eqref{equ:CommonLCMformula} to infinite-dimensional Hilbert spaces.
Thus far, \eqref{equ:CommonLCMformula} has been derived under the assumptions that $\cH$ is finite-dimensional and that $C_{V}$ is invertible \citep[Section~4]{ernst2015analysis}.
In addition, working in the spaces of (affine) Hilbert--Schmidt operators, we establish a rigorous justification for the regularised version of \eqref{equ:CommonLCMformula}.

Our second contribution is a simple alternative derivation and intuitive explanation of the widely used formula for the conditional mean embedding (CME), a method used in machine learning to perform the conditioning of random variables by embedding them into RKHSs, where it reduces to an affine transformation similar to \eqref{equ:CommonLCMformula}
\citep{song2009hilbert,fukumizu2013kernel,klebanov2019rigorous}.
This result follows almost directly from the fact that, by the reproducing property,
$\bE[U|V]$ coincides with its best affine approximation $\bE^{\ABLin}[U|V]$.

Note that this paper considers only centered \mbox{(cross-)}covariance operators defined by \eqref{equ:DefinitionCovarianceOperators}.
Some, but not all, of the results can be proven similarly for uncentered operators defined by \eqref{equ:DefinitionUncentredCovarianceOperators}, the theory for which is less general, since it allows only for strictly linear instead of affine approximations, i.e., one would be restricted to fitting the probability density or data points in \Cref{fig:VizualizationCMandLCM} with a straight line through the origin.

This paper is structured as follows.
\Cref{section:RelatedWork} briefly surveys related work in statistics, machine learning, and dynamical systems.
\Cref{section:GeneralSetupAndNotation} establishes notation and standing assumptions for the remainder of the paper.
\Cref{section:LinearConditionalMean} forms the core of the paper, in which we study the rigorous generalisation of the LCE to the infinite-dimensional Hilbert space context and also consider multiple formulations of the linear conditional covariance operator.
We analyse their basic properties (\Cref{theorem:BasicPropertiesLCM,counterexample:CounterexamplesPropertiesLCM}) and derive explicit formulae for them in several regimes (\Cref{theorem:LinearConditionalMeanUnderRangeAssumption,theorem:LinearConditionalMeanIncompatibleCase,theorem:LinearConditionalMeanRegularized}).
In \Cref{section:ApplicationToCMEsAlternative,section:GaussianConditioning} these ideas are applied to kernel conditional mean embeddings of random variables into RKHSs and to the conditioning of infinite-dimensional Gaussian random vectors, respectively.
Some closing remarks are given in \Cref{section:ClosingRemarks}, after which all the proofs of results in the main text are given in \Cref{section:Proofs}.
\Cref{section:TechnicalResults} contains technical supporting results.

%% file: Section_02_RelatedWork.tex

\section{Related Work}
\label{section:RelatedWork}

Formula \eqref{equ:CommonLCMformula} is the fundamental solution in linear least squares regression (or general linear models) and can be interpreted as the \emph{best linear unbiased estimate} (BLUE);
see e.g.\ \citet[Sections~2.3.1 and 3.2]{hastie2009elements}.
\Cref{fig:VizualizationCMandLCM} illustrates the connection between $\gamma_{U|V}^{\ABLin}$ and linear regression:
the two coincide if the probability distribution $\bP_{VU}$ of $(V,U)$ is an empirical distribution $\bP_{VU} = J^{-1}\sum_{j=1}^{J} \delta_{(v_j,u_j)}$, where $(v_j,u_j)$, $j=1,\dots,J$, are (or can be thought of as) measurements or data points.

Apart from the connection to linear regression, this work is related to several fields of applied mathematics.
First and foremost, it should be seen as a systematic and rigorous treatment as well as an extension of \emph{Bayes linear analysis}, which has been introduced and investigated by Michael Goldstein and his collaborators, see e.g.\ \citet{Goldstein1999} and \citet{GoldsteinWooff2007} and the references therein.
Furthermore, \citet[p.~9]{Stein1999} offers a Bayesian interpretation of the BLUE in special cases, namely that in the ``uninformative'' infinite-variance limit of a Gaussian prior, the limiting posterior is Gaussian with the BLUE as its conditional expectation.

The LCE is applied in a variety of fields:

In geostatistics, the LCE appears in form of the Kriging estimate for the value of a random field at unexplored locations given available (noisy) data of the random field at measurement locations \citep{ChilesDelfiner2012,Stein1999}.

In data assimilation, formula \eqref{equ:CommonLCMformula} defines the update scheme of the K\'alm\'an filter and its many variants, including the ensemble K\'alm\'an filter \citep{Evensen2009b}.
Although this update rule is typically interpreted as a Gaussian approximation ---  
``in the large ensemble size limit the EnKF [\dots]\ does not reproduce the filtering distribution, except in the linear Gaussian case'' \citep{Schillings2017} --- it has been argued by \citet[Section~4]{ernst2015analysis} that it should rather be seen as the best linear approximation of the required conditional expectations.

In machine learning, the method of \emph{conditional mean embedding} (CME; \citealp{song2009hilbert,fukumizu2013kernel}) applies the conditioning formula \eqref{equ:CommonLCMformula} to random variables embedded into RKHSs, where it becomes exact (i.e.\ $\gamma_{U|V}^{\ABLin} = \gamma_{U|V}$) under certain conditions;
see \citet{klebanov2019rigorous}.
\Cref{section:ApplicationToCMEsAlternative} provides an alternative derivation of the CME formula based on linear conditional expectations and thereby a natural justification of CMEs based in BLUEs;
to the best of our knowledge, this connection has not been made before.

In the field of dynamical systems, the LCE is an important estimate of the \emph{Koopman operator}
\begin{align*}
	\mathcal{K}_{\tau} & \colon L^{\infty}(\cX) \to L^{\infty}(\cX),
	&
	\mathcal{K}_{\tau} f (x)
	& \defeq
	\bE[f(X_{t+\tau}) | X_{t} = x] ,
\end{align*}
of a time-homogeneous Markov chain $(X_t)_{t\in\mathcal{T}}$, where $\tau>0$, $\cX\subseteq\bR^{d}$, and $\mathcal{T}$ is typically $\bZ$, $\bN$, $\bR$, or $[0, \infty)$.
Independently of one another, the dynamical systems, fluid dynamics, and molecular dynamics communities have developed data-driven dimensionality reduction techniques based on the eigendecomposition of this approximation, resulting in the methods of \emph{time-lagged independent component analysis} (TICA) and \emph{dynamic mode decomposition} (DMD).
More generally, the data can be transformed to some feature space $\tilde{\cX}$ by a feature map $\psi\colon\cX\to \tilde{\cX}$ prior to computing the linear approximation, resulting in the \emph{variational approach of conformation dynamics} (VAC) or \emph{empirical dynamic mode decomposition} (EDMD).
The connections among these methods have been discussed in detail by \citet{klus2018data};
see also the references therein to the original papers on these methods.
As in the case of CMEs, if the feature map $\psi$ is the canonical feature map $\psi(x) = k(x,\quark)$ of an RKHS $\cH$ with reproducing kernel $k$, then our analysis is applicable and it reveals the exactness of conditioning formula \eqref{equ:CommonLCMformula}, thereby annihilating one of the main sources of error --- the other being the approximation of the means and covariance operators $\mu_U$, $\mu_V$, $C_{V}$, and $C_{UV}$.
The resulting kernelised versions of VAC and EDMD have been studied by \citet{schwantes2015modeling} and \citet{klus2020eigendecomposition};
the exactness of \eqref{equ:CommonLCMformula}, which we establish, strengthens the analytical power of these methods.
We also mention that there are time-inhomogeneous variants of these methods, namely \emph{coherent mode decomposition} (CMD) and the \emph{variational approach for Markov processes} (VAMP) and their kernelised version \emph{kernel canonical correlation analysis} (kernel CCA; \citealp{klus2019kernel});
again, the established exactness of formula \eqref{equ:CommonLCMformula} in RKHSs can be exploited for analysing these approaches.

%% file: Section_03_SetupAndNotation.tex

\section{Preliminaries and Notation}
\label{section:GeneralSetupAndNotation}


This paper will make use of various standing assumptions and items of notation, which we collect here for easy reference.

Throughout, $(\cF, \innerprod{ \quark }{ \quark }_{\cF}), (\cG, \innerprod{ \quark }{ \quark }_{\cG})$, and $(\cH, \innerprod{ \quark }{ \quark }_{\cH})$ will be separable real Hilbert spaces and $U\in L^2(\Omega, \Sigma, \bP;\cG)$, $V\in L^2(\Omega, \Sigma, \bP;\cH)$, and $W\in L^2(\Omega, \Sigma, \bP;\cF)$, where $(\Omega, \Sigma, \bP)$ is a fixed probability space.

$\cG$-valued expected values $\mu_U \defeq \bE[U] \defeq \int_{\Omega} U(\omega) \, \rd \bP (\omega)$ are always meant in the sense of a Bochner integral \citep[Chapter~II]{diestel1977}, as are the cross-covariance operators
\begin{equation}
\label{equ:DefinitionCovarianceOperators}
C_{UV}
\defeq
\Cov[U, V]
\defeq
\bE[(U - \bE[U]) \otimes (V - \bE[V])]
=
\bE[ U \otimes V ] - \bE[U] \otimes \bE[V]
\end{equation}
from $\cH$ into $\cG$, where, for $h \in \cH$ and $g \in \cG$, the outer product $g \otimes h \colon \cH \to \cG$ is the rank-one linear operator $(g \otimes h) (h') \defeq \innerprod{ h }{ h' }_{\cH}\, g$.
Naturally, we write $C_{U} = \Cov[U]$ for the covariance operator $\Cov[U, U]$, which is self-adjoint, non-negative, and trace-class \citep{baker1973joint, sazonov1958}, and all of the above reduces to the usual definitions in the scalar-valued case.
Using \Cref{theorem:BakerDecompositionCrosscovarianceOperator} and \citet[Lemmas~16.7 and 16.21]{meise1997introduction}, it follows that the cross-covariance operators $C_{UV}$ and $C_{VU} = C_{UV}^{\ast}$ are also trace-class (and in particular Hilbert--Schmidt) operators.
In \Cref{section:ApplicationToCMEsAlternative} we will briefly consider \emph{uncentred} \mbox{(cross-)}covariance operators
\begin{equation}
	\label{equ:DefinitionUncentredCovarianceOperators}
	\uu{C}_{UV}
	\defeq
	\uu{\Cov}[U, V]
	\defeq
	\bE[U\otimes V],
	\qquad
	\uu{C}_{U}
	\defeq
	\uu{\Cov}[U]
	\defeq
	\uu{\Cov}[U,U].
\end{equation}

The orthogonal projection onto a closed linear subspace $F$ of a Hilbert space $\cH$ will be denoted by $P_{F}^{\cH}$, or just $P_{F}$ whenever $\cH$ is clear from context.
Further, we abbreviate
$L^2(\Omega, \Sigma, \bP;\cG)$ by $L^2(\bP;\cG)$ and further by $L^2(\bP)$ if $\cG = \bR$; and
$L^2(\Omega,\tilde{\Sigma},\bP|_{\tilde{\Sigma}};\cG)$ by $L^2(\Omega,\tilde{\Sigma};\cG)$ for any sub-$\sigma$-algebra $\tilde{\Sigma}\subseteq \Sigma$.
$\bP_{X}$ denotes the distribution of a random variable $X\colon \Omega\to\cX$, i.e.\ the pushforward $X_{\#}\bP$ of $\bP$ under $X$.

For a linear operator $A \colon \cH \to \cG$ between Hilbert spaces $\cH$ and $\cG$, its Moore--Penrose pseudo-inverse $A^{\dagger} \colon \dom A^{\dagger} \to \cH$ is the unique extension of
\[
	A |_{(\ker A)^{\perp}}^{-1} \colon \ran A \to (\ker A)^{\perp}
\]
to a linear operator $A^{\dagger}$ defined on $\dom A^{\dagger} \defeq (\ran A) \oplus (\ran A)^{\perp} \subseteq \cG$ subject to the criterion that $\ker A^{\dagger} = (\ran A)^{\perp}$.
In general, $\dom A^{\dagger}$ is a dense but proper subpace of $\cG$ and $A^{\dagger}$ is an unbounded operator;
global definition and boundedness of $A^{\dagger}$ occur precisely when $\ran A$ is closed in $\cG$ \citep[Section~2.1]{engl1996regularization}.

The following spaces of linear and affine operators from $\cH$ to $\cG$ will play a fundamental r\^ole in the approximation of $\gamma_{U|V}$.

\begin{definition}
	\label{def:AllRelevantSpaces}
	Let $\cH$, $\cG$, and $V$ be as above.
	We define the following spaces of linear and affine operators from $\cH$ to $\cG$:
	\begin{align*}
		\BLin(\cH;\cG)
		&\defeq
		\{\gamma \colon \cH \to \cG \mid \gamma \text{ is a bounded linear operator} \},
		\\
		\ABLin(\cH;\cG)
		&\defeq
		\{\gamma \colon \cH \to \cG \mid \gamma(h) = b+Ah \text{ for some } b\in\cG, A \in \BLin(\cH;\cG) \},
		\\[1ex]
		\Lin_V(\cH;\cG)
		&\defeq
		\{\gamma \colon \cH \to \cG \mid \gamma \text{ is linear and } \gamma\circ V \in L^2(\bP;\cG) \},
		\\
		\ALin_V(\cH;\cG)
		&\defeq
		\{\gamma \colon \cH \to\cG \mid \gamma(h) = b+Ah \text{ for some } b\in\cG, A \in \Lin_V(\cH;\cG) \},
		\\[1ex]
		\HS(\cH;\cG)
		&\defeq
		\{\gamma \colon \cH \to \cG \mid \gamma \text{ is a Hilbert--Schmidt operator} \},
		\\
		\AHS(\cH;\cG)
		&\defeq
		\{\gamma \colon \cH \to \cG \mid \gamma(h) = b+Ah \text{ for some } b\in\cG, A \in \HS(\cH;\cG) \}.
		\intertext{Note well that elements of $\Lin_V (\cH;\cG)$ and $\ALin_V (\cH;\cG)$ may be unbounded operators, although their unboundedness is in some sense restricted by the square-integrability requirement.
		For any collection $\mathsf{\Gamma}$ of affine or linear operators $\gamma\colon \cH \to \cG$ we set $\mathsf{\Gamma}\circ V \defeq \{ \gamma\circ V \mid \gamma\in\mathsf{\Gamma} \}$ and}
		\overline{\BLin_{\cG} \circ V}
		&\defeq
		\overline{\BLin(\cH;\cG)\circ V}^{ L^2(\Omega, \Sigma, \bP;\cG)},
		\qquad
		\overline{\ABLin_{\cG} \circ V}
		\defeq
		\overline{\ABLin(\cH;\cG)\circ V}^{ L^2(\Omega, \Sigma, \bP;\cG)}.
	\end{align*}
\end{definition}

Here and henceforth, overlines and superscripts denote topological closures.
The operator norm will be denoted by $\norm{\quark}$.
For any affine operator $\gamma\in\ALin_V(\cH;\cG)$, $\gamma(h) = b+Ah,\, b\in\cG,\, A \in \Lin_V(\cH;\cG)$, we define the ``non-affine part'' by $\overline{\gamma}\defeq A$.
The Hilbert--Schmidt inner product will be denoted by $\innerprod{\gamma_1}{\gamma_2}_{\HS} \defeq \trace(\gamma_1^{\ast}\gamma_2) = \trace(\gamma_1\gamma_2^{\ast})$, where $\gamma_1,\gamma_2\in \HS(\cH;\cG)$, and the corresponding norm by $\norm{\quark}_{\HS}$.
Further, for $\gamma,\gamma'\in \AHS(\cH;\cG)$,
we define the seminorm $\norm{\gamma}_{\AHS}\defeq \norm{\overline{\gamma}}_{\HS}$ and the semi-inner product $\innerprod{\gamma}{\gamma'}_{\AHS} \defeq \innerprod{\overline{\gamma}}{\overline{\gamma}'}_{\HS}$.

\begin{proposition}
	\label{prop:CharacterizeClosureOfBcircV}
	With the notation above,
	\[
		\overline{\BLin_{\cG} \circ V}
		\subseteq
		\Lin_V(\cH;\cG) \circ V,
		\qquad
		\overline{\ABLin_{\cG} \circ V}
		\subseteq
		\ALin_V(\cH;\cG) \circ V.
	\]
\end{proposition}

\proofinappendix

%% file: Section_04_LinearConditionalMean.tex

\section{Linear Conditional Expectation and Covariance}
\label{section:LinearConditionalMean}

It is well known that the conditional expectation $\bE[U|V]$ is the orthogonal projection of $U$ onto $L^2(\Omega, \sigma(V), \bP; \cG)$;
see \Cref{footnote:best_approximation}.
Since $\bE[U|V]$ is $\sigma(V)$-measurable, the Doob--Dynkin lemma \citep[Lemma 1.13]{kallenberg2006foundations} implies the existence of a Borel-measurable function $\gamma_{U|V}\colon \cH\to\cG$ such that $\bE[U|V] = \gamma_{U|V}\circ V$ a.s.
In particular, $\gamma_{U|V}$ minimizes the functional
\begin{equation}
	\label{equ:FunctionalConditionalExpectation}
	\cE_{U|V}(\gamma)
	\defeq
	\norm{U - \gamma\circ V}_{ L^2(\Omega,\Sigma,\bP;\cG)}^{2}
	=
	\bE \bigl[ \norm{U - \gamma\circ V}_{\cG}^{2} \bigr]
\end{equation}
within the class of Borel-measurable functions $\gamma\colon\cH\to \cG$.
Since $\bE[U|V]$ is unique (as an orthogonal projection), $\gamma_{U|V}$ is unique $\bP_{V}$-a.e.\ and we set $\bE[U|V=v] \defeq \gamma_{U|V}(v)$ for $v\in\cH$.

It seems natural to define the best linear approximation (see \Cref{footnote:TerminologyClashLinearAffine}) of the conditional expectation as $\bE^{\ABLin}[U|V] = \gamma_{U|V}^{\ABLin}\circ V$, where $\gamma_{U|V}^{\ABLin}$ minimizes $\cE_{U|V}(\gamma)$ within the class $\ABLin(\cH;\cG)$ of bounded affine operators,
in other words, as the $L^2(\bP;\cG)$-orthogonal projection of $U$ onto $\ABLin(\cH;\cG)\circ V$.
Since this space is not closed in $L^2(\bP;\cG)$, the proper definition uses the projection onto its closure.
In line with the definition of the conditional covariance operator
\begin{equation}
	\label{equ:ClassicalConditionalCovarianceOperator}
	\Cov[U,W|V]
	\defeq
	\bE\bigl[ R[U|V] \otimes R[W|V] \, \big| \, V\bigr],
	\qquad
	R[U|V]
	\defeq
	U - \bE[U|V],
\end{equation}
we further define the linear conditional covariance operator $\Cov^{\ABLin}[U,W|V]$ as follows.

\begin{definition}
	With the notation of \Cref{section:GeneralSetupAndNotation}, define the \emph{linear conditional expectation} (LCE) $\bE^{\ABLin}[U|V]$ (also called \emph{adjusted expectation}, \citealp[Section 3.1]{GoldsteinWooff2007}), the \emph{linear conditional residual} $R^{\ABLin}[U|V]$, and the \emph{linear conditional covariance operator} (LCC)
	$\Cov^{\ABLin}[U,W|V]$ of $U$ given $V$ by
	\begin{align*}
		\bE^{\ABLin}[U|V]
		&\defeq
		P_{\overline{\ABLin_{\cG} \circ V}}U,
		\\
		R^{\ABLin}[U|V]
		& \defeq
		U - \bE^{\ABLin}[U|V],
		\\
		\Cov^{\ABLin}[U,W|V]
		& \defeq
		\bE^{\ABLin} \bigl[ R^{\ABLin}[U|V] \otimes R^{\ABLin}[W|V] \, \big| \, V\bigr].
		\intertext{Further, define the \emph{average linear conditional covariance operator} (ALCC) by}
		\Cov^{\ABLin}_{V}[U,W]
		& \defeq
		\bE\bigl[ R^{\ABLin}[U|V] \otimes R^{\ABLin}[W|V] \bigr].
	\end{align*}
	By \Cref{prop:CharacterizeClosureOfBcircV}, $\bE^{\ABLin}[U|V]$ will be of the form $\gamma_{U|V}^{\ABLin}\circ V$, where $\gamma_{U|V}^{\ABLin} \in \ALin_V(\cH;\cG)$ is unique $\bP_{V}$-a.e.\ and will be referred to as the \emph{linear conditional expectation function} (LCEF).
	As usual, we set $\Cov [U|V] \defeq \Cov [U,U|V]$, $\Cov^{\ABLin}[U|V] \defeq \Cov^{\ABLin}[U,U|V]$, and $\Cov_{V}^{\ABLin}[U] \defeq \Cov_{V}^{\ABLin}[U,U]$.
\end{definition}

Since $\bE^{\ABLin}[U|V]$ and $\Cov^{\ABLin}[U|V]$ are defined as $L^2(\bP;\cG)$-orthogonal projection, all statements and identities in the following subsections only hold $\bP$-a.s.

\begin{remark}
	\label{remark:ComparisonOfConditionalCovarianceOperatorDefinitions}
	\citet[Section 3.3]{GoldsteinWooff2007} call $\Cov^{\ABLin}_{V}[U,W]$ the \emph{adjusted covariance} and argue that this is the proper way to define the linear analogue of the conditional covariance.
	However, this definition is not in line with the classical conditional covariance \eqref{equ:ClassicalConditionalCovarianceOperator} because it fails to condition on $V$ a second time;
	see \Cref{example:AverageConditionalCovariance} below.
	Note that the ALCC $\Cov^{\ABLin}_{V}[U,W]$ is therefore not a random variable but rather the expected value of the LCC $\Cov^{\ABLin}[U,W|V]$, hence our term ``average linear conditional covariance'';
	see \Cref{thm:PropertiesLinearConditionalCovariance}, where we also show that it coincides with the well-known Gaussian conditional covariance formula,
	$\Cov^{\ABLin}_{V}[U,W] = C_{UW} - C_{UV} C_{V}^{\dagger} C_{VW}$ (and similarly its more general version in the incompatible case).
	While $\Cov^{\ABLin}_{V}[U,U]$ is always non-negative (see \Cref{thm:PropertiesLinearConditionalCovariance}\ref{item:UpperBoundExpectedConditionalCovariance}), $\Cov^{\ABLin}[U,W|V]$ can take on negative values (see \Cref{counterexample:CounterexamplesPropertiesLCM}\ref{item:NonNegativityLCC}).
\end{remark}

\begin{example}
	\label{example:AverageConditionalCovariance}
	Consider the following simple example of an LCE and an LCC.
	Let $\cH = \cG = \bR$, let $\bP$ be the uniform distribution on $\Omega = \{ 1,2,3,4 \}$, and let $V$ and $U$ be as defined below and illustrated in \Cref{fig:VisualizationConditionalCovariance}.
	\begin{center}
		\begin{tabular}{x{3em} x{3em} x{3em} x{7em} x{7em}}
			\toprule
			$\omega\in\Omega$ & $V(\omega)$ & $U(\omega)$ & $\bE^{\ABLin}[U|V](\omega)$ & $\Cov^{\ABLin}[U|V](\omega)$
			\\
			\midrule
			$1$ & $0$ & $\phantom{-}1$ & $0$ & 1 \\
			$2$ & $0$ & $-1$ & $0$ & 1 \\
			$3$ & $1$ & $\phantom{-}2$ & $0$ & 4 \\
			$4$ & $1$ & $-2$ & $0$ & 4 \\
			\bottomrule
		\end{tabular}
	\end{center}
	By symmetry, $\bE^{\ABLin}[U|V] = \bE[U|V] = 0$ and, since $(V,R^{\ABLin}[U|V]^{2})$ takes on only two values, $\Cov^{\ABLin}[U|V]$ coincides with the classical conditional covariance $\Cov[U|V]$.
	The ALCC $\Cov^{\ABLin}_{V}[U] = \nicefrac{5}{2}$ only captures its expected value.
	\begin{figure}[t]
		\centering
		\includegraphics[width=0.45\textwidth]{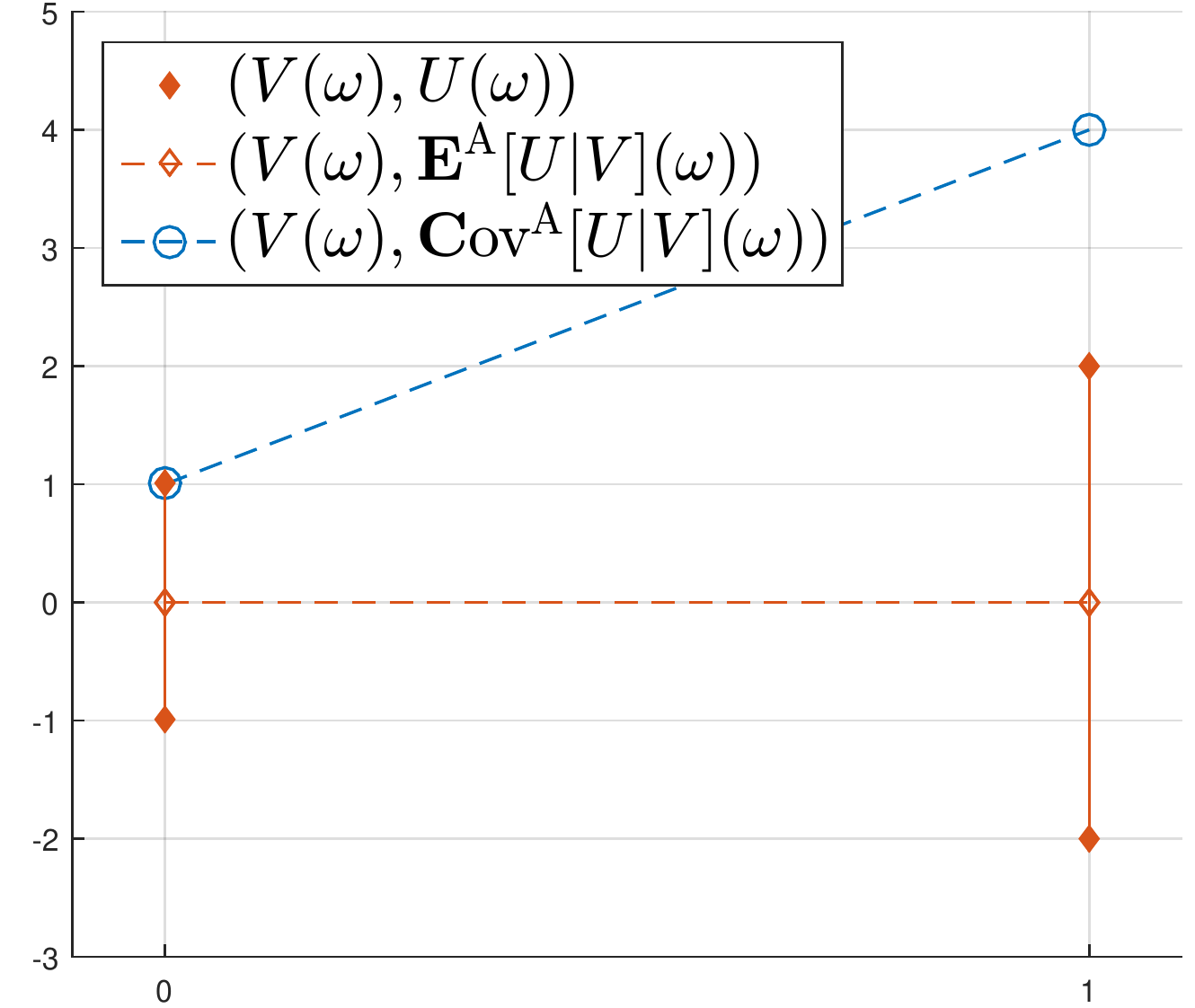}
		\caption{In \Cref{example:AverageConditionalCovariance}, the conditional expectation $\bE[U|V]$ as well as the conditional covariance $\Cov[U|V]$ happen to be affine, which is why they coincide with the LCE $\bE^{\ABLin}[U|V]$ and the LCC $\Cov^{\ABLin}[U|V]$, respectively.
		The ALCC $\Cov^{\ABLin}_{V}[U] = \nicefrac{5}{2}$ captures the expected value of $\Cov^{\ABLin}[U|V]$.}
		\label{fig:VisualizationConditionalCovariance}
	\end{figure}
\end{example}

The main aim of this section is to investigate basic properties of, and provide explicit formulae for, the LCE and the LCC.

\subsection{Basic Properties of the LCE}
\label{section:BasicPropertiesLCE}

This section highlights, in \Cref{theorem:BasicPropertiesLCM,counterexample:CounterexamplesPropertiesLCM} respectively, the ways in which the LCE shares and lacks the key properties of the exact conditional expectation.
We call attention to the non-trivial conditions that appear to be necessary for the LCE version of the dominated convergence theorem;
see \Cref{theorem:BasicPropertiesLCM}\ref{item:DominatedConvergenceLCMmodified}, 	\Cref{remark:DCT}, and \Cref{counterexample:CounterexamplesPropertiesLCM}\ref{item:DominatedConvergenceLCMconventional}.
As mentioned above, all statements concerning $\bE^{\ABLin}[U|V]$ and $\Cov^{\ABLin}[U|V]$ only hold $\bP$-a.s.


\begin{lemma}
	\label{lemma:CovarianceFormulaForLCM}
	With the notation of \Cref{section:GeneralSetupAndNotation}, let $W\in \overline{\ABLin_{\cF} \circ V}$.
	Then, a.s.,
	\[
		\bE\bigl[ R^{\ABLin}[U|V] \bigr] = 0,
		\qquad
		\Cov\bigl[ R^{\ABLin}[U|V] , W \bigr] = 0.
	\]
	In particular,
	\[
		\Cov\bigl[ \bE^{\ABLin}[U|V] , V \bigr] = \Cov[U,V] \text{ a.s.}
	\]
\end{lemma}


\proofinappendix


By way of comparison with the exact conditional expectation, some basic properties of the LCE are summarised by the following theorem (a martingale property together with a martingale convergence theorem are postponed to \Cref{thm:MartingaleProperty}).


\begin{theorem}[Basic properties satisfied by the LCE and the LCC]
	\label{theorem:BasicPropertiesLCM}
	With the notation of \Cref{section:GeneralSetupAndNotation}, let $U'$ and $U_{k}\in L^2(\Omega,\Sigma,\bP;\cG)$ for $k\in\bN$.
	Further, let $\varphi\in \ABLin(\cH;\cF)$.
	The LCE fulfils the following basic properties $\bP$-a.s.:
	\begin{enumerate}[label = (\alph*)]
		\item
		\label{item:StabilityLCM}
		\textbf{stability:}
		\\[1ex]
		$\bE^{\ABLin}[U|V]
		=
		\bE[U|V]$,
		if $\bE[U|V] \in \overline{\ABLin_{\cG}\circ V}$, in particular,
		\\[1ex]
		$\bE^{\ABLin}[U|V] = g$ a.s.\ whenever $U=g\in\cG$ a.s.,
		$\bE^{\ABLin}[V|V] = V$ and
		$\bE^{\ABLin}[\varphi \circ V|V] = \varphi \circ V$;

		\item
		\label{item:LinearityLCM}
		\textbf{linearity:}
		\\[1ex]
		$\bE^{\ABLin}[aU+bU'|V]
		=
		a\bE^{\ABLin}[U|V] + b\bE^{\ABLin}[U'|V]$ for any $a,b\in\bR$ and
		\\[1ex]
		$\bE^{\ABLin}[\psi(U)|V] =\psi\bigl( \bE^{\ABLin}[U|V] \bigr)$ for any $\psi\in \ABLin(\cG;\cF)$;

		\item
		\label{item:SelfAdjointnessLCM}
		\textbf{self-adjointness:}
		\\[1ex]
		$\bE\bigl[ \innerprod{U'}{\bE^{\ABLin}[U|V]}_{\cG} \bigr]
		=
		\bE\bigl[ \innerprod{\bE^{\ABLin}[U'|V]}{\bE^{\ABLin}[U|V]}_{\cG} \bigr]
		=
		\bE\bigl[ \innerprod{\bE^{\ABLin}[U'|V]}{U}_{\cG} \bigr]$;

		\item
		\label{item:LawOfTotalExpectationLCM}
		\textbf{law of total linear expectation:}
		\\[1ex]
		$\bE\bigl[ \bE^{\ABLin}[U|V] \bigr]
		=
		\bE[U]$;

		\item
		\label{item:ConditionalExpectationLCM}
		\textbf{compatibility with conditional expectation:}
		\\[1ex]
		$\bE\bigl[ \bE^{\ABLin}[U|V] \, \big|\, W\bigr]
		=
		\bE^{\ABLin}\bigl[ \bE[U|W] \, \big|\, V\bigr]$;
		\\[1ex]
		$\bE\bigl[ \bE^{\ABLin}[U|V] \, \big|\, V\bigr]
		=
		\bE^{\ABLin}\bigl[ \bE[U|V] \, \big|\, V\bigr]
		=
		\bE^{\ABLin}[U|V]$;

		\item
		\label{item:TowerPropertyLCM}
		\textbf{tower properties:}
		\\[1ex]
		$\bE^{\ABLin}\bigl[ \bE[U|V] \, \big|\, W\bigr]
		=
		\bE^{\ABLin}[U|W]$ if $\sigma(W)\subseteq \sigma(V)$;
		\\[1ex]
		$\bE^{\ABLin}\bigl[ U \, \big|\, \bE^{\ABLin}[U|V] \bigr]
		=
		\bE^{\ABLin}[U|V]$;
		\\[1ex]
		$\bE^{\ABLin}\bigl[ \bE^{\ABLin}[U|V] \, \big|\, \varphi \circ V\bigr]
		=
		\bE^{\ABLin}\bigl[ \bE[U|V] \, \big|\, \varphi \circ V\bigr]
		=
		\bE^{\ABLin}[U| \varphi \circ V]$, in particular,
		\\[1ex]
		$\bE^{\ABLin}\bigl[ \bE^{\ABLin}[U| (V,W)] \, \big|\, V\bigr]
		=
		\bE^{\ABLin}\bigl[ \bE[U|(V,W)] \, \big|\, V\bigr]
		=
		\bE^{\ABLin}[U|V]$;

		\item
		\label{item:LawOfTotalCovarianceLCM}
		\textbf{law of total linear covariance:}
		\\[1ex]
		$\Cov[U,W]
		=
		\Cov\bigl[ \bE^{\ABLin}[U|V] , \bE^{\ABLin}[W|V] \bigr]
		+
		\bE\bigl[ \Cov^{\ABLin}[U,W|V] \bigr]$, in particular,
		\\[1ex]
		$\Cov[U] \geq \Cov\bigl[\bE[U|V]\bigr] \geq \Cov\bigl[\bE^{\ABLin}[U|V]\bigr] \ge 0$.

		\item
		\label{item:PullingOutIndependentLCM}
		\textbf{pulling out independent factors:}
		\\[1ex]
		$\bE^{\ABLin}[W\otimes U|V]
		=
		\bE[W]\otimes \bE^{\ABLin}[U|V]$,
		if $W$ is independent of $(U,V)$, in particular,
		\\[1ex]
		$\bE^{\ABLin}[W|V]
		=
		\bE[W]$,
		if $V$ and $W$ are independent (this also follows from 		\ref{item:StabilityLCM});

		\item
		\label{item:DominatedConvergenceLCMmodified}
		\textbf{$L^2$-dominated convergence theorem (DCT):}
		\\[1ex]
		$\norm{ \bE^{\ABLin}[U_{k}|V] - \bE^{\ABLin}[U|V]}_{\cG}\xrightarrow[k\to\infty]{\text{a.s.}}0$ if $C_{V}$ has finite rank and either of the following conditions holds:
		\begin{enumerate}[label=(\greek*)]
			\item \label{item:DCTwithSquareIntegrableDominatingFunction} $\norm{U_{k} - U}_{\cG}\xrightarrow[k\to\infty]{\text{a.s.}}0$ and $\norm{U_{k}}_{\cG} \le Y$ for all $k\in\bN$ and some $Y\in L^{2}(\bP)$,
			\item \label{item:DCTunderL2Convergence} $\norm{U_{k} - U}_{L^{2}(\bP;\cG)}\xrightarrow[k\to\infty]{}0$.
		\end{enumerate}
	\end{enumerate}
\end{theorem}


\proofinappendix


\begin{remark}[Sufficient conditions for the DCT]
	\label{remark:DCT}
	Note that in \Cref{theorem:BasicPropertiesLCM}\ref{item:DominatedConvergenceLCMmodified}\ref{item:DCTwithSquareIntegrableDominatingFunction} the dominating random variable $Y\in L^{2}(\bP)$ is assumed to be square-integrable, which is slightly stronger than the conventional assumption $Y\in L^{1}(\bP)$ (a counterexample to the sufficiency of the latter assumption is provided in \Cref{counterexample:CounterexamplesPropertiesLCM}\ref{item:DominatedConvergenceLCMconventional}).
	On the other hand, \ref{item:DCTunderL2Convergence} is a particularly weak assumption, too weak for analogous statements on the (regular) conditional expectation $\bE[\quark | V]$ in place of $\bE^{\ABLin}[\quark | V]$.

	So far, we could only prove the DCT under the condition that $C_{V}$ has finite rank.
	Note that counterexamples can easily be constructed (see below) if one only assumes \ref{item:DCTunderL2Convergence}.
	However, the validity of the DCT under \ref{item:DCTwithSquareIntegrableDominatingFunction} for $C_{V}$ of infinite rank remains an open problem.
	One obstacle here is the missing monotonicity of the LCE, see \Cref{counterexample:CounterexamplesPropertiesLCM}\ref{item:MonotonicityLCM}: $\norm{U_{k}}_{\cG} \le Y$ a.s.\ does not imply that $\norm{\bE^{\ABLin}[U_{k}|V]}_{\cG} \le Y$ a.s.

	For a counterexample to the DCT under assumption \ref{item:DCTunderL2Convergence} (without the finite-rank assumption) consider a centered Gaussian random variable $V$ on $\cH = \ell^{2}$ with Karhunen--Lo\`eve expansion $V = \sum_{n} \sigma_{n}Z_{n}e_{n}$, where $\sigma_{n} > 0$ for all $n\in\bN$, $\sum_{n} \sigma_{n}^{2} < \infty$, $Z_{n}\stackrel{\text{i.i.d.}}{\sim}\cN(0,1)$ and $(e_{n})_{n\in\bN}$ is the canonical basis of $\cH = \ell^{2}$.
	Choose $\cG = \bR$ and $U_{k} = \delta_{k} Z_{k}$ where $\delta_{k} \searrow 0$ such that $\bP[A_{k}] = 1/k$ for $A_{k} = [\delta_{k}Z_{k} \geq \varepsilon]$ and $\varepsilon = 1$.
	Then $\bE^{\ABLin} [U_{k}|V]=U_{k}$ by definition of the LCE and, since the family of events $(A_{k})_{k\in \bN}$ is independent and $\sum_{k} \bP[A_{k}] = \infty$, the second Borel--Cantelli lemma implies $\bP[A_{k} \text{ i.o.}] = 1$.
	Hence, by \citet[Theorem~4.1.6]{bremaud2017discrete}, $\bE^{\ABLin} [U_{k}|V]=U_{k}$ does not converge to zero a.s., while \ref{item:DCTunderL2Convergence} is satisfied since $\delta_{k} \to 0$.
\end{remark}

It is also worth mentioning that the LCE does \emph{not} fulfil several important properties of conditional expectations.
These are summarised by the following statements and the actual counterexamples are provided in the proof in \Cref{section:Proofs}.
Note in particular that there are scalar-valued counterexamples in each case, and so these deficiencies of the LCE are not merely a consequence of the Hilbert space context.

\begin{theorem}[Basic properties \textcolor{red}{\underline{\textbf{not}}} satisfied by the LCE and the LCC]
	\label{counterexample:CounterexamplesPropertiesLCM}
	For each of the following desired properties of the LCE, there is an explicit counterexample using $\bR$-valued random variables $U$, $U'$, $V$, etc.\ satisfying the assumptions of \Cref{theorem:BasicPropertiesLCM}. As usual, all statements have to be understood $\bP$-a.s.
	\begin{enumerate}[label = (\alph*)]
		\item
		\label{item:MonotonicityLCM}
		\textbf{monotonicity:} \textsuperscript{\textcolor{red}{\textbf{(invalid)}}}
		\\[1ex]
		$U \geq U'$ implies $\bE^{\ABLin}[U|V] \geq \bE^{\ABLin}[U'|V]$ in the case $\cG = \bR$;

		\item
		\label{item:TriangleInequalityLCM}
		\textbf{triangle inequality:} \textsuperscript{\textcolor{red}{\textbf{(invalid)}}}
		\\[1ex]
		$\bignorm{ \bE^{\ABLin}[U|V] }_{\cG}
		\leq
		\bE^{\ABLin}[\norm{U}_{\cG}|V]$;

		\item
		\label{item:JensensInequalityLCM}
		\textbf{Jensen's inequality:} \textsuperscript{\textcolor{red}{\textbf{(invalid)}}}
		\\[1ex]
		$f\bigl( \bE^{\ABLin}[U|V] \bigr)
		\leq
		\bE^{\ABLin}[f(U)|V]$ for any convex\footnote{Note that Jensen's (in-)equality holds for affine functions $f\in \ABLin(\cG;\cF)$, cf.\ \Cref{theorem:BasicPropertiesLCM}\ref{item:LinearityLCM}.} function $f\colon \cG\to\bR$;

		\item
		\label{item:PullingOutKnownLCM}
		\textbf{pulling out known factors:} \textsuperscript{\textcolor{red}{\textbf{(invalid)}}}
		\\[1ex]
		$\bE^{\ABLin}[f(V) \, U|V]
		=
		f(V) \, \bE^{\ABLin}[U|V]$
		for measurable\footnote{Our counterexample shows that property \ref{item:PullingOutKnownLCM} is invalid even if ``measurable'' is strengthened to ``linear''.} maps $f\colon \cH\to\bR$;

		\item
		\label{item:WrongTowerPropertyLCM}
		\textbf{(yet another) tower property:} \textsuperscript{\textcolor{red}{\textbf{(invalid)}}}
		\\[1ex]
		$\bE^{\ABLin}\bigl[ \bE^{\ABLin}[U|V] \, \big|\, W\bigr]
		=
		\bE^{\ABLin}[U|W]$ if $\sigma(W)\subseteq \sigma(V)$;		

		\item
		\label{item:FatouLCM}
		\textbf{Fatou's lemma:} \textsuperscript{\textcolor{red}{\textbf{(invalid)}}}
		\\[1ex]
		Let $\cG = \bR$ and $\bE^{\ABLin}[\inf_{k \in \bN} U_k | V]< \infty$ (alternatively, $\bE[\inf_{k \in \bN} U_k | V]< \infty$). Then
		\\[1ex]
		$\bE^{\ABLin}\bigl[\liminf_{k\to\infty} U_k \big| V \bigr]
		\leq
		\liminf_{k\to\infty} \bE^{\ABLin}[U_k | V]$;

		\item
		\label{item:DominatedConvergenceLCMconventional}
		\textbf{$L^1$-dominated convergence theorem:} \textsuperscript{\textcolor{red}{\textbf{(invalid)}}}
		\\[1ex]
		If $\norm{U_{k} - U}_{\cG}\xrightarrow[k\to\infty]{\text{a.s.}}0$ and $\norm{U_{k}}_{\cG} \le Y$ for all $k\in\bN$ and some $Y\in L^{1}(\bP)$, then
		\newline
		$\norm{ \bE^{\ABLin}[U_{k}|V] - \bE^{\ABLin}[U|V]}_{\cG}\xrightarrow[k\to\infty]{\text{a.s.}}0$.
		
		\item
		\label{item:ContractivenessLCM}
		\textbf{$L^p$-contractivity for $p \neq 2$:} \textsuperscript{\textcolor{red}{\textbf{(invalid)}}} 
		\\[1ex]
		$\bE^{\ABLin}[\quark | V]$ is a contractive projection of $L^{p}(\bP;\cG)$ spaces for some $1\leq p \neq 2$, i.e.
		\\[1ex]
		$\bE \bigl[ \norm{\bE^{\ABLin}[U|V]}_{\cG}^{p} \bigr]
		\leq
		\bE \bigl[ \norm{U}_{\cG}^{p} \bigr]$ for any $U\in L^{p}(\Omega,\Sigma,\bP ; \cG)$.%

		\item
		\label{item:NonNegativityLCC}
		\textbf{non-negativity of the LCC:} \textsuperscript{\textcolor{red}{\textbf{(invalid)}}}
		\\[1ex]
		The LCC $\Cov^{\ABLin}[U|V]$ is non-negative, $\Cov^{\ABLin}[U|V] \geq 0$.
	\end{enumerate}
\end{theorem}

Note that, in contrast to \Cref{counterexample:CounterexamplesPropertiesLCM}\ref{item:ContractivenessLCM}, $\bE^{\ABLin}[\quark | V]$ \emph{is} a contractive projection on $L^{2}(\bP;\cG)$;
this follows directly from the definition of the LCE as an $L^{2}(\bP;\cG)$-orthogonal projection.

\proofinappendix

\subsection{Explicit Formula for the LCE: Compatible Case}

We are first going assume that $\ran C_{VU}\subseteq \ran C_{V}$, which, following \citet{corach2000oblique} and \citet{owhadi2015conditioning}, we call the \emph{compatible case}.
In this case, the orthogonal projection $\bE^{\ABLin}[U|V]$ of $U$ onto $\overline{\ABLin_{\cG}\circ V}$ turns out to lie in $\ABLin(\cH;\cG)\circ V$, which is generally not closed in $L^2(\Omega,\Sigma,\bP;\cG)$.
The following theorem provides an explicit formula for the (affine) conditional mean and generalises \citet[Lemma~4.1]{ernst2015analysis}.


\begin{theorem}[Formula for the LCE: compatible case]
	\label{theorem:LinearConditionalMeanUnderRangeAssumption}
	With the notation of \Cref{section:GeneralSetupAndNotation}, if the range inclusion $\ran C_{VU}\subseteq \ran C_{V}$ holds, then the operator $C_V^{\dagger} C_{VU}\colon \cG\to\cH$ is bounded and the operator $\gamma_{U|V}^{\ABLin} \in \ABLin(\cH;\cG)$ defined by
	\begin{align*}
		\gamma_{U|V}^{\ABLin}(v)
		&\defeq
		\mu_U + (C_{V}^\dagger C_{VU})^{\ast} (v-\mu_V)
	\end{align*}
	minimizes the functional $\cE_{U|V}$ given by \eqref{equ:FunctionalConditionalExpectation}
	within $\ABLin(\cH;\cG)$.
	In particular, $\bE^{\ABLin}[U|V] =	\gamma_{U|V}^{\ABLin}\circ V$ a.s., i.e.\ $\gamma_{U|V}^{\ABLin}$ is an LCEF.
\end{theorem}


\proofinappendix


\Cref{theorem:LinearConditionalMeanUnderRangeAssumption} is a genuine generalisation of the case $\dim\cH < \infty$ for the following reason, which is a direct consequence of \Cref{theorem:BakerDecompositionCrosscovarianceOperator}:


\begin{corollary}
	\label{corollary:RangeInclusionIfClosedRange}
	With the notation of \Cref{section:GeneralSetupAndNotation}, the condition $\ran C_{VU}\subseteq \ran C_{V}$ is always fulfilled whenever $C_{V}$ has closed range, and, in particular, if $\cH$ is finite dimensional.
\end{corollary}


\subsection{Explicit Formula for the LCE: Incompatible Case}

We are now going to treat the general case, in which the orthogonal projection $\bE^{\ABLin}[U|V]$ of $U$ can not be expected to lie in $\ABLin(\cH;\cG)\circ V$.
We are therefore going to approximate $\bE^{\ABLin}[U|V]$ by a sequence of bounded (in fact, even finite-rank) operators $\gamma_{U|V}^{(n)} \in \ABLin(\cH;\cG)$ composed with $V$, where the convergence will hold in the $L^{2}(\bP;\cG)$ norm as well as a.s.
This requires some additional notation.

\begin{notation}
	\label{notation:NotationForIncompatibleLCM}
	Let $\dim \cH = \infty$\footnote{This assumption is not substantial and we make it merely for the sake of simplifying our notation. Note that the finite-dimensional case has been analysed in the previous subsection.}
	and recall the notation of \Cref{section:GeneralSetupAndNotation}.
	Consider the eigendecomposition of the covariance operator $C_V$,
	\[
		C_V
		=
		\sum_{n\in \bN} \sigma_n\, h_n\otimes h_n,
		\qquad
		\sigma_n^{2} \geq 0,
	\]
	where $(h_n)_{n\in\bN}$ is a complete orthonormal system of $\cH$.
	Let $n\in\bN$, $\cH^{(n)} \defeq \spn \{ h_1,\dots,h_n \}$, $V^{(n)} \defeq P_{\cH^{(n)}} V$ and
	\[
		C
		\defeq
		\begin{pmatrix}
			C_{U} & C_{UV}
			\\
			C_{VU} & C_{V}
		\end{pmatrix},
		\qquad
		C^{(n)}
		\defeq
		P_{\cH^{(n)}} C P_{\cH^{(n)}}
		=
		\begin{pmatrix}
			C_{U} & C_{UV}^{(n)}
			\\
			C_{VU}^{(n)} & C_{V}^{(n)}
		\end{pmatrix}.
	\]
	Since $C_{V}^{(n)}$ has finite rank, \Cref{theorem:BakerDecompositionCrosscovarianceOperator} yields $\ran C_{VU}^{(n)}\subseteq \ran C_{V}^{(n)}$ and we can define the operator $\gamma_{U|V}^{(n)} \in \ABLin(\cH ; \cG)$ by
	\[
		\gamma_{U|V}^{(n)}(v)
		\defeq
		\mu_U +
		\bigl( C_{V}^{(n)\dagger} C_{VU}^{(n)}\bigr)^{\ast}(v-\mu_V).
	\]
	Further, by \Cref{theorem:BakerDecompositionCrosscovarianceOperator} and adopting the notation therein, the operator $M_{VU}\defeq (C_{V}^{1/2})^{\dagger} C_{VU} = R_{VU} C_{U}^{1/2}\colon \cG\to \cH$ is well defined and bounded (in fact, it is even Hilbert--Schmidt).
\end{notation}


\begin{lemma}
	\label{lemma:MartingaleLemmaEqualityOfSpaces}
	With the notation of \Cref{section:GeneralSetupAndNotation} and using \Cref{notation:NotationForIncompatibleLCM},
	\begin{equation}
	\label{equ:TechnicalIdentityOfSpaces}
	\overline{\ABLin_{\cG}\circ V} \cap L^{2}(\Omega,\sigma(V^{(n)}) ; \cG) = \overline{\ABLin_{\cG}\circ V^{(n)}}.
	\end{equation}
\end{lemma}


\proofinappendix


\begin{theorem}[Martingale property and martingale convergence theorem]
	\label{thm:MartingaleProperty}
	With the notation of \Cref{section:GeneralSetupAndNotation} and using \Cref{notation:NotationForIncompatibleLCM},
	\begin{enumerate}[label = (\alph*)]
		\item
		\label{item:MartingaleProperty}
		$(\bE^{\ABLin}[U|V^{(n)}])_{n\in\bN}$ is a martingale with respect to the filtration $(\sigma(V^{(n)}))_{n\in\bN}$; more precisely,
		\begin{equation}
			\label{equ:MartingaleProperty}
			\bE \bigl[ \bE^{\ABLin}[U|V] \, \big|\, V^{(n)} \bigr]
			=
			\bE^{\ABLin}[U|V^{(n)}] \text{ a.s.;}	
		\end{equation}
		\item
		\label{item:MartingaleConvergenceTheorem}
		the following martingale convergence theorem holds a.s.\ and in $L^{p}(\bP;\cG)$ for $1\le p \le 2$:
		\begin{equation}
			\label{equ:MartingaleConvergenceTheorem}
			\bE^{\ABLin}[U|V^{(n)}]
			\xrightarrow[n\to\infty]{}
			\bE^{\ABLin}[U|V].
		\end{equation}		
	\end{enumerate}	
\end{theorem}


\proofinappendix


\begin{theorem}[Formula for the LCE: incompatible case]
	\label{theorem:LinearConditionalMeanIncompatibleCase}
	With the notation of \Cref{section:GeneralSetupAndNotation} and using \Cref{notation:NotationForIncompatibleLCM}, the operators $\gamma_{U|V}^{(n)}$ minimize the functional $\cE_{U|V}$ given by \eqref{equ:FunctionalConditionalExpectation} within $\ABLin(\cH ; \cG)$ for $n\to\infty$, i.e.
	\[
		\cE_{U|V}(\gamma_{U|V}^{(n)})
		\ \xrightarrow[n\to\infty]{}\
		\inf_{\gamma \in \ABLin(\cH ; \cG)} \cE_{U|V}(\gamma).
	\]
	In other words,
	\begin{equation}
		\label{equ:L2ConvergenceToLCE}
		\bignorm{\bE^{\ABLin}[U|V] - \gamma_{U|V}^{(n)} \circ V}_{ L^2(\Omega,\Sigma,\bP;\cG)} \xrightarrow[n\to\infty]{} 0.
	\end{equation}
	Further, denoting the pushforward of $\bP$ under $V$ by $\bP_{V}$, $\gamma_{U|V}^{(n)} \circ V$ converges to $\bE^{\ABLin}[U|V]$ a.s.,
	\begin{equation}
		\label{equ:ASconvergenceToLCE}
		\bignorm{ \gamma_{U|V}^{(n)}(v) - \bE^{\ABLin}[U|V=v] }_{\cG}
		\xrightarrow[n\to \infty]{}
		0
		\qquad
		\text{for $\bP_{V}$-a.e.\ } v\in\cH.
	\end{equation}
\end{theorem}


\proofinappendix


\subsection{Explicit Formula for the LCE: Regularised Case}
\label{section:RegularizedLCM}

In most practical applications the means and (cross-)covariance operators of $U$ and $V$ are not accessible explicitly, but have to be approximated empirically from data (in the simplest case, from independent and identically distributed samples $(u_n,v_n)\sim \bP_{UV}$, where $\bP_{UV}$ denotes the joint distribution of $U$ and $V$).
Since the Moore--Penrose pseudo-inverse $C_{V}^{\dagger}$ shows an unstable behaviour when approximated empirically \citep[Section~SM2]{klebanov2019rigorous}, it is typically replaced by its regularised version $(C_{V} + \varepsilon \Id_{\cH})^{-1}$, where $\varepsilon > 0$ is a regularisation parameter.
The following theorem shows that this is a principled way to address this issue, since the resulting operators minimize a perturbed functional $\cE_{U|V}^{\textup{reg}}$.
The natural space of operators in this context turns out to be the space of affine Hilbert--Schmidt operators.

\begin{theorem}[Formula for the LCE: regularised case]
	\label{theorem:LinearConditionalMeanRegularized}
	With the notation of \Cref{section:GeneralSetupAndNotation}, the operator $\gamma_{\varepsilon}^{\AHS}\in \AHS (\cH;\cG)$ defined by
	\begin{align*}
		\gamma_{\varepsilon}^{\AHS}(v)
		&\defeq
		\mu_U + C_{UV} (C_{V} + \varepsilon \Id_{\cH})^{-1}(v - \mu_V)
	\end{align*}
	minimizes the Tikhonov--Philipps-regularised functional
	\[
	\cE_{U|V}^{\textup{reg}}(\gamma)
	=
	\cE_{U|V}(\gamma) +
	\varepsilon \norm{\gamma}_{\AHS}^{2}
	\]
	within $\AHS (\cH;\cG)$, where $\cE_{U|V}$ is given by \eqref{equ:FunctionalConditionalExpectation}.
\end{theorem}

\proofinappendix

\subsection{Explicit Formula for the Linear Conditional Covariance}

Before we derive a formula for the LCC $\Cov^{\ABLin}[U,W|V]$, the following theorem states an explicit formula for and some basic properties of the ALCC $\Cov^{\ABLin}_{V}[U,W]$.
In particular, as mentioned earlier in \Cref{remark:ComparisonOfConditionalCovarianceOperatorDefinitions}, we characterise the ALCC as the expected value $\bE\bigl[\Cov^{\ABLin}[U,W|V]\bigr]$ of the LCC --- hence our terminology for each of these conditional covariances.


\begin{theorem}[Properties of the ALCC]
	\label{thm:PropertiesLinearConditionalCovariance}
	With the notation of \Cref{section:GeneralSetupAndNotation} and using \Cref{notation:NotationForIncompatibleLCM},
	\begin{enumerate}[label = (\alph*)]
		\item
		\label{item:AverageConditionalCovariance}
		$\bE\bigl[\Cov^{\ABLin}[U,W|V]\bigr]
		=
		\Cov^{\ABLin}_{V}[U,W]$;
		\item
		\label{item:UpperBoundExpectedConditionalCovariance}
		$\Cov^{\ABLin}_{V}[U]
		=
		\bE\bigl[\Cov^{\ABLin}[U|V]\bigr]
		\ge
		\bE\bigl[\Cov[U|V]\bigr]
		\ge
		0$;
		\item
		\label{item:GaussianFormulaForConditionalCovariance}
		$\Cov^{\ABLin}_{V}[U,W]
		=
		C_{UW} - M_{VU}^{\ast} M_{VW}$,
		\\[1ex]
		in particular, in the compatible case $\ran C_{VW} \subseteq \ran C_{V}$,
		\\[1ex]
		$\displaystyle
		\Cov^{\ABLin}_{V}[U,W]
		=
		C_{UW} - C_{UV} C_{V}^{\dagger} C_{VW}$.
	\end{enumerate}
\end{theorem}


\proofinappendix


\begin{remark}
The inequality $\Cov^{\ABLin}_{V}[U]\geq 0$ can also be seen more directly using 	\ref{item:GaussianFormulaForConditionalCovariance}:
\[
\Cov^{\ABLin}_{V}[U]
=
C_{U} - M_{VU}^{\ast} M_{VU}
=
C_{U} - (R_{VU} C_{U}^{1/2})^{\ast} R_{VU} C_{U}^{1/2}
=
C_{U}^{1/2} \bigl( \Id_\cG - R_{VU}^{\ast}R_{VU} \bigr) C_{U}^{1/2}
\ge
0,
\]
where we used \Cref{notation:NotationForIncompatibleLCM} and $\norm{R_{VU}}\le 1$ (cf.\ \Cref{theorem:BakerDecompositionCrosscovarianceOperator}).
While its expected value is always non-negative, the LCC $\Cov^{\ABLin}[U|V]$ itself can take on negative values, see \Cref{counterexample:CounterexamplesPropertiesLCM}\ref{item:NonNegativityLCC}.
\end{remark}

\begin{remark}
	Computing the conditional covariance $\Cov[U|V=v]$ for various $v\in\cH$ is often too costly, in which case one can focus on its mean $\bE[\Cov[U|V]]$.
	The above statements show that the average LCC $\bE\bigl[\Cov^{\ABLin}[U|V]\bigr]
	=
	\Cov^{\ABLin}_{V}[U]$,
	which can be computed by the Gaussian conditional covariance formula, never underestimates the true expected conditional covariance $\bE\bigl[\Cov[U|V]\bigr]$.
\end{remark}

As a consequence, we obtain an explicit formula for the LCC $\Cov^{\ABLin}[U,W|V]$:


\begin{corollary}[Formula for the LCC]
	\label{corollary:FormulaForLCC}
	With the notation of \Cref{section:GeneralSetupAndNotation}, let
	\[
		Z \defeq R^{\ABLin}[U|V] \otimes R^{\ABLin}[W|V] \colon \Omega\to \HS(\cF;\cG) .
	\]
	Then, using \Cref{notation:NotationForIncompatibleLCM},
	\[
		\Cov^{\ABLin}[U,W|V]
		=
		C_{UW} - M_{VU}^{\ast} M_{VW} +
		\lim_{n\to\infty} \bigl(C_{V}^{(n)\dagger}C_{VZ}^{(n)}\bigr)^{\ast} (V-\mu_{V}) \text{ a.s.},
	\]
	where the limit is in the Hilbert--Schmidt norm and
	\begin{align*}
		\mu_{Z}
		&=
		C_{UW} - M_{VU}^{\ast} M_{VW},
		\\
		C_{VZ}
		&=
		\bE[V\otimes (U-\bE^{\ABLin}[U|V]) \otimes (W-\bE^{\ABLin}[W|V])]
		- \mu_{V}\otimes \mu_{Z}.
	\end{align*}
	In particular, in the compatible case with $\ran C_{VW} \subseteq \ran C_{V}$ and $\ran C_{VZ} \subseteq \ran C_{V}$,
	\[
		\Cov^{\ABLin}[U,W|V]
		=
		C_{UW} - C_{UV} C_{V}^{\dagger} C_{VW} +
		(C_{V}^{\dagger}C_{VZ})^{\ast} (V-\mu_{V}) \text{ a.s.}
	\]
\end{corollary}


\proofinappendix


%% file: Section_05_CME_alternative.tex

\section{Application to Kernel Conditional Mean Embeddings}
\label{section:ApplicationToCMEsAlternative}

The above results have a beautiful application to the derivation of conditional mean embeddings (CMEs), a concept used in machine learning to perform conditioning of random variables after embedding them into suitable reproducing kernel Hilbert spaces (RKHSs).
To this end, let $\cH$ and $\cG$ be two RKHSs over measurable spaces $\cX$ and $\cY$ respectively, with reproducing kernels $k$ and $\ell$ and canonical feature maps $\varphi(x) \defeq k(x,\quark)$ and $\psi(y) \defeq \ell(y,\quark)$.

For two random variables $X\colon\Omega\to\cX$ and $Y\colon\Omega\to\cY$ with joint distribution $\bP_{XY}$ and corresponding marginal distributions $\bP_X$ and $\bP_Y$ such that $V \defeq \varphi(X) \in L^2(\Omega,\Sigma,\bP;\cH)$ and $U \defeq \psi(Y) \in L^2(\Omega,\Sigma,\bP;\cG)$, respectively, the CME $\bE[U|X]$ can be characterised by the linear-algebraic transformation
\begin{align}
	\label{equ:CMEFormulaCentered}
	\bE[U|X]
	&=
	\mu_U + (C_{V}^{\dagger} C_{VU})^\ast \, (\varphi(X) - \mu_{V}),
\end{align}
which holds under appropriate technical assumptions \citep{klebanov2019rigorous}.
Formula \eqref{equ:CMEFormulaCentered} can be interpreted as saying that application of the K\'alm\'an update or BLUE formulae to RKHS embeddings of random variables realises the embedding of conditional distributions.

The theory on (affine) linear conditional means from \Cref{section:LinearConditionalMean} provides an alternative proof and a more insightful and explanatory derivation of \eqref{equ:CMEFormulaCentered}.
The main idea is to find conditions under which the conditional expectation $\bE[U|V]$ agrees with the linear conditional expectation $\bE^{\ABLin}[U|V]$.
Assuming $\varphi$ to be injective implies that $\bE[U|X] = \bE[U|V]$ and \eqref{equ:CMEFormulaCentered} then follows directly from \Cref{theorem:LinearConditionalMeanUnderRangeAssumption}, while \Cref{theorem:LinearConditionalMeanIncompatibleCase} implies the more generally applicable formula in \Cref{thm:CMEunderBstar}.
The reason why one could hope for $\bE[U|V] = \bE^{\ABLin}[U|V]$ to hold is the celebrated \emph{kernel trick}, the guiding theme of RKHS-based methods:
many nonlinear problems in the original spaces $\cX$ and $\cY$ (here, conditioning) become linear-algebraic problems when embedded into the corresponding RKHSs $\cH$ and $\cG$.

\subsection{Setup and Notation}
\label{section:CMEAdditionalSetupAndNotation}

Here, with apologies for the large notational overhead relative to the brevity of the results in \Cref{section:DerivationCME}, we introduce the precise technical assumptions and notation needed for the validity of the CME approach;
see \citet[Section 2]{klebanov2019rigorous} for a detailed exposition.

Regarding the kernel mean embedding of random variables $X\colon\Omega\to\cX$ and $Y\colon\Omega\to\cY$ into RKHSs $\cH$ and $\cG$ over $\cX$ and $\cY$, respectively, with reproducing kernels $k$ and $\ell$ via the canonical feature maps $\varphi(x) \defeq k(x,\quark)$ and $\psi(y) \defeq \ell(y,\quark)$ we make the following basic assumptions:

\begin{assumption}
	\label{assumption:CME}
	\begin{enumerate}[label=(\alph*)]
		\item \label{assumpitem:CMEsets_spaces}
		The space $\cX$ is a measurable space and $\cY$ is a Borel space.
		\item \label{assumpitem:CMEsets_kernels_RKHS_featuremaps}
		The kernel functions $k \colon \cX \times \cX \to \bR$ and $\ell \colon \cY \times \cY \to \bR$ are symmetric positive definite and measurable and the corresponding RKHSs $(\cH,\langle \quark,\quark\rangle_{\cH})$ and $(\cG,\langle \quark,\quark\rangle_{\cG})$ are separable.
		Moreover, the canonical feature map $\varphi(x) \defeq k(x,\quark)$ is injective.\footnote{The injectivity of $\varphi$ was not required for the derivations in \citet{klebanov2019rigorous}. However, it represents a minor restriction since one typically considers characteristic kernels, which implies injectivity of $\varphi$.}
		\item \label{assumpitem:CMErandomVariablesAndSpaces}
		The random variables $V \defeq \varphi(X)$ and $U \defeq \psi(Y)$ lie in $ L^2(\Omega,\Sigma,\bP;\cG)$ and $L^2(\Omega,\Sigma,\bP;\cH)$, respectively.
		\item \label{assumpitem:NoNontrivialZerosInH}
		For any $h\in\cH$ we have $\|h\|_{\cH} = 0$ if and only if $h=0$ $\bP_X$-a.e.\ in $\cX$.
	\end{enumerate}
\end{assumption}

By \citet[Theorem~5.3]{kallenberg2006foundations}, \Cref{assumption:CME}\ref{assumpitem:CMEsets_spaces} ensures the existence of a $\bP_X$-a.e.-unique regular version of the conditional probability distribution $\bP_{Y|X=x}$, $x\in\cX$, for random variables $X\colon \Omega\to\cX$ and $Y\colon\Omega\to\cY$.
Moreover, by \citet[Lemma~4.25]{steinwart2008support}, \Cref{assumption:CME}\ref{assumpitem:CMEsets_kernels_RKHS_featuremaps} guarantees the measurability of $\varphi\colon \cX \to \cH$ and $\psi\colon \cY \to \cG$, respectively.
\Cref{assumption:CME}\ref{assumpitem:CMErandomVariablesAndSpaces} implies that $\cH$ (resp.\ $\cG$) is continuously embedded in the pre-Hilbert space $\cL^2(\bP_X)$ (resp.\ $\cL^2(\bP_Y)$).
Furthermore, it also follows that $\bE[\norm{\psi(Y)}_{\cG}^{2}|X=x] < \infty$ and that $\cG$ is continuously embedded in $\cL^2(\bP_{Y|X=x})$ for all $x\in\cX_{Y}$, where $\cX_{Y}\subseteq \cX$ has full $\bP_{X}$ measure, see \citet[Section 2]{klebanov2019rigorous}.
\Cref{assumption:CME}\ref{assumpitem:NoNontrivialZerosInH} clearly holds if $k$ is continuous and $\supp(\bP_X) = \cX$.
In particular, it allows us to view $\cH$ as a subspace of the Lebesgue Hilbert space $L^2(\bP_X)$.

Subsequently, we work with Bochner spaces $L^{2}(\bP_{X};\cF)$ where $\cF$ denotes another separable Hilbert space (which in our case will be equal to either $\bR$ or $\cG$).
Recall that the space $L^{2}(\bP_{X};\cF)$ is isometrically isomorphic to the Hilbert tensor product space $\cF \otimes L^{2}(\bP_{X})$.
We comment on the various perspectives on tensor product space exploited in our proofs in Remark \ref{remark:ViewTensorProductAsHS} below.
For stating our second main result, we require the following definitions.

\begin{notation}\label{notation:CME}
	\begin{enumerate}[label=(\alph*)]
		\item \label{item:L2andL2C}
		Given a separable Hilbert space $\cF$ we define $L_{\cC}^{2}(\bP_{X}; \cF)$ to be the quotient space $L^{2}(\bP_{X}; \cF) / \cC$, where
		\begin{align*}
			\cC
			\defeq
			\{ f\in L^{2}(\bP_{X}; \cF) \mid \exists c\in \cF\colon f(x) = c \text{ for $\bP_{X}$-a.e.\ } x\in\cX \},
			\\[0.6ex]
			\innerprod{ [f_1] }{ [f_2] }_{L_{\cC}^{2}(\bP_{X}; \cF)}
			\defeq \innerprod{ f_1 - \bE[f_1(X)] }{ f_2 - \bE[f_2(X)] }_{L^{2}(\bP_{X}; \cF)}.
		\end{align*}
		In the case $\cF = \bR$, we abbreviate the space $L_{\cC}^{2}(\bP_{X}; \bR)$ by $L_\cC^2(\bP_{X})$ and for any subspace $\cU\subseteq L^{2}(\bP_{X};\cF)$ we define $\cU_{\cC} \defeq \cU / (\cU \cap \cC)$ and identify it with a subspace of $L_{\cC}^{2}(\bP_{X}; \cF)$.

		\item
		\label{assumpitem:ConditionalMeanInL2G}
		Furthermore, we define the main object of our interest, the conditional mean
		\[
			\fm\colon\cX\to\cG,
			\qquad
			\fm(x)
			\defeq
			\begin{cases}
				\bE[U|X=x], & \text{for } x\in\cX_{Y},\\
				0, & \text{otherwise.}
			\end{cases}
		\]
		Note that $\fm\in L^{2}(\bP_{X};\cG)$, since Jensen's inequality, the law of total expectation, and \Cref{assumption:CME}\ref{assumpitem:CMErandomVariablesAndSpaces} together yield that
		\[
			\norm{\fm}_{L^{2}(\bP_{X};\cG)}^{2}
			=
			\bE\bigl[\norm{\bE[\psi(Y)|X]}_{\cG}^{2}\bigr]
			\le
			\bE\bigl[\bE[\norm{U}_{\cG}^{2}|X]\bigr]
			=
			\bE\bigl[\norm{U}_{\cG}^{2}\bigr]
			<
			\infty.
		\]

		\item
		\label{assumpitem:FG}
		We also introduce the notation
		\[
			f_g(x) \defeq \bE[g(Y)| X=x] = \innerprod{g}{\fm(x)}_{\cG},
		\]
		mainly for the comparison of our formulations to \cite{klebanov2019rigorous}.
	\end{enumerate}
\end{notation}

\begin{remark}
	\label{remark:ViewTensorProductAsHS}
	For $\cF = \cH$ and $\cF = L^2(\bP_{X})$, we will view the Hilbert tensor product space $\cG\otimes\cF$ from three\footnote{In fact, there is another viewpoint on $\cG\otimes\cF$, namely as a set of functions from $Y\times X$ to $\bR$, where $(g\otimes f) (y,x) \defeq g(y)f(x)$, in which case $\cG\otimes\cF \subseteq L^2(\bP_{Y} \otimes \bP_{X})$.} perspectives:
	\begin{itemize}
		\item
		$\cG\otimes\cF$ is isometrically isomorphic to the space $\HS(\cF;\cG)$ of Hilbert--Schmidt operators from $\cF$ to $\cG$ \citep[Chapter 12]{aubin2000} and we sometimes view $g\otimes f \in \cG\otimes\cF \cong \HS(\cF;\cG)$ as the corresponding mapping from $\cF$ to $\cG$ given by
		\[
				[g\otimes f]_{\cF\to\cG}(f') \defeq\innerprod{f}{f'}_{\cF}\, g.
		\]
		\item
		$\cG\otimes\cF$ can be viewed as a set of functions from $\cX$ to $\cG$ \citep[Chapter 12]{aubin2000}.
		Thus, we can view $g\otimes f \in \cG\otimes\cF \subseteq L^2(\bP_{X};\cG)$ accordingly as a function on $\cX$ taking values in $\cG$:
		\[
		[g\otimes f]_{\cX\to\cG}(x) \defeq f(x)\, g
		=
		[g\otimes f]_{\cH\to\cG}(\varphi(x)),
		\]
		where the last equality holds for $\cF = \cH$ by the reproducing property, in which case we have established the important observation
		\begin{equation}
		\label{equ:IdentityFromViewingOtimesAsHS}
		\ff_{\cX\to\cG}(x) = \ff_{\cH\to\cG}(\varphi(x)),
		\qquad
		\ff \in \cG\otimes\cH,
		\
		x\in\cX.
		\end{equation}
		Note that $\cG\otimes\cF$ is indeed a subspace of $L^2(\bP_{X};\cG)$, which is obvious in the case $\cF = L^2(\bP_{X})$, and, in the case $\cF = \cH$, follows from
		\[
		\bignorm{\ff_{\cX\to\cG}}_{L^{2}(\bP_{X};\cG)}^{2}
		=
		\bE\bigl[\bignorm{\ff_{\cX\to\cG}(X)}_{\cG}^{2}\bigr]
		\stackrel{\eqref{equ:IdentityFromViewingOtimesAsHS}}{\le}
		\bignorm{\ff_{\cH\to\cG}}_{\HS}^{2}\, \bE\bigl[ \bignorm{\varphi(X)}_{\cH}^{2}\bigr]
		<
		\infty,
		\]
		where we used \Cref{assumption:CME}\ref{assumpitem:CMErandomVariablesAndSpaces} in the last step.
		This is hardly surprising, since $\cH \subseteq L^2(\bP_{X})$ by \Cref{assumption:CME}\ref{assumpitem:CMErandomVariablesAndSpaces},\ref{assumpitem:NoNontrivialZerosInH}, but not trivial as we use the RKHS norm $\norm{\quark}_{\cH}$ in the construction of $\cG\otimes \cH$, which may not agree with $\norm{\quark}_{L^2(\bP_{X})}$.
		\item
		Since tensor products are commutative up to isometric isomorphism, $\cG\otimes\cF$ is also isometrically isomorphic to $\HS(\cG;\cF)$ and we can analogously set
		\[
		[g \otimes f]_{\cG \to \cF}(g') \defeq \innerprod{g}{g'}_{\cF}\, f, \qquad f \in \cF, \ g,\, g' \in \cG.
		\]
		We will sometimes use the resulting identities for arbitrary $\ff \in \cG\otimes\cF$: with $x\in\cX$
		\begin{equation}
		\label{equ:SomeTensorProductIdentities}
		\ff_{\cG\to\cF}(g)(x)
		=
		\innerprod{\ff_{\cX\to\cG}(x)}{g}_{\cG},
		\qquad
		\innerprod{\ff}{g\otimes f}_{\cG\otimes\cF}
		=
		\innerprod{\ff_{\cG\to\cF}(g)}{f}_{\cF}.
		\end{equation}
	\end{itemize}
	However, we will drop the indices $\cF\to\cG$, $\cX\to\cG$ and $\cG\to\cF$ in the following, since it will always be clear which version we mean, whenever we apply $\ff\in \cG\otimes\cF$ to some element of $\cF$, $\cX$, or $\cG$, respectively.
\end{remark}

The typical assumption for CMEs is that the functions $f_g$ introduced in \Cref{notation:CME}\ref{assumpitem:FG} must lie in $\cH$ for all $g\in\cG$.
\citet{klebanov2019rigorous} discuss several weaker assumptions on $f_g$, which we are going to adopt in this paper.
However, the main purpose of using $f_g$ is that, for $g=\psi(y)$ with $y\in\cY$, $f_{\psi(y)} = \mu_{Y|X = \quark}(y)$ and, in fact, all results in \citet{klebanov2019rigorous} rely solely on these special cases of $f_g$.
It is therefore meaningful to restate these assumptions in terms of $\fm$ rather than $f_g$.

\begin{assumption}
	\label{assumption:AssumptionHierarchyCMEalmosteverywhere}
	Under \Cref{assumption:CME} and using \Cref{notation:CME}, we introduce the following assumptions on the functions $\fm \in L^{2}(\bP_{X};\cG)$:
	\begin{itemize}
		\item[\mylabel{assump:strongCME}{($\text{A}$)}]
		$\fm \in \cG\otimes \cH$;
		\item[\mylabel{assump:weakerCME}{($\text{B}$)}]
		$[\fm] \in (\cG\otimes \cH)_{\cC}$;
		\item[\mylabel{assump:weakCME}{($\text{C}$)}]
		$\displaystyle P_{\overline{(\cG\otimes \cH)_{\cC}}^{L_{\cC}^{2}(\bP_{X};\cG)}} [\fm] \in (\cG\otimes \cH)_{\cC}$;
		\item[\mylabel{assump:weakCMEuncentred}{($\uu{\text{C}}$)}]
		$\displaystyle P_{\overline{\cG\otimes \cH}^{L^{2}(\bP_{X};\cG)}} \fm \in \cG\otimes \cH$;
		\item[\mylabel{assump:strongCMElimit}{($\text{A}^{\ast}$)}]
		$\fm \in \overline{\cG\otimes \cH}^{L^{2}(\bP_{X};\cG)}$;
		\item[\mylabel{assump:weakerCMElimit}{($\text{B}^{\ast}$)}]
		$[\fm] \in \overline{(\cG\otimes \cH)_{\cC}}^{L_{\cC}^{2}(\bP_{X};\cG)}$.
	\end{itemize}
\end{assumption}

\begin{remark}
	\label{remark:UniversalityImpliesAstar}
	Note that these assumptions are slightly stronger than the corresponding assumptions on $f_g$ in \cite[Section 3]{klebanov2019rigorous}.
	The corresponding implications are formulated in \Cref{proposition:AssumptionHierarchyCMEimpliesOldVersion} below.
	However, by \Cref{lemma:CharacteristicImpliesBstar}, \ref{assump:weakerCMElimit} still follows from $k$ being characteristic, providing a verifiable condition for the applicability of the corresponding CME formula (\Cref{thm:CMEunderBstar}).
	Also, as before, \ref{assump:strongCMElimit} follows from $k$ being $L^{2}$-universal.
	Indeed, if $\cG\otimes \cH$ is dense in $L^2(\bP_{X};\cG)$, then
	\[
		\overline{\cG\otimes \cH}^{L^{2}(\bP_{X};\cG)}
		\subseteq
		\overline{\cG \otimes \overline{\cH}^{L^{2}(\bP_{X})}}^{\cG \otimes L^{2}(\bP_{X})}
		=
		\overline{\cG \otimes L^{2}(\bP_{X})}^{\cG \otimes L^{2}(\bP_{X})}
		=
		L^{2}(\bP_{X}; \cG).
	\]
	These and further relations among the conditions in \Cref{assumption:AssumptionHierarchyCMEalmosteverywhere} are summarised in \Cref{fig:HierarchyOfAssumptions}.
\end{remark}


\begin{proposition}
	\label{proposition:AssumptionHierarchyCMEimpliesOldVersion}
	The conditions in \Cref{assumption:AssumptionHierarchyCMEalmosteverywhere} imply the corresponding assumptions on $f_g$ in \citet[Section 3]{klebanov2019rigorous} (here marked with a subscript ``old'').
	More precisely, under \Cref{assumption:CME} and using \Cref{notation:CME},
	\begin{itemize}[leftmargin=10em]
		\item[\mylabel{assump:strongCMEold}{($\text{A}_{\textup{old}}$)}]
		\makebox[0pt][r]{\ref{assump:strongCME}$\quad\Longrightarrow$ \hspace{4em}}		
		$f_g \in \cH$ for each $g\in\cG$;
		\item[\mylabel{assump:weakerCMEold}{($\text{B}_{\textup{old}}$)}]
		\makebox[0pt][r]{\ref{assump:weakerCME}$\quad\Longrightarrow$ \hspace{4em}}	
		$[f_g] \in \cH_{\cC}$ for each $g\in\cG$;
		\item[\mylabel{assump:weakCMEold}{($\text{C}_{\textup{old}}$)}]
		\makebox[0pt][r]{\ref{assump:weakCME}$\quad\Longrightarrow$ \hspace{4em}}	
		$\displaystyle P_{\overline{\cH_{\cC}}^{L_{\cC}^{2}(\bP_{X})}} [f_g] \in \cH_{\cC}$ for each $g\in\cG$;
		\item[\mylabel{assump:weakCMEuncentredold}{($\uu{\text{C}}_{\textup{old}}$)}]
		\makebox[0pt][r]{\ref{assump:weakCMEuncentred}$\quad\Longrightarrow$ \hspace{4em}}	
		$\displaystyle P_{\overline{\cH}^{L^{2}(\bP_{X})}} f_g \in \cH$ for each $g\in\cG$;
		\item[\mylabel{assump:strongCMElimitold}{($\text{A}_{\textup{old}}^{\ast}$)}]
		\makebox[0pt][r]{\ref{assump:strongCMElimit}$\quad\Longrightarrow$ \hspace{4em}}	
		$f_g \in \overline{\cH}^{L^{2}(\bP_{X})}$ for each $g\in\cG$;
		\item[\mylabel{assump:weakerCMElimitold}{($\text{B}_{\textup{old}}^{\ast}$)}]
		\makebox[0pt][r]{\ref{assump:weakerCMElimit}$\quad\Longrightarrow$ \hspace{4em}}	
		$f_g \in \overline{(\cH)_{\cC}}^{L_{\cC}^{2}(\bP_{X})}$ for each $g\in\cG$.
	\end{itemize}
	Further,	
	\begin{itemize}[leftmargin=7em]
		\item[]
		\makebox[0pt][r]{\ref{assump:weakCME}$\quad\Longrightarrow$ \hspace{1em}}		
		$\ran C_{VU}\subseteq \ran C_{V}$;
		\item[]
		\makebox[0pt][r]{\ref{assump:weakCMEuncentred}$\quad\Longrightarrow$ \hspace{1em}}	
		$\ran \uu C_{VU}\subseteq \ran \uu C_{V}$.		
	\end{itemize}	
\end{proposition}


\proofinappendix


\begin{lemma}
	\label{lemma:CharacteristicImpliesBstar}
	Under \Cref{assumption:CME} and using \Cref{notation:CME}, if $k$ is a characteristic kernel, then $(\cG\otimes \cH)_{\cC}$ is dense in $L_{\cC}^{2}(\bP_{X}; \cG)$ and \Cref{assump:weakerCMElimit}\ref{assump:weakerCMElimit} is satisfied.	
\end{lemma}


\proofinappendix


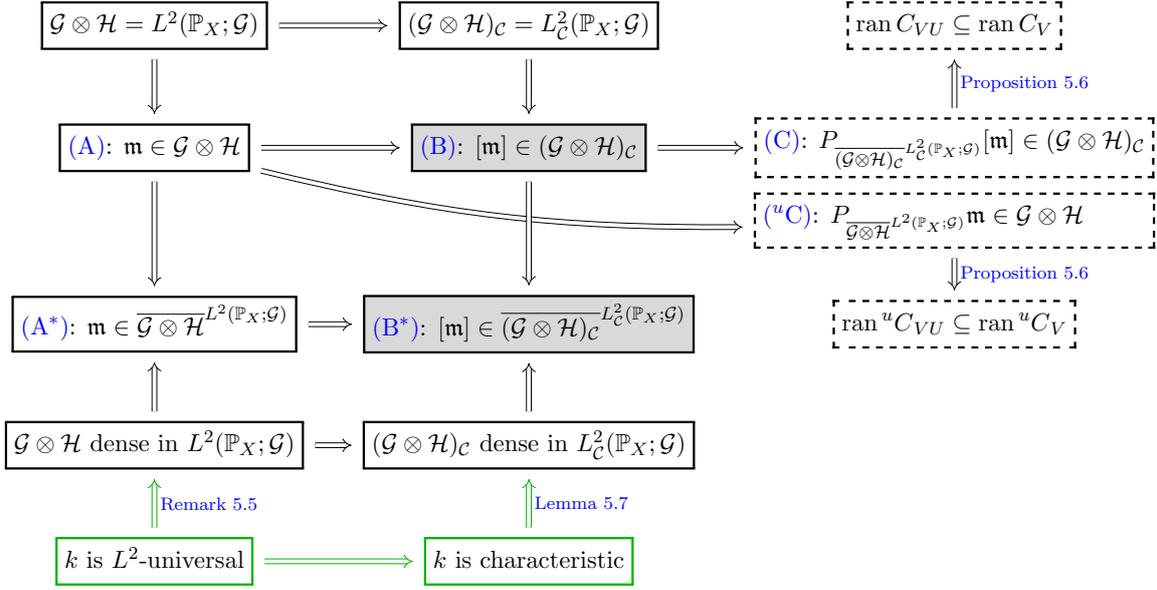
\begin{figure}[t]
	\centering
	\adjustbox{scale=0.85}{
		\begin{tikzcd}[column sep=1.5em,row sep=2em]
			\mlnode{$\cG\otimes \cH = L^2(\bP_{X};\cG)$}
			\arrow[r,Rightarrow]
			\arrow[d,Rightarrow]
			&
			\mlnode{$(\cG\otimes \cH)_{\cC} = L_{\cC}^{2}(\bP_{X};\cG)$}
			\arrow[d,Rightarrow]
			&
			\dashedmlnode{$\ran C_{VU}\subseteq \ran C_{V}$}
			\arrow[d,Leftarrow,"\text{\Cref{proposition:AssumptionHierarchyCMEimpliesOldVersion}}"]
			\\
			\mlnode{\ref{assump:strongCME}: $\fm \in \cG\otimes \cH$}
			\arrow[r,Rightarrow]
			\arrow[dd,Rightarrow]
			\arrow[bend right=7,drr,Rightarrow]
			&
			\filledmlnode{\ref{assump:weakerCME}: $[\fm] \in (\cG\otimes \cH)_{\cC}$}
			\arrow[r,Rightarrow]
			\arrow[dd,Rightarrow]
			&
			\dashedmlnode{\ref{assump:weakCME}: $P_{\overline{(\cG\otimes \cH)_{\cC}}^{L_{\cC}^{2}(\bP_{X};\cG)}} [\fm] \in (\cG\otimes \cH)_{\cC}$}
			\\[-5ex]
			&
			&
			\dashedmlnode{\ref{assump:weakCMEuncentred}: $P_{\overline{\cG\otimes \cH}^{L^{2}(\bP_{X};\cG)}} \fm \in \cG\otimes \cH$\hspace{2.5em}\ }
			\arrow[d,Rightarrow,"\text{\Cref{proposition:AssumptionHierarchyCMEimpliesOldVersion}}"]
			\\[-2ex]
			\mlnode{\ref{assump:strongCMElimit}: $\fm \in \overline{\cG\otimes \cH}^{L^{2}(\bP_{X};\cG)}$}
			\arrow[r,Rightarrow]
			&
			\filledmlnode{\ref{assump:weakerCMElimit}: $[\fm] \in \overline{(\cG\otimes \cH)_{\cC}}^{L_{\cC}^{2}(\bP_{X};\cG)}$}
			\arrow[d,Leftarrow]
			&
			\dashedmlnode{$\ran \uu C_{VU}\subseteq \ran \uu C_{V}$}
			\\
			\mlnode{$\cG\otimes \cH \text{ dense in } L^2(\bP_{X};\cG)$}
			\arrow[r,Rightarrow]
			\arrow[u,Rightarrow]
			\arrow[d,Leftarrow,boxgreen,"\text{\Cref{remark:UniversalityImpliesAstar}}"]
			&
			\mlnode{$(\cG\otimes \cH)_{\cC} \text{ dense in } L_{\cC}^{2}(\bP_{X};\cG)$}
			\arrow[d,Leftarrow,boxgreen,"\text{\Cref{lemma:CharacteristicImpliesBstar}}"]
			&
			\\
			\greenmlnode{$k$ is $ L^2$-universal}
			\arrow[r,Rightarrow,boxgreen]
			&
			\greenmlnode{$k$ is characteristic}
			&
		\end{tikzcd}
	}
	\caption{A hierarchy of CME-related assumptions.
		Sufficient conditions for validity of the CME formula are indicated by solid boxes while the insufficient Assumptions \ref{assump:weakCME} and \ref{assump:weakCMEuncentred}, indicated by dashed boxes, have several strong theoretical implications.
		\Cref{assump:weakerCMElimit}\ref{assump:weakerCMElimit} is the most favorable one, since it is verifiable in practice, and, by \Cref{lemma:CharacteristicImpliesBstar}, in particular is fulfilled if the kernel is universal or even just characteristic (marked in green).
		The shaded boxes correspond to \Cref{thm:CMEunderB,thm:CMEunderBstar}.
	}
	\label{fig:HierarchyOfAssumptions}
\end{figure}

\subsection{Derivation of the CME Formula}
\label{section:DerivationCME}

We are now in a position to re-derive the CME formula under \Cref{assumption:AssumptionHierarchyCMEalmosteverywhere}.


\begin{theorem}[{CME under \Cref{assumption:AssumptionHierarchyCMEalmosteverywhere}\ref{assump:strongCME} or \ref{assump:weakerCME}}]
	\label{thm:CMEunderB}
	Under \Cref{assumption:CME,assumption:AssumptionHierarchyCMEalmosteverywhere}\ref{assump:weakerCME} the operator $C_{V}^{\dagger} C_{VU}\colon \cG\to\cH$ is bounded and
	\begin{equation}
		\label{equ:CentredCMEunderB}
		\bE[U|X]
		=
		\mu_{U} + (C_{V}^{\dagger} C_{VU})^\ast (\varphi(X) - \mu_{V})
		\quad
		\text{a.s.}
	\end{equation}
\end{theorem}


\proofinappendix


\begin{remark}
	Note that we have proven a stronger statement than originally intended.
	Namely, the CME operator $(C_{V}^{\dagger} C_{VU})^\ast $ is not just bounded, but even Hilbert--Schmidt.
	However, it can be argued that this property is already hidden in the assumptions, namely in \Cref{assumption:AssumptionHierarchyCMEalmosteverywhere}\ref{assump:weakerCME}, since $\cG\otimes \cH \cong \HS(\cH;\cG)$;
\end{remark}

\begin{remark}
	A similar statement can be proven for uncentered covariance operators under the stronger \Cref{assumption:AssumptionHierarchyCMEalmosteverywhere}\ref{assump:strongCME}, see \citet[Theorem~5.3]{klebanov2019rigorous}.
\end{remark}


\begin{theorem}[{CME under \Cref{assumption:AssumptionHierarchyCMEalmosteverywhere}\ref{assump:weakerCMElimit}}]
	\label{thm:CMEunderBstar}
	Under \Cref{assumption:CME,assumption:AssumptionHierarchyCMEalmosteverywhere}\ref{assump:weakerCMElimit} and using \Cref{notation:NotationForIncompatibleLCM}, the operators $\gamma_{U|V}^{(n)}$ satisfy
	\begin{align*}
		&\bignorm{\bE[U|X] - \gamma_{U|V}^{(n)}(\varphi(X))}_{L^2(\bP;\cG)}
		\xrightarrow[n \to \infty]{} 0,
		\\
		&\bignorm{\bE[U|X=x] - \gamma_{U|V}^{(n)}(\varphi(x)) }_{\cG}
		\xrightarrow[n \to \infty]{} 0
		& \text{for $\bP_{X}$-a.e.\ $x\in\cX$.}
	\end{align*}
\end{theorem}


\proofinappendix


%% file: Section_06_GaussianConditioning.tex

\section{Application to Gaussian Conditioning in Hilbert Spaces}
\label{section:GaussianConditioning}

While the conditioning of a Gaussian random variable $(U,V)\colon \Omega \to \cG \oplus \cH$ on its second component is a well-established concept \citep{mandelbaum1984linear,hairer2005analysis}, the most general case (where $C_{V}$ is not necessarily injective) has only been treated rather recently by \citet{owhadi2015conditioning}.
In that work, by developing an approximation theory for shorted operators in terms of oblique projections and applying the martingale convergence theorem, the authors derive approximating sequences for both the conditional expectation $\bE[U|V]$ and the conditional covariance operator $\Cov[U|V]$.

The formula that \citet[Theorem~3.3]{owhadi2015conditioning} obtain for the conditional expectation $\bE[U|V]$ is identical to \eqref{equ:ASconvergenceToLCE} (with $\bE[U|V]$ in place of $\bE^{\ABLin}[U|V]$).
Similar to CMEs in \Cref{section:ApplicationToCMEsAlternative}, our theory provides an alternative derivation of this formula by
\begin{enumerate}[label = (\roman*)]
\item \label{item:GaussianCEFlinear}
proving $\bE[U|V] \in \overline{\ABLin_{\cG}\circ V}$, implying the identity $\bE[U|V] = \bE^{\ABLin}[U|V]$;
\item applying \Cref{theorem:LinearConditionalMeanIncompatibleCase}.
\end{enumerate}
Let us give a short sketch of the proof:



\begin{proof}[Proof sketch for \ref{item:GaussianCEFlinear}]
	Let $\overline{U} \defeq U - \mu_{U}$, $\overline{V} \defeq V - \mu_{V}$ and $\gamma \defeq \gamma_{\overline{U}|\overline{V}}\colon \cH\to\cG$ be the corresponding CEF.
	\citet[Theorem~3.11]{tarieladze2007disintegration} show that there exists a linear subspace $\tilde{\cH}$ of $\cH$ such that $\overline{V}\in \tilde{\cH}$ a.s.\ and the restriction $\gamma |_{\tilde{\cH}}$ is linear.
	In the proof, the authors further construct a sequence $\gamma_n\in \BLin(\cH;\cG)$ such that
	\[
	\gamma_n(h) = \sum_{i=1}^{n} \innerprod{h}{h_i} \gamma(h_i)
	\xrightarrow[n\to\infty]{}
	\gamma (h)
	\quad
	\text{for all } h\in \tilde{\cH}.
	\]
	Using the Karhunen--Lo\`eve expansion of $\overline{V}$, one can prove
	\[
	\bignorm{\gamma_{n}\circ \overline{V} - \gamma\circ \overline{V}}_{L^{2}(\bP;\cG)}^{2}
	\xrightarrow[n\to\infty]{}
	0.
	\]
	Hence $\bE[\overline{U}|\overline{V}] = \gamma\circ \overline{V} \in \overline{\BLin_{\cG} \circ V}$ and thereby
	$\bE[U|V] = \mu_{U} + \gamma\circ (V-\mu_{V}) \in \overline{\ABLin_{\cG} \circ V}$.
\end{proof}


The appeal to \citet[Theorem~3.11]{tarieladze2007disintegration} is somewhat unsatisfactory, since close inspection of the proof of that theorem reveals that it in fact establishes the entire conditional mean formula for Gaussian conditioning.
In this sense, our derivation of the formula for the conditional \emph{mean} $\bE[U|V]$ in the Gaussian case is not novel.
However, let us now turn to the conditional \emph{covariance} $\Cov[U|V]$.

It is well known that the conditional covariance is constant for Gaussian random variables (i.e.\ it does not depend on the value of the conditioning variable), and that, since $\bE[U|V] = \bE^{\ABLin}[U|V]$, it coincides with the ALCC, $\Cov[U|V] = \Cov^{\ABLin}_{V}[U]$.
Therefore, our results show that, in contrast to the conditional expectation, the conditional covariance does require the approximating sequence established by \citet[Theorem~3.4]{owhadi2015conditioning}, but the explicit formula from \Cref{thm:PropertiesLinearConditionalCovariance}\ref{item:GaussianFormulaForConditionalCovariance} applies.
In summary, we can consider three versions of the Gaussian conditional covariance formula:
\begin{itemize}
	\item
	the \emph{invertible case} in which $C_{V}$ is invertible and, in particular, $\cH$ is finite dimensional:
	\[
	\Cov[U|V]
	=
	C_{U} - C_{UV} C_{V}^{-1} C_{VU}.
	\]
	\item
	the \emph{compatible case} in which $\ran C_{VU}\subseteq \ran C_{V}$:
	\\[1ex]
	By \Cref{theorem:DouglasExistenceQ}, the operator $C_V^\dagger C_{VU}\in\BLin(\cG;\cH)$ is well defined and bounded and
	\[
	\Cov[U|V]
	=
	C_{U} - C_{UV} C_{V}^{\dagger} C_{VU}.
	\]
	\item
	the \emph{incompatible} (or \emph{general}) \emph{case}:
	\\[1ex]
	By \Cref{theorem:BakerDecompositionCrosscovarianceOperator}, the operator $M_{VU}\defeq (C_{V}^{1/2})^{\dagger} C_{VU}$ is well-defined and bounded and
	\[
	\Cov[U|V]
	=
	C_{U} - M_{VU}^{\ast} M_{VU}.
	\]
\end{itemize}

%% file: Section_07_ClosingRemarks.tex

\section{Closing Remarks}
\label{section:ClosingRemarks}

This paper presents a rigorous theory of the linear conditional expectation (LCE) $\bE^{\ABLin}[\quark | V]$ that strongly extends the existing theory on Bayes linear analysis \citep{GoldsteinWooff2007}.

After the definitions of the linear conditional expectation $\bE^{\ABLin}[U | V]$ and the linear conditional covariance (LCC) operator $\Cov^{\ABLin}[U,W|V]$ --- which is related to, but differs from, the so-called adjusted covariance used in Bayes linear statistics --- we studied in detail which properties of the common conditional expectation $\bE[\quark | V]$ and conditional covariance $\Cov[\quark, \quark|V]$ hold for their linear approximations.
Amongst others, we proved several tower properties and the laws of total expectation and total covariance.
On the other hand, $\bE^{\ABLin}[\quark | V]$ is neither monotonic, nor contractive in $L^{p}(\bP;\cG)$ (except for $p=2$, which is clear from its definition) and does not fulfil the triangle inequality.
The dominated convergence theorem holds only under modified assumptions and, so far, could only be proved under the assumption that $C_{V}$ has finite rank (see \Cref{theorem:BasicPropertiesLCM}\ref{item:DominatedConvergenceLCMmodified}, \Cref{remark:DCT}, and \Cref{counterexample:CounterexamplesPropertiesLCM}\ref{item:DominatedConvergenceLCMconventional}).

We derived explicit formulae for both the LCE and the LCC, distinguishing between the so-called compatible (simple) and the incompatible (hard) case, as well as providing a regularised formula for the LCE.

Naturally, whenever $\bE^{\ABLin}[U | V] = \bE[U | V]$, these formulae apply for the common conditional expectation.
This trivial observation allowed us to provide an alternative derivation of the Gaussian conditioning formulae and give a simple and intuitive proof of the widely-used technique of conditional mean embeddings (CMEs) in machine learning:
it turns out that, if $U = \psi(Y)$ and $V = \varphi(X)$ are reproducing kernel Hilbert space (RKHS) embeddings of some random variables $X$ and $Y$, then the above property holds true under rather mild conditions.

One direction for future work is the derivation of optimal regularisation schemes $\varepsilon(n) \to 0$ when the regularised case considered in \Cref{section:RegularizedLCM} is applied to empirical sample data consisting of $n \to \infty$ data points.
We anticipate that this will be a rich vein of research, but also decidedly non-trivial, since such problems admit no general solution and effective strategies (be they a priori, a posteriori, or heuristic) rely on appropriate source conditions for the unknowns.

Finally, we note that this work has concentrated on \emph{centred} \mbox{(cross-)}covariance operators, which are associated with affine approximations of the conditional expectation function.
Some, but not all of the statements can also be proved for \emph{uncentred} operators, which are associated with linear approximations of the CEF and are often used in practice;
for example, the uncentred formulation is commonly used for CMEs.
However, the theory with uncentred operators has weaker statements and is more restrictive, and so we strongly encourage the use of centred operators.

%% file: Section_08_Proofs.tex

\section{Proofs}
\label{section:Proofs}


\begin{proof}[Proof of \Cref{prop:CharacterizeClosureOfBcircV}]
	We only give the proof of the second statement, which is similar to the first one but slightly more technical.
	Let $W \in \overline{\ABLin_{\cG}\circ V}$ and $(\gamma_{n})_{n\in\bN}$ be a sequence in $\ABLin(\cH;\cG)$, such that
	\[
		\norm{\gamma_{n}\circ V - W}_{L^{2}(\bP;\cG)}
		\xrightarrow[n\to\infty]{\text{a.s.}}
		0.
	\]
	This implies that, for $\overline{V} \defeq V-\mu_{V}$ and $\overline{W} \defeq W-\mu_{W}$,
	\[
		\bignorm{\overline{\gamma}_{n}\circ \overline{V} - \overline{W}}_{L^{2}(\bP;\cG)}
		\xrightarrow[n\to\infty]{\text{a.s.}}
		0.
	\]
	By \citet[Corollary 2.32]{folland1999real} there exists a subsequence $(\overline{\gamma}_{n_{k}})_{k\in\bN}$ of $(\overline{\gamma}_{n})_{n\in\bN}$ such that
	\[
		\bignorm{\overline{\gamma}_{n_{k}}\circ \overline{V}(\omega) - \overline{W}(\omega)}_{\cG}
		\xrightarrow[k\to\infty]{}
		0
		\text{ for $\bP$-a.e.\ $\omega\in\Omega$.}
	\]
	Define the linear subspace $\cH_{0} \subseteq \cH$ and the (possibly unbounded) linear operator $A_{0}\colon \cH_{0} \to \cG$ by
	\begin{align*}
		\cH_{0}
		& \defeq
		\Set{ h \in \cH }{ \overline{\gamma}_{n_{k}}(h) \text{ converges in } \cG },
		&
		A_{0}(h)
		& \defeq
		\lim_{k\to\infty} \overline{\gamma}_{n_{k}}(h)
		\text{ for } h\in \cH_{0},
	\end{align*}
	and extend $A_{0}$ trivially\footnote{Choose a Hamel basis $\cB_1$ of $\cH_{0}$, extend it to a basis $\cB_1\cup \cB_2$ of $\cH$ and set $A \defeq \tilde{A}$ on $\cB_1$ and $A \defeq 0$ on $\cB_2$.} to a linear operator $A$ on $\cH$.
	Then $\overline{V}\in \cH_{0}$ a.s.\ and $A \circ \overline{V} = \overline{W}$ a.s.
	Considering the affine operator $\gamma\colon \cH\to \cG$ given by $\gamma(h) = \mu_{W} + A(h-\mu_{V})$ yields that $\gamma\in \ALin_V(\cH;\cG)$ and $\gamma\circ V = W$ a.s.
\end{proof}


\begin{proof}[Proof of \Cref{lemma:CovarianceFormulaForLCM}]
	Since $\bE\bigl[ \bE^{\ABLin}[U|V] \bigr] = \bE[U]$ (which follows from $\bE$ and $\bE^{\ABLin}[\quark |V]$ being orthogonal projections, see the law of total linear expectation in \Cref{theorem:BasicPropertiesLCM}\ref{item:LawOfTotalExpectationLCM}), it follows that $\bE\bigl[ R^{\ABLin}[U|V] \bigr] = 0$.
	Hence, by \Cref{lemma:L2InnerProductTraceCovariance}, for any $\gamma\in\BLin(\cH;\cG)$,
	\[
	0
	=
	\innerprod{R^{\ABLin}[U|V]}{\gamma\circ V}_{L^2(\bP;\cG)}
	=
	\trace \Bigl( \Cov \bigl[R^{\ABLin}[U|V] , V \bigr] \, \gamma^{\ast} \Bigr).
	\]
	By \Cref{lemma:TraceOfProductImpliesZero} this implies that $\Cov\bigl[ R^{\ABLin}[U|V] , V \bigr] = 0$.
	Now let $W\in \overline{\ABLin_{\cF} \circ V}$.
	By \Cref{prop:CharacterizeClosureOfBcircV}, $W = \gamma\circ V$ for some $\gamma\in \ALin_V(\cH;\cF)$.
	Hence, invoking \Cref{lemma:L2InnerProductTraceCovariance} another time,
	\[
	\Cov\bigl[ R^{\ABLin}[U|V] , W \bigr]
	=
	\bigl(\overline{\gamma} \, \Cov\bigl[ V , R^{\ABLin}[U|V] \bigr]\bigr)^{\ast}
	=
	0 ,
	\]
	which completes the proof.
	(Note that the finite trace of the cross-covariance operator was essential in the above argument.)
\end{proof}


\begin{proof}[Proof of \Cref{theorem:BasicPropertiesLCM}]
	Properties \ref{item:StabilityLCM}--\ref{item:TowerPropertyLCM}, except for the second statement on linearity in \ref{item:LinearityLCM} and the second tower property in \ref{item:TowerPropertyLCM}, follow directly from the definitions of $\bE[U]$, $\bE[U|V]$ and $\bE^{\ABLin}[U|V]$ as orthogonal projections of $U$, the identity $\bE[\innerprod{\quark}{\quark}_{\cG}] = \innerprod{\quark}{\quark}_{L^{2}(\bP;\cG)}$ and the inclusions
	\[
		\ABLin(\cH;\cG)\circ V \subseteq L^2(\sigma(V);G) \subseteq L^2(\Sigma;G),
		\qquad
		\ABLin(\cF;\cG)\circ \varphi \subseteq \ABLin(\cH;\cG).
	\]

	For the second statement on linearity in \ref{item:LinearityLCM} first note that
	$\psi\bigl( \bE^{\ABLin}[U|V] \bigr) \in \overline{\ABLin_{\cF}\circ V}$.
	\Cref{lemma:L2InnerProductTraceCovariance,lemma:CovarianceFormulaForLCM} imply that, for any $\gamma \in \ABLin(\cH;\cF)$,
	\[
		\innerprod{\psi(U) - \psi\bigl(\bE^{\ABLin}[U|V]\bigr)}{\gamma\circ V}_{L^{2}(\bP;\cF)}
		=
		\trace\bigl( \psi\, \Cov\bigl[ \bE^{\ABLin}[U|V] , V \bigr] \overline{\gamma}^{\ast}
		-
		\psi\, \Cov[U,V] \overline{\gamma}^{\ast} \bigr)
		=
		0,
	\]
	which completes the proof of \ref{item:LinearityLCM}.

	For second tower property in \ref{item:TowerPropertyLCM} let
	$
	\bE^{\ABLin}[U|V] = \gamma \circ V \in \overline{\ABLin_{\cG}\circ V}
	$
	(using \Cref{prop:CharacterizeClosureOfBcircV}) and assume that there exists $\delta \in \ABLin(\cG;\cG)$ such that
	\[
		\bignorm{U - \delta \circ \bE^{\ABLin}[U|V]}_{L^2(\bP;\cG)}
		<
		\bignorm{U - \bE^{\ABLin}[U|V]}_{L^2(\bP;\cG)}.
	\]
	Then $\delta \circ \gamma \circ V \in \overline{\ABLin_{\cG}\circ V}$ is a better $L^2(\bP;\cG)$-approximation of $U$ than $\bE^{\ABLin}[U|V]$, which contradicts the definition of $\bE^{\ABLin}[U|V]$.

	For the law of total linear covariance \ref{item:LawOfTotalCovarianceLCM}, first note that, by the law of total linear expectation \ref{item:LawOfTotalExpectationLCM} and \Cref{lemma:CovarianceFormulaForLCM},
	\[
		\bE\bigl[ \Cov^{\ABLin}[U,W|V] \bigr]
		=
		\bE \bigl[ \bE^{\ABLin} \bigl[ R^{\ABLin}[U|V] \otimes R^{\ABLin}[W|V] \, \big| \, V\bigr] \bigr]
		=
		\Cov \bigl[ R^{\ABLin}[U|V] , R^{\ABLin}[W|V]\bigr].
	\]
	Hence, again by \Cref{lemma:CovarianceFormulaForLCM},
	\begin{align*}
		\Cov[U,W]
		&=
		\Cov\bigl[ \bE^{\ABLin}[U|V] + R^{\ABLin}[U|V] , \bE^{\ABLin}[W|V] + R^{\ABLin}[W|V] \bigr]
		\\
		&=
		\Cov\bigl[ \bE^{\ABLin}[U|V], \bE^{\ABLin}[W|V] \bigr]
		+ 0 + 0 +
		\Cov\bigl[ R^{\ABLin}[U|V] , R^{\ABLin}[W|V] \bigr]
		\\
		&=
		\Cov\bigl[ \bE^{\ABLin}[U|V], \bE^{\ABLin}[W|V] \bigr]
		+
		\bE\bigl[ \Cov^{\ABLin}[U,W|V] \bigr],
	\end{align*}
	proving the law of total linear covariance in \ref{item:LawOfTotalCovarianceLCM}.
	The inequality $\Cov[U] \geq \Cov\bigl[\bE[U|V]\bigr]$ is well known (it follows from the common law of total covariance).
	For the second inequality, let $U'\defeq \bE[U|V]$.
	By the first tower property in \ref{item:TowerPropertyLCM}, $\bE^{\ABLin}[U'|V] = \bE^{\ABLin}[U|V]$.
	Hence, by the law of total linear covariance that was just established,
	\[
		\Cov[U']
		=
		\Cov\bigl[\bE^{\ABLin}[U'|V]\bigr] + \bE\bigl[ \Cov^{\ABLin}[U'|V] \bigr]
		=
		\Cov\bigl[\bE^{\ABLin}[U|V]\bigr] + \Cov\bigl[ R^{\ABLin}[U'|V] \bigr]
		\ge
		\Cov\bigl[\bE^{\ABLin}[U|V]\bigr],
	\]
	where we used \Cref{lemma:CovarianceFormulaForLCM} and the law of total linear expectation \ref{item:LawOfTotalExpectationLCM} in the second step, finalising the proof of \ref{item:LawOfTotalCovarianceLCM}.

	In order to prove \ref{item:PullingOutIndependentLCM}, note that the independence of $W$ and $(U,V)$ implies
	\[
		\Cov[W\otimes U , V]
		=
		\bE[W \otimes U \otimes V] - \bE[W \otimes U] \otimes \mu_{V}
		=
		\mu_{W} \otimes \bE[U\otimes V] - \mu_{W}\otimes \mu_{U} \otimes \mu_{V}
		=
		\mu_{W} \otimes C_{UV}.
	\]
	Further, by \Cref{lemma:CovarianceFormulaForLCM}, $\Cov\bigl[ \bE^{\ABLin}[U|V] , V \bigr] = C_{UV}$.
	Since $\bE\bigl[ \mu_{W} \otimes \bE^{\ABLin}[U|V] \bigr] = \mu_{W}\otimes \mu_{U} = \bE[W \otimes U]$, \Cref{lemma:L2InnerProductTraceCovariance} implies that, for any $\gamma\in \ABLin(\cH;\cF \otimes \cG)$,
	\begin{align*}
		\innerprod{W\otimes U - \mu_{W}\otimes \bE^{\ABLin}[U|V]}{\gamma \circ V}_{L^2(\bP;\cF \otimes \cG)}
		&=
		\trace \bigl(
		\Cov[W \otimes U - \mu_{W} \otimes \bE^{\ABLin}[U|V] , V \bigr]
		\, \overline{\gamma}^{\ast} \bigr)
		\\
		&=
		\trace \Bigl(\bigl(
		\Cov[W \otimes U , V]
		-
		\mu_{W} \otimes C_{UV}
		\bigr)\, \overline{\gamma}^{\ast} \Bigr)
		\\
		&=
		0.
	\end{align*}
	Therefore, $\mu_{W}\otimes \bE^{\ABLin}[U|V]$ is the $L^2(\bP;\cF \otimes \cG)$-orthogonal projection of $W \otimes U$ onto
	\linebreak
	$\overline{\ABLin_{\cF\otimes\cG}\circ V}$.

	In order to prove\footnote{A simpler proof can be obtained from using \Cref{theorem:LinearConditionalMeanUnderRangeAssumption}: After establishing \eqref{equ:ProofOfDCTconvergenceOfMeanAndCovariance}, \ref{item:DominatedConvergenceLCMmodified} follows from
		\[
		\bE^{\ABLin}[U_{k} | V]
		=
		\mu_{U_{k}} + (C_{V}^{\dagger} C_{V U_{k}})^{\ast} (V - \mu_{V})
		\xrightarrow[k\to\infty]{\text{a.s.}}
		0.
		\]
	}
	\ref{item:DominatedConvergenceLCMmodified} first note that, by linearity \ref{item:LinearityLCM}, we may assume that $U=0$.
	Further, since $\text{\ref{item:DCTwithSquareIntegrableDominatingFunction}} \implies \text{\ref{item:DCTunderL2Convergence}}$, we may simply assume that \ref{item:DCTunderL2Convergence} holds.
	This implies that
	\begin{equation}
		\label{equ:ProofOfDCTconvergenceOfMeanAndCovariance}
		\mu_{U_{k}} \xrightarrow[k\to\infty]{}0,
		\qquad
		\norm{C_{U_{k}}} \xrightarrow[k\to\infty]{}0,
		\qquad
		C_{U_{k} V}
		=
		C_{U_{k}}^{1/2} R_{U_{k} V} C_{V}^{1/2}
		\xrightarrow[k\to\infty]{}0,
	\end{equation}
	where we used \Cref{theorem:BakerDecompositionCrosscovarianceOperator} and adopted the notation therein (note that $C_{U_{k} V}$ has finite rank, since $C_{V}$ has finite rank by assumption).
	By \Cref{lemma:CovarianceFormulaForLCM}	and \Cref{lemma:L2InnerProductTraceCovariance},
	\[
		\overline{\gamma}_{U_{k}|V}^{\ABLin} C_{V}
		=
		\Cov[\gamma_{U_{k}|V}^{\ABLin} \circ V , V]
		=
		\Cov[\bE^{\ABLin}[U_{k}|V] , V]
		=
		C_{U_{k}V}
		\xrightarrow[k\to\infty]{} 0.
	\]
	Therefore, $\overline{\gamma}_{U_{k}|V}^{\ABLin}\big |_{\ran C_{V}} \xrightarrow[k\to\infty]{} 0$ and, by the assumption of finite rank,
	\[
		V\in \ran C_{V} \text{ a.s.}\ \quad \text{ and } \quad \overline{\gamma}_{U_{k}|V}^{\ABLin} \circ V
	\xrightarrow[k\to\infty]{\text{a.s.}} 0.
	\]
	Denoting the constant part of $\gamma_{U_{k}|V}^{\ABLin}$ by $b_{k}\in\cG$, i.e.\ $\gamma_{U_{k}|V}^{\ABLin}(v) = b_k + \overline{\gamma}_{U_{k}|V}^{\ABLin}(v)$ for $v\in\cH$, the law of total linear expectation \ref{item:LawOfTotalExpectationLCM} implies
	\[
		b_{k} + \overline{\gamma}_{U_{k}|V}^{\ABLin} \mu_{V}
		=
		\bE\bigl[ \bE^{\ABLin}[U_{k}|V] \bigr]
		=
		\mu_{U_{k}}
		\xrightarrow[k\to\infty]{\text{a.s.}} 0.
	\]
	Since $\mu_{V} \in \ran C_{V}$ and $\overline{\gamma}_{U_{k}|V}^{\ABLin}\big |_{\ran C_{V}} \xrightarrow[k\to\infty]{} 0$, we obtain $b_{k} \xrightarrow[k\to\infty]{} 0$ and thereby
	\[
		\bE^{\ABLin}[U_{k}|V]
		=
		b_k + \overline{\gamma}_{U_{k}|V}^{\ABLin} \circ V
		\xrightarrow[k\to\infty]{\text{a.s.}} 0.
	\]
\end{proof}


\begin{proof}[Proof of \Cref{counterexample:CounterexamplesPropertiesLCM}]
	We choose $\cH = \cG = \cF = \bR$ for all counterexamples provided in this proof.
	For counterexamples to \ref{item:MonotonicityLCM}--\ref{item:WrongTowerPropertyLCM}, let $\bP$ to be the uniform distribution on $\Omega = \{ 1,2,3 \}$ and the random variables $V$, $U_{1} \defeq V$ and $U_{2} \defeq W \defeq \absval{V}$ be given by 	
	\begin{center}
		\begin{tabular}{x{3em} x{7em} x{10em} x{6em} x{6em}}
			\toprule
			$\omega \in \Omega$ & $U_{1}(\omega) = V(\omega)$ & $U_{2}(\omega) = W = \absval{V(\omega)}$ &  $\bE^{\ABLin}[U_{1}|V](\omega)$ &  $\bE^{\ABLin}[U_{2}|V](\omega)$
			\\
			\midrule
			$1$ & $-1$ & $1$ & $-1$ & $\nicefrac{2}{3}$ \\
			$2$ & $\phantom{-}0$ & $0$ & $\phantom{-} 0$ & $\nicefrac{2}{3}$ \\
			$3$ & $\phantom{-}1$ & $1$ & $\phantom{-} 1$ & $\nicefrac{2}{3}$ \\
			\bottomrule
		\end{tabular}
	\end{center}
	Clearly, $\bE^{\ABLin}[U_{1}|V] = U_{1} = V$ and, by solving a simple linear regression (or simply by symmetry and \Cref{theorem:BasicPropertiesLCM}\ref{item:LawOfTotalExpectationLCM}), $\bE^{\ABLin}[U_{2}|V] \equiv \nicefrac{2}{3}$, as illustrated in \Cref{fig:CounterexampleMonotonicityTriangleJensenPullOutFactors} (left).	
	Therefore, $U_{2}\ge U_{1}$, but $\bE^{\ABLin}[U_{2}|V] \ngeq \bE^{\ABLin}[U_{1}|V]$, disproving \ref{item:MonotonicityLCM}.	
	Further, $\absval{\bE^{\ABLin}[U_{1}|V]} \le \bE^{\ABLin}[\absval{U_{1}}|V]$ does not hold, providing a counterexample to \ref{item:TriangleInequalityLCM} and \ref{item:JensensInequalityLCM}.	
	For $f\colon \bR\to\bR,\ f(x) = x$, we obtain $f(V)\bE^{\ABLin}[U_{1}|V] = V^{2}$, which clearly cannot equal $\bE^{\ABLin}[f(V)U_{1}|V]$, since it is not an affine transformation of $V$, disproving \ref{item:PullingOutKnownLCM}.	
	Finally,
	$\bE^{\ABLin}\bigl[ \bE^{\ABLin}[U_2|V] \, \big|\, W\bigr]
	=
	\bE^{\ABLin} [ \nicefrac{2}{3} | W ]
	=
	\nicefrac{2}{3}$ a.s.,
	which clearly differs from $\bE^{\ABLin}[U_2|W] = W$ and thereby provides a counterexample to \ref{item:WrongTowerPropertyLCM}.
	
	\begin{figure}[t]
		\centering
		\begin{subfigure}[b]{0.45\textwidth}
			\centering
			\includegraphics[height=0.9\textwidth]{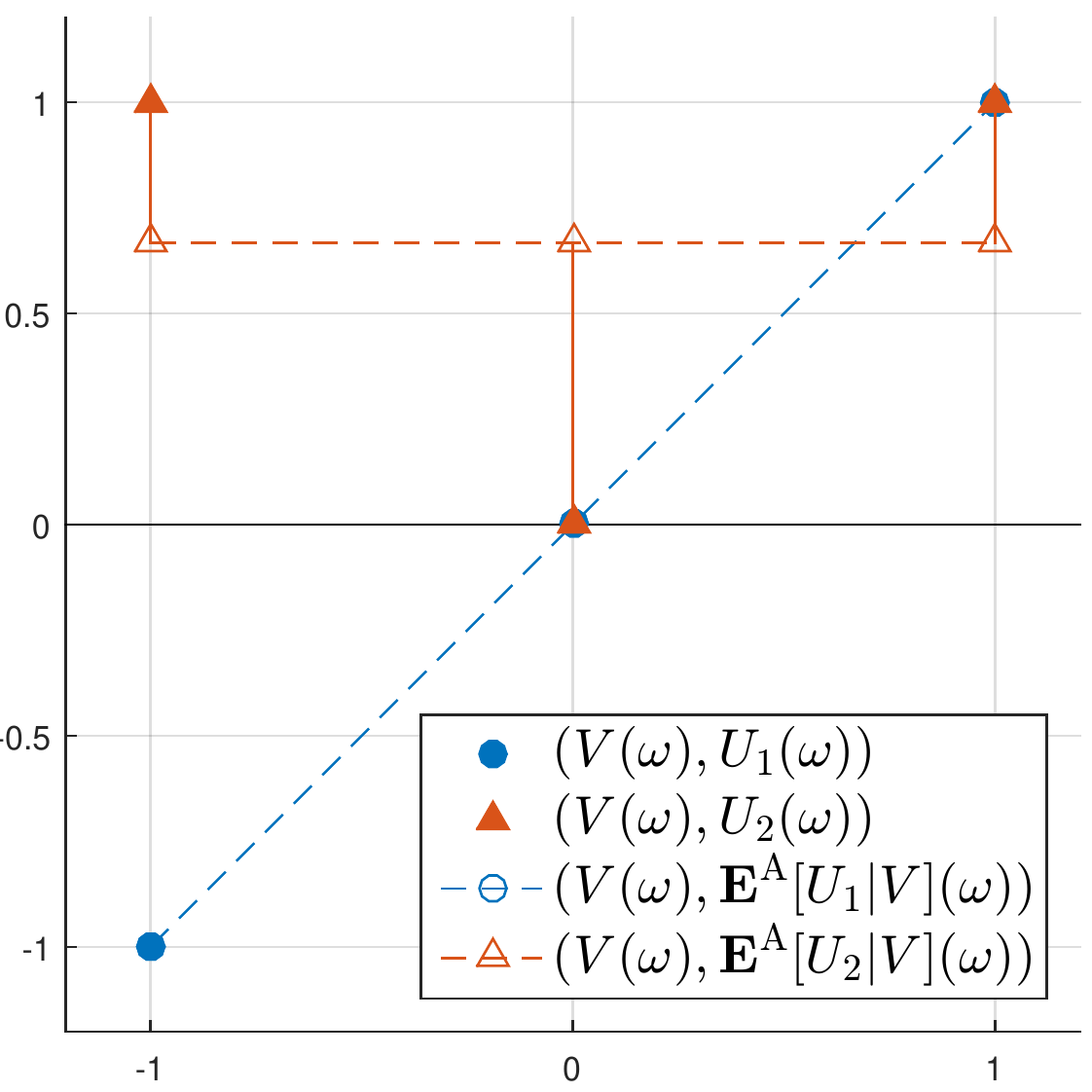}
		\end{subfigure}
		\hspace{1em}
		\begin{subfigure}[b]{0.45\textwidth}
			\centering
			\includegraphics[height=0.9\textwidth]{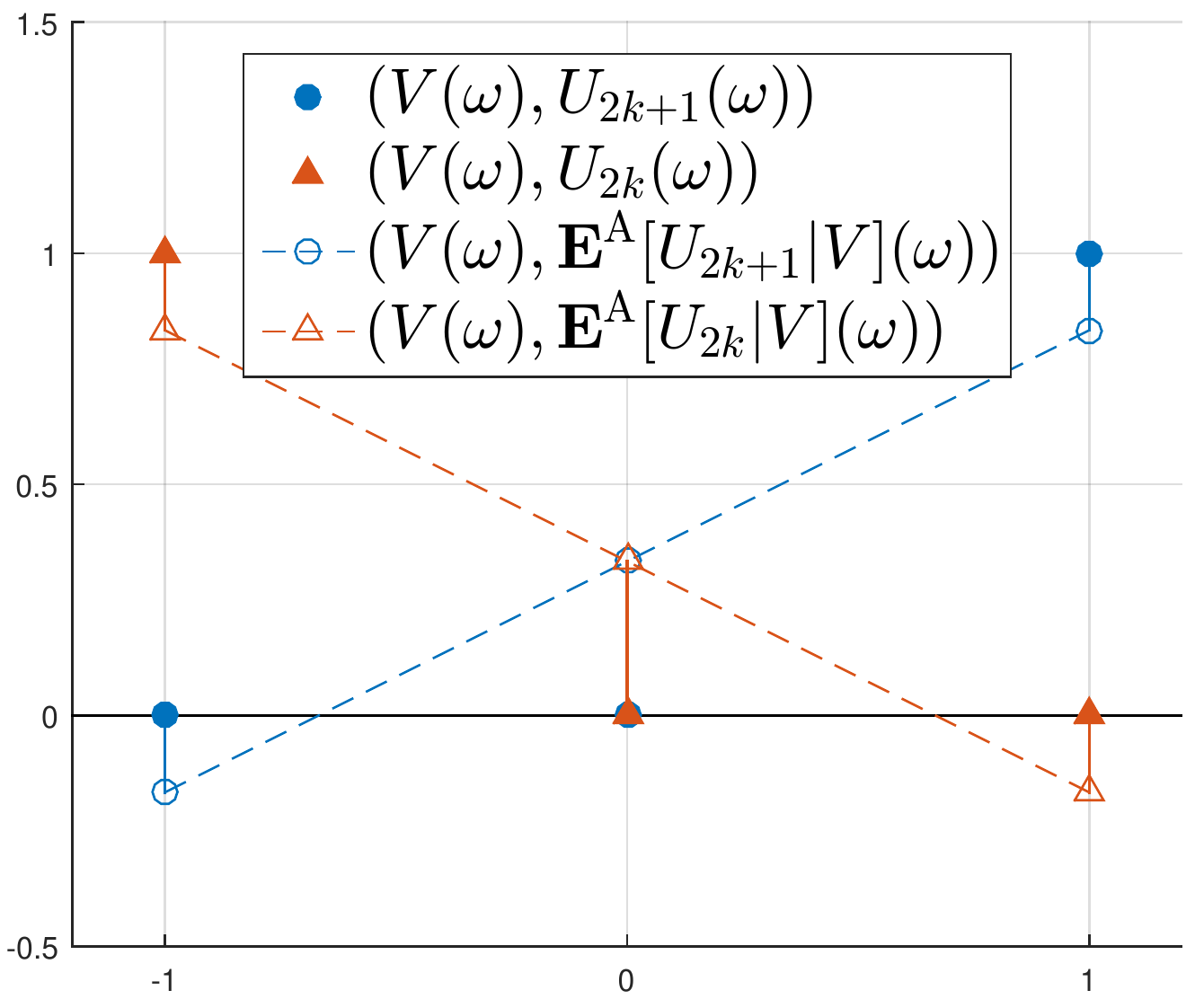}
		\end{subfigure}
		\caption{
			\emph{Left:}
			Counterexample to \Cref{counterexample:CounterexamplesPropertiesLCM}\ref{item:MonotonicityLCM}--\ref{item:WrongTowerPropertyLCM} as described above, with the corresponding LCEs $\bE^{\ABLin}[U_{k}|V]$, $k=1,2$.
			\emph{Right:}
			Counterexample to Fatou's lemma (\Cref{counterexample:CounterexamplesPropertiesLCM}\ref{item:FatouLCM}) with corresponding LCEs $\bE^{\ABLin}[U_{k}|V]$, $k\in\bN$.
		}
		\label{fig:CounterexampleMonotonicityTriangleJensenPullOutFactors}
	\end{figure}

	Since $\bE^{\ABLin}$ lacks monotonicity, a counterexample to \ref{item:FatouLCM} is easy to construct.
	Consider the uniform distribution $\bP$ on $\Omega = \{ 1,2,3 \}$ and $V$ as well as the sequence $(U_k)_{k\in\bN}$ given by
	\begin{center}
		\begin{tabular}{x{3em} x{3em} x{5em} x{4em} x{6em} x{10em}}
			\toprule
			$\omega\in\Omega$ & $V(\omega)$ & $U_{2k+1}(\omega)$ & $U_{2k}(\omega)$ & $\displaystyle \liminf_{k\to\infty} U_{k}(\omega)$ & $\displaystyle\liminf_{k\to\infty} \bE^{\ABLin}[U_{k}|V](\omega)$
			\\
			\midrule
			$1$ & $-1$ & $0$ & $1$ & 0 & $-\nicefrac{1}{6}$ \\
			$2$ & $\phantom{-}0$ & $0$ & $0$ & 0 & $\phantom{-}\nicefrac{2}{6}$ \\
			$3$ & $\phantom{-}1$ & $1$ & $0$ & 0 & $-\nicefrac{1}{6}$ \\
			\bottomrule
		\end{tabular}
	\end{center}
	Then $\bE^{\ABLin}[\liminf_{k\to\infty} U_{k}|V] = \bE^{\ABLin}[0|V] = 0$, while $\liminf_{k\to\infty} \bE^{\ABLin}[U_{k}|V](\omega) < 0$ for $\omega = 1$ and $\omega=3$, which follows from the solution of a simple linear regression problem and is visualised in \Cref{fig:CounterexampleMonotonicityTriangleJensenPullOutFactors} (right).

	Let us now construct a counterexample to the (conventional) dominated convergence theorem \ref{item:DominatedConvergenceLCMconventional}.
	Let $\varepsilon>0$ and $\alpha = (2+2\varepsilon)^{-1}$, e.g.\ $\varepsilon = 1/4$ and $\alpha = 2/5$.
	Let $\bP$ be the uniform distribution on $\Omega = [-1,1]$ and
	\begin{align*}
		V(\omega)
		&=
		\begin{cases}
			(1+\omega)^{-\alpha} - 1 & \text{for } \omega\in[-1,0],
			\\
			-(1-\omega)^{-\alpha} + 1 & \text{for } \omega\in[0,1],
		\end{cases}
		\\[1ex]
		U_{k}(\omega)
		&=
		\begin{cases}
			(1+\omega)^{-2\alpha} - 1 & \text{for } \omega\in[\frac{1}{2k}-1,\frac{1}{k}-1],
			\\
			-(1-\omega)^{-2\alpha} + 1 & \text{for } \omega\in[1-\frac{1}{k},1-\frac{1}{2k}],
			\\
			0 & \text{otherwise,}
		\end{cases}
	\end{align*}
	as illustrated in \Cref{fig:CounterexampleDCT}.
	Clearly, each $U_{k}$ is bounded and thereby lies in $L^{2}(\bP;\cG)$ and $U_{k}\xrightarrow[k\to\infty]{\text{a.s.}} 0$.
	Then, $\bE^{\ABLin}[U_{k}|V] = a_{k} V + b_{k}$ for some $a_{k},b_{k}\in\bR$ where $b_{k} = 0$ for symmetry reasons.
	Let $\beta \defeq 3\alpha - 1$ and note that $\beta > 0$ for sufficiently small $\varepsilon$ ($\beta = 1/5$ in the above example).
	A straightforward computation shows that
	\begin{align*}
		\norm{ a_{k} V - U_{k} }_{L^{2}(\bP;\cG)}^{2}
		&=
		a_{k}^{2} \, \norm{ V }_{L^{2}(\bP;\cG)}^{2}
		-
		2 a_{k} \innerprod{V}{U_{k}}_{L^{2}(\bP;\cG)}
		+
		\norm{ U_{k} }_{L^{2}(\bP;\cG)}^{2}
		\\
		&=
		\frac{2(1+\varepsilon)}{\varepsilon} \, a_{k}^{2}
		-
		\frac{4k^{\beta}}{\beta} \, (2^{\beta}-1) \, a_{k}
		+
		\norm{ U_{k} }_{L^{2}(\bP;\cG)}^{2},
	\end{align*}
	which is minimised by
	$
	a_{k}
	=
	\tfrac{(2^{\beta}-1) \varepsilon}{\beta (1+\varepsilon) } \, k^{\beta}
	\xrightarrow[k\to\infty]{} \infty,
	$
	which contradicts $\bE^{\ABLin}[U_{k}|V] = a_{k} V \xrightarrow[k\to\infty]{\text{a.s.}} 0$.

	\begin{figure}[t]
		\centering
		\begin{subfigure}[b]{0.45\textwidth}
			\centering
			\includegraphics[width=\textwidth]{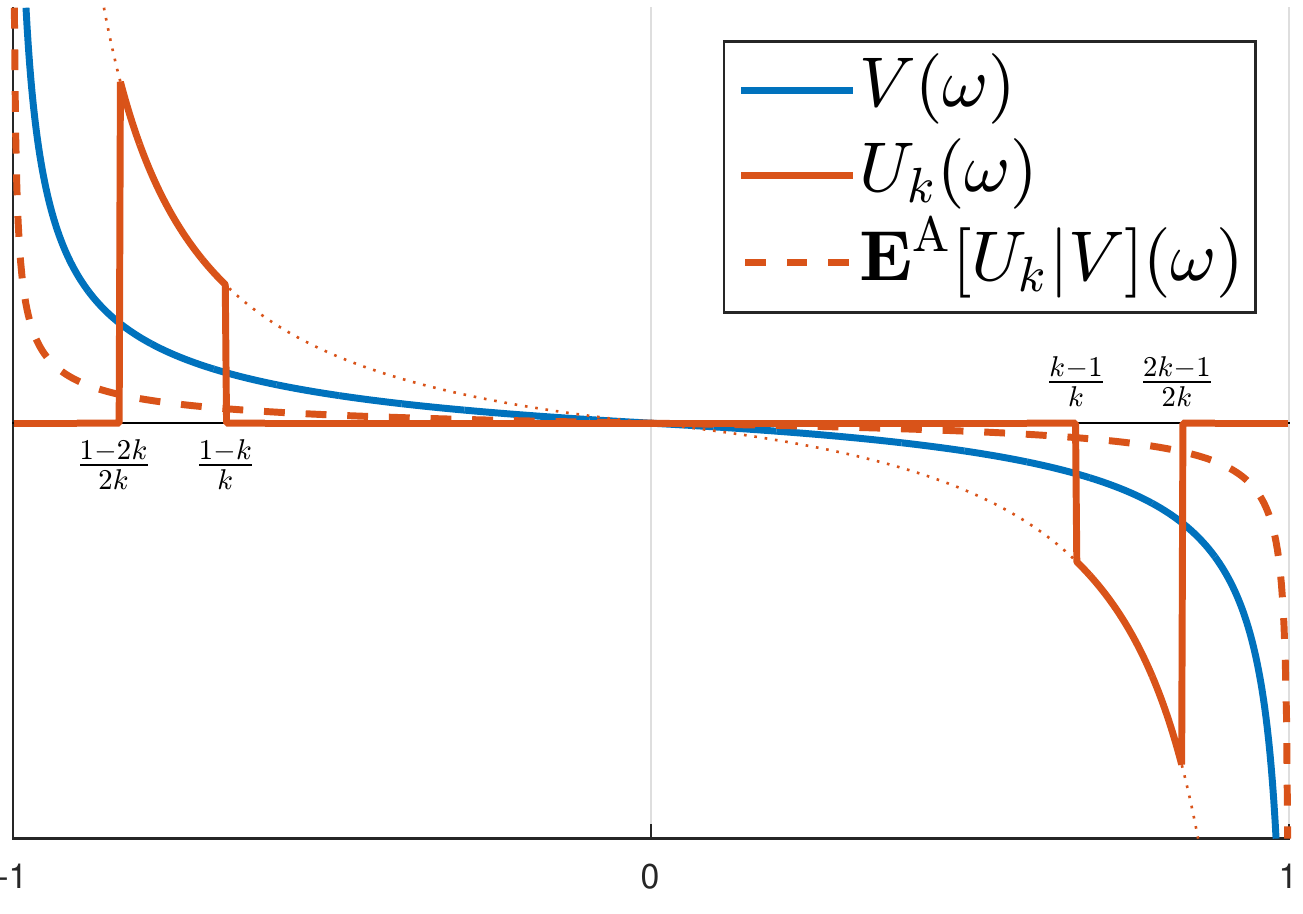}
		\end{subfigure}
		\hspace{1em}
		\begin{subfigure}[b]{0.45\textwidth}
			\centering
			\includegraphics[width=\textwidth]{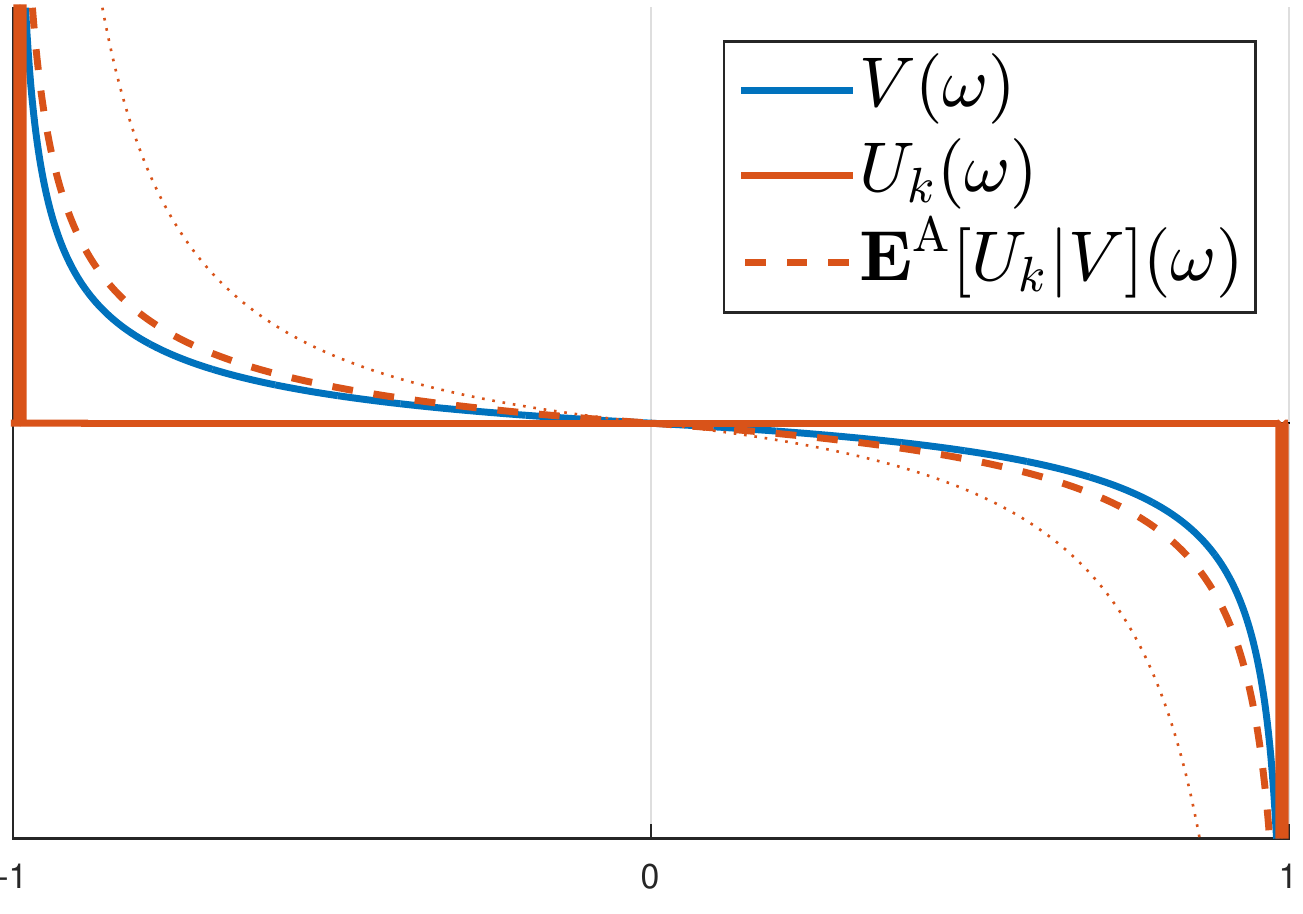}
		\end{subfigure}
		\caption{
			Counterexample to the dominated convergence theorem (\Cref{counterexample:CounterexamplesPropertiesLCM}\ref{item:DominatedConvergenceLCMconventional}).
			Note that we plot $\bE^{\ABLin}[U_{k}|V]$ as a function of $\omega$ (and not of $V$) which is why it not a linear function in contrast to the other plots, while being a multiple of $V$ by the factor $\alpha_{k}$.
			For sufficiently small $\varepsilon>0$, this factor $\alpha_{k}$ increases with $k$ (here $k=3$ (left), $k=70$ (right) and $\varepsilon=0.01$) as can be seen from the dotted red lines in the two plots.
			Therefore, in contrast to $(U_{k})_{k\in\bN}$, the sequence of LCEs $(\bE^{\ABLin}[U_{k}|V])_{k\in\bN}$ does not converge to zero a.s.
		}
		\label{fig:CounterexampleDCT}
	\end{figure}

	The following two counterexamples disprove \ref{item:ContractivenessLCM} for every $p<2$ and for every $p>2$, respectively (for $p=2$ the projection is clearly contractive, since it is orthogonal).
	Choose $\bP$ to be the uniform distribution on $\Omega = \{ 1,2,3,4 \}$ and the random variables $V$, $U_{1}$ and $U_{2}$ in the following way, where $\varepsilon \in (0,1)$ is a free parameter yet to be chosen.
	\begin{center}
		\begin{tabular}{x{3em} x{3em} x{3em} x{3em}}
			\toprule
			$\omega\in\Omega$ & $V(\omega)$ & $U_{1}(\omega)$ & $U_{2}(\omega)$
			\\
			\midrule
			$1$ & $-1$ & $-1$ & $-1$ \\
			$2$ & $-\varepsilon$ & $\phantom{-}0$ & $-2 \varepsilon$ \\
			$3$ & $\phantom{-}\varepsilon$ & $\phantom{-}0$ & $\phantom{-}2 \varepsilon$ \\
			$4$ & $\phantom{-}1$ & $\phantom{-}1$ & $\phantom{-}1$ \\
			\bottomrule
		\end{tabular}
	\end{center}
	
	\begin{figure}[t]
		\centering
		\begin{subfigure}[b]{0.4\textwidth}
			\centering
			\includegraphics[width=\textwidth]{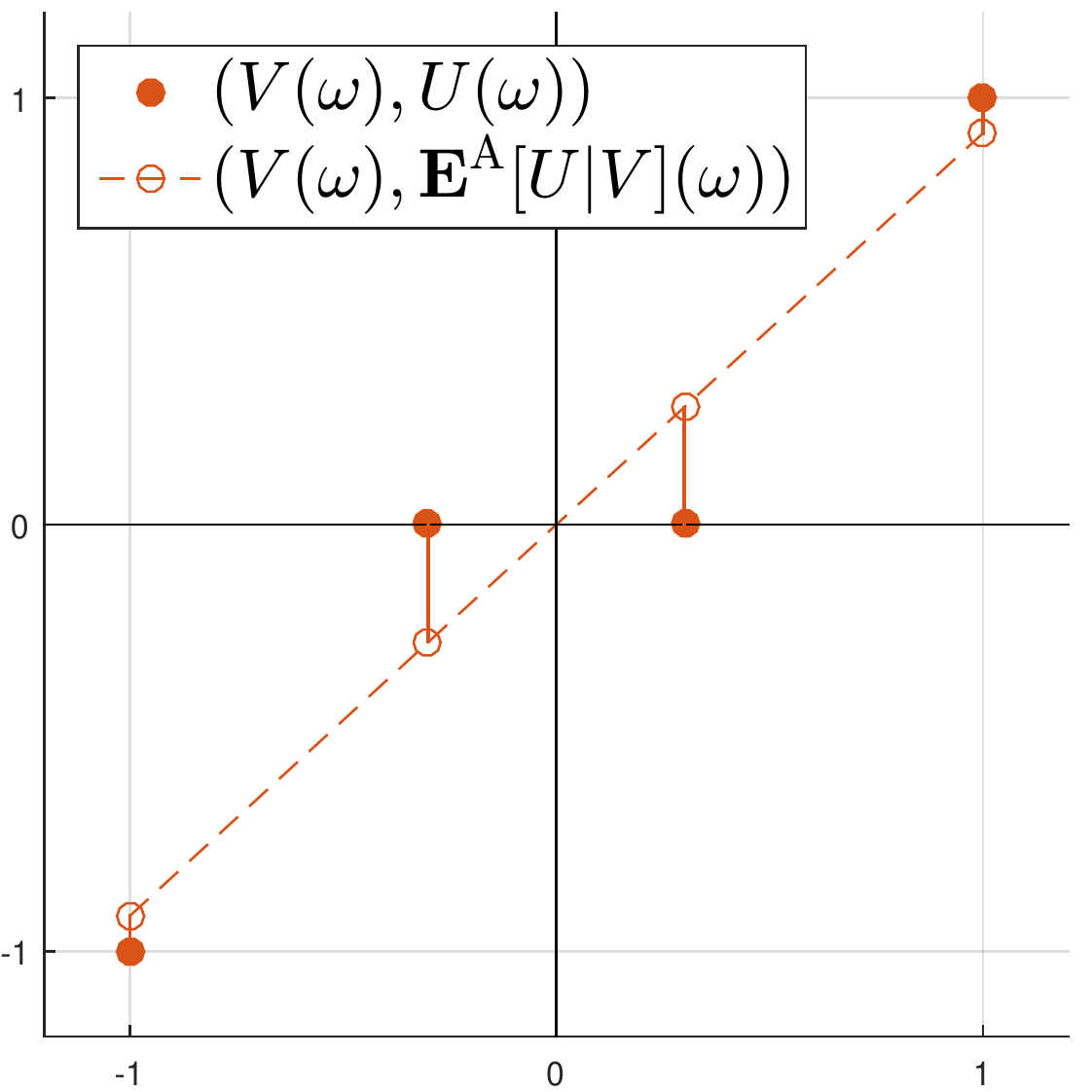}
		\end{subfigure}
		\hspace{1em}
		\begin{subfigure}[b]{0.4\textwidth}
			\centering
			\includegraphics[width=\textwidth]{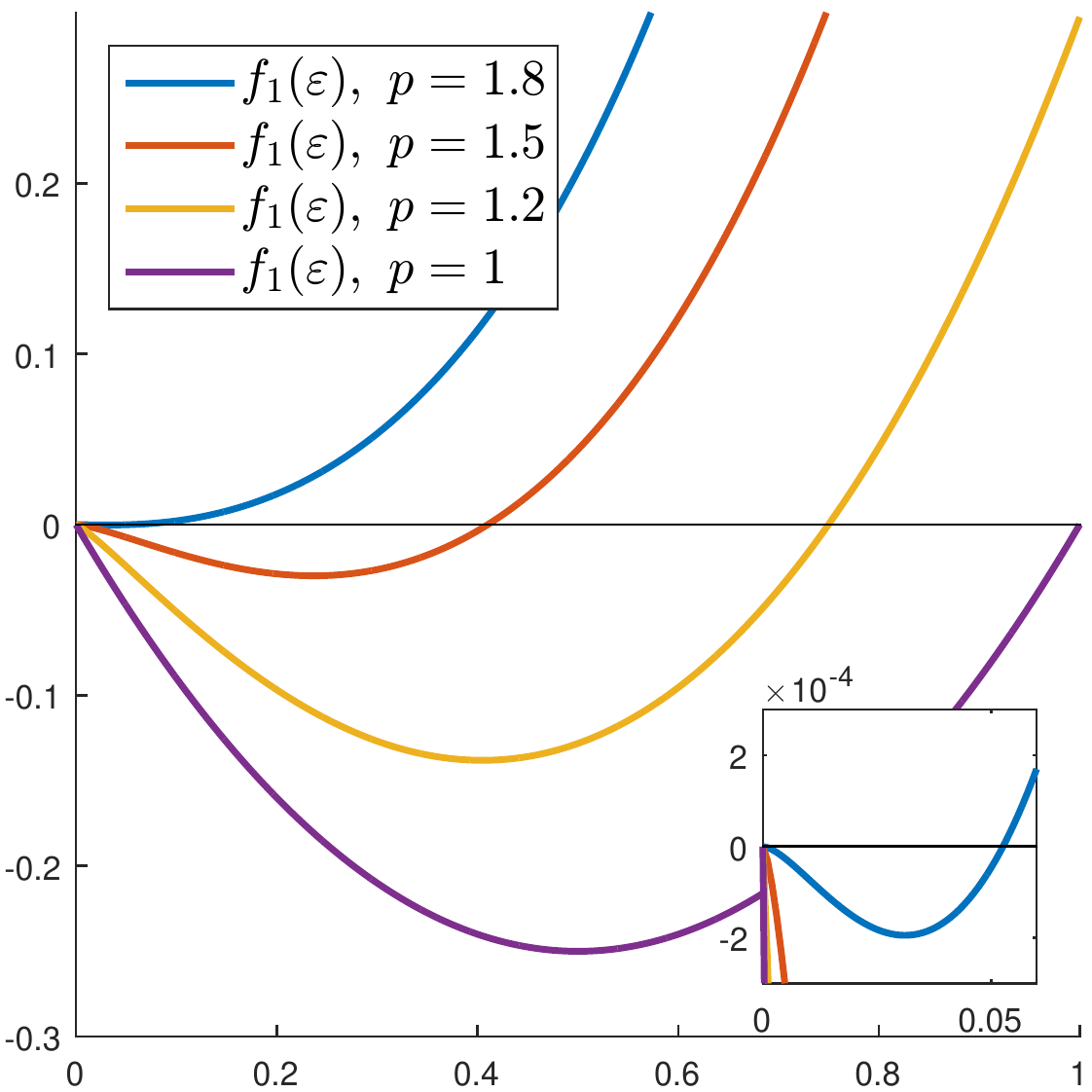}
		\end{subfigure}
		\vfill
		\begin{subfigure}[b]{0.4\textwidth}
			\centering
			\includegraphics[width=\textwidth]{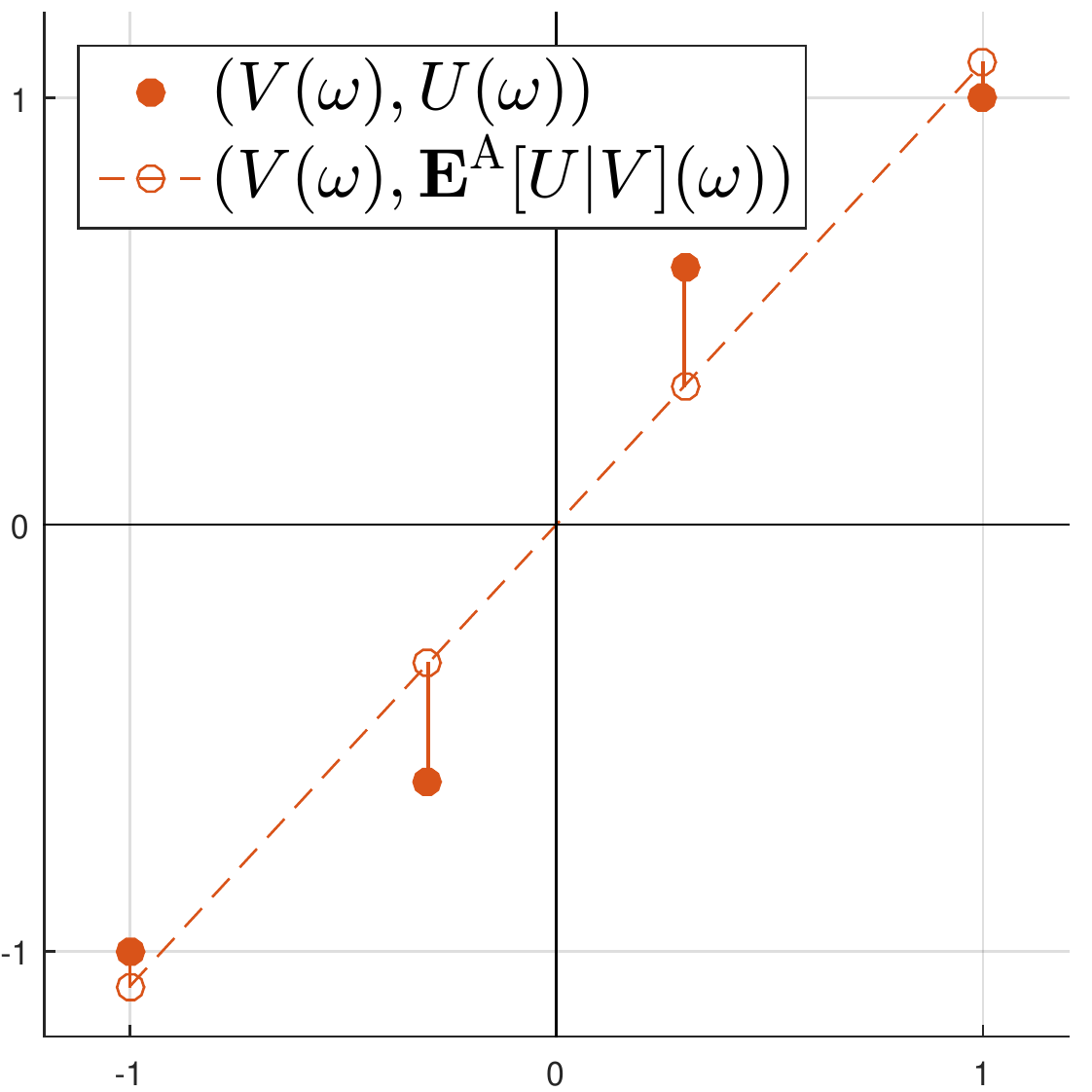}
		\end{subfigure}
		\hspace{1em}
		\begin{subfigure}[b]{0.4\textwidth}
			\centering
			\includegraphics[width=\textwidth]{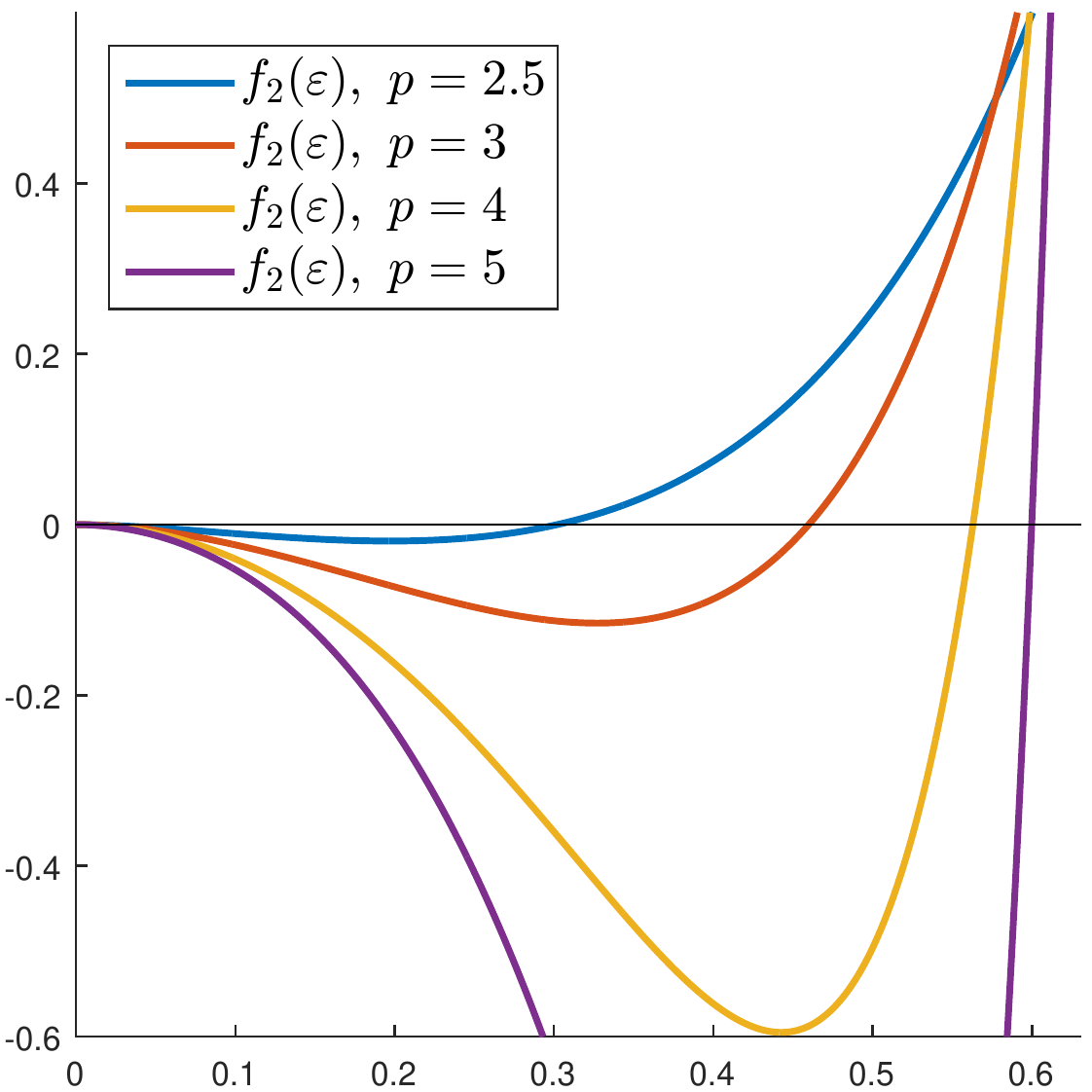}
		\end{subfigure}
		\caption{
			\emph{Top:} Counterexample to \Cref{counterexample:CounterexamplesPropertiesLCM}\ref{item:ContractivenessLCM} for $1\leq p<2$.
			\emph{Bottom:} Counterexample to \Cref{counterexample:CounterexamplesPropertiesLCM}\ref{item:ContractivenessLCM} for $p > 2$.
			\emph{Left:} LCEs for the above examples and $\varepsilon=0.3$.
			\emph{Right:} The functions $f_{1}(\varepsilon)$ and $f_{2}(\varepsilon)$ for several values of $p$. For every $p<2$ (respectively $p>2$) there exists a sufficiently small $\varepsilon > 0$ such that $f_j(\varepsilon)<0$, $j=1,2$.
		}
		\label{fig:NewCounterexampleContractiveness}
	\end{figure}
	
	Again, the computation of $\bE^{\ABLin}[U_{j}|V] = a_{j}V + b_{j}$, $a_{j},b_{j}\in\bR$, $j=1,2$, reduces to a linear regression that is solved by
	\[
	b_1 = b_2 = 0,
	\qquad
	a_{1} = \tfrac{1}{1 + \varepsilon^{2}},
	\qquad
	a_{2} = \tfrac{1 + 2\varepsilon^{2}}{1 + \varepsilon^{2}}.
	\]
	Note that $\bE[\absval{U_{1}}^{p}] = \tfrac{1}{2}$, $\bE[\absval{U_{2}}^{p}] = \tfrac{1}{2}(1+(2\varepsilon)^{p})$ and $\bE \bigl[ \absval{\bE^{\ABLin}[U_{j}|V]}^{p} \bigr] = \tfrac{1}{2}a_{j}^{p}(1+\varepsilon^{p})$, $j=1,2$.
	It follows that the inequality in \ref{item:ContractivenessLCM} for $U=U_{j}$, $j=1,2$, holds whenever
	\begin{align}
	\label{equ:ContractivenessFunction1}
	f_{1}(\varepsilon)
	& \defeq
	(1 + \varepsilon^{2})^{p} - (1 + \varepsilon^{p})
	\ge
	0,
	\\
	\label{equ:ContractivenessFunction2}
	f_{2}(\varepsilon)
	& \defeq
	(1+\varepsilon^{2})^{p} (1 + 2^{p} \varepsilon^{p}) - (1+2\varepsilon^{2})^{p} (1 + \varepsilon^{p})
	\ge
	0,
	\end{align}
	respectively.
	Bernoulli's inequality $(1+x)^{r} \leq 1+rx$ for $x\geq -1$ and exponents $0\leq r \leq 1$ implies that, for $1\leq p\leq 2$,
	\begin{align*}
	f_{1}(\varepsilon)
	&=
	(1 + \varepsilon^{2}) (1 + \varepsilon^{2})^{p-1} - (1 + \varepsilon^{p})
	\\
	&\leq
	(1 + \varepsilon^{2}) \bigl(1 + (p-1)\varepsilon^{2}\bigr) - (1 + \varepsilon^{p})
	\\
	&=
	p \varepsilon^{2} + (p-1)\varepsilon^{4} - \varepsilon^{p}.
	\end{align*}
	Since, for any $p<2$, $p \varepsilon^{2} + (p-1)\varepsilon^{4} < \varepsilon^{p}$ for sufficiently small $\varepsilon$, we can falsify \eqref{equ:ContractivenessFunction1} and thereby disprove \ref{item:ContractivenessLCM} for any $p<2$.
	
	For fixed $p>2$ consider the Taylor polynomial of degree $2$ for $f_{2}$, namely $T_{2}f_{2}(\varepsilon) = -p\varepsilon^{2}$;
	note that this is \emph{not} a Taylor polynomial of $f_2$ for $p<2$.
	Hence, \eqref{equ:ContractivenessFunction2} cannot hold for sufficiently small $\varepsilon$, providing a counterexample to \ref{item:ContractivenessLCM} for any $p>2$.
	
	\Cref{fig:NewCounterexampleContractiveness} illustrates the counterexamples for $\varepsilon = 0.3$ (left) as well as the functions $f_{1}(\varepsilon)$ and $f_{2}(\varepsilon)$ for several values of $p$.

	We now give a counterexample to \ref{item:NonNegativityLCC}.
	Let $\bP$ be the uniform distribution on $\Omega = \{ 1,\dots,6 \}$ and $V$ and $U$ given by
	\begin{center}
		\begin{tabular}{x{3em} x{3em} x{3em} x{7em} x{7em} x{7em}}
			\toprule
			$\omega\in\Omega$ & $V(\omega)$ & $U(\omega)$ & $\bE^{\ABLin}[U|V](\omega)$ & $R^{\ABLin}[U|V]^{2}(\omega)$ & $\Cov^{\ABLin}[U|V](\omega)$
			\\
			\midrule
			$1$ & $-1$ 				& $\phantom{-}1$ 	&
			\multirow{2}{*}{$\hspace{-2em}\left.\begin{array}{l} \\ \\ \end{array}\right\rbrace \hspace{1em} 0$} &
			\multirow{2}{*}{1} 		& \multirow{2}{*}{$\tfrac{1}{6}(7-N^{2})$} \\
			$2$ & $-1$ 				& $-1$ 			\\
			$3$ & $\phantom{-}0$ 	& $\phantom{-}1$ 	&
			\multirow{2}{*}{$\hspace{-2em}\left.\begin{array}{l} \\ \\ \end{array}\right\rbrace \hspace{1em} 0$} &
			\multirow{2}{*}{1} 		& \multirow{2}{*}{$\tfrac{1}{3}(2+N^{2})$} \\
			$4$ & $\phantom{-}0$ 	& $-1$ 			\\
			$5$ & $\phantom{-}1$ 	& $\phantom{-}N$ 	&
			\multirow{2}{*}{$\hspace{-2em}\left.\begin{array}{l} \\ \\ \end{array}\right\rbrace \hspace{1em} 0$} &
			\multirow{2}{*}{$N^{2}$}& \multirow{2}{*}{$\tfrac{1}{6}(1+5N^{2})$} \\
			$6$ & $\phantom{-}1$ 	& $-N$ 			\\
			\bottomrule
		\end{tabular}
	\end{center}
	where $N>0$. By symmetry, $\bE^{\ABLin}[U|V] = \bE[U|V] = 0$ while $\Cov^{\ABLin}[U|V] = \tfrac{1}{2}(N^{2}-1) V + \tfrac{1}{3}(N^{2}+2)$, which follows from the solution of a simple linear regression problem and is visualised in \Cref{fig:NegativeConditionalCovariance}.
	If $N>0$ is sufficiently large, $\Cov^{\ABLin}[U|V](\omega)$ clearly takes on negative values for $\omega = 1,2$.
	\begin{figure}[t]
		\centering
		\includegraphics[width=0.45\textwidth]{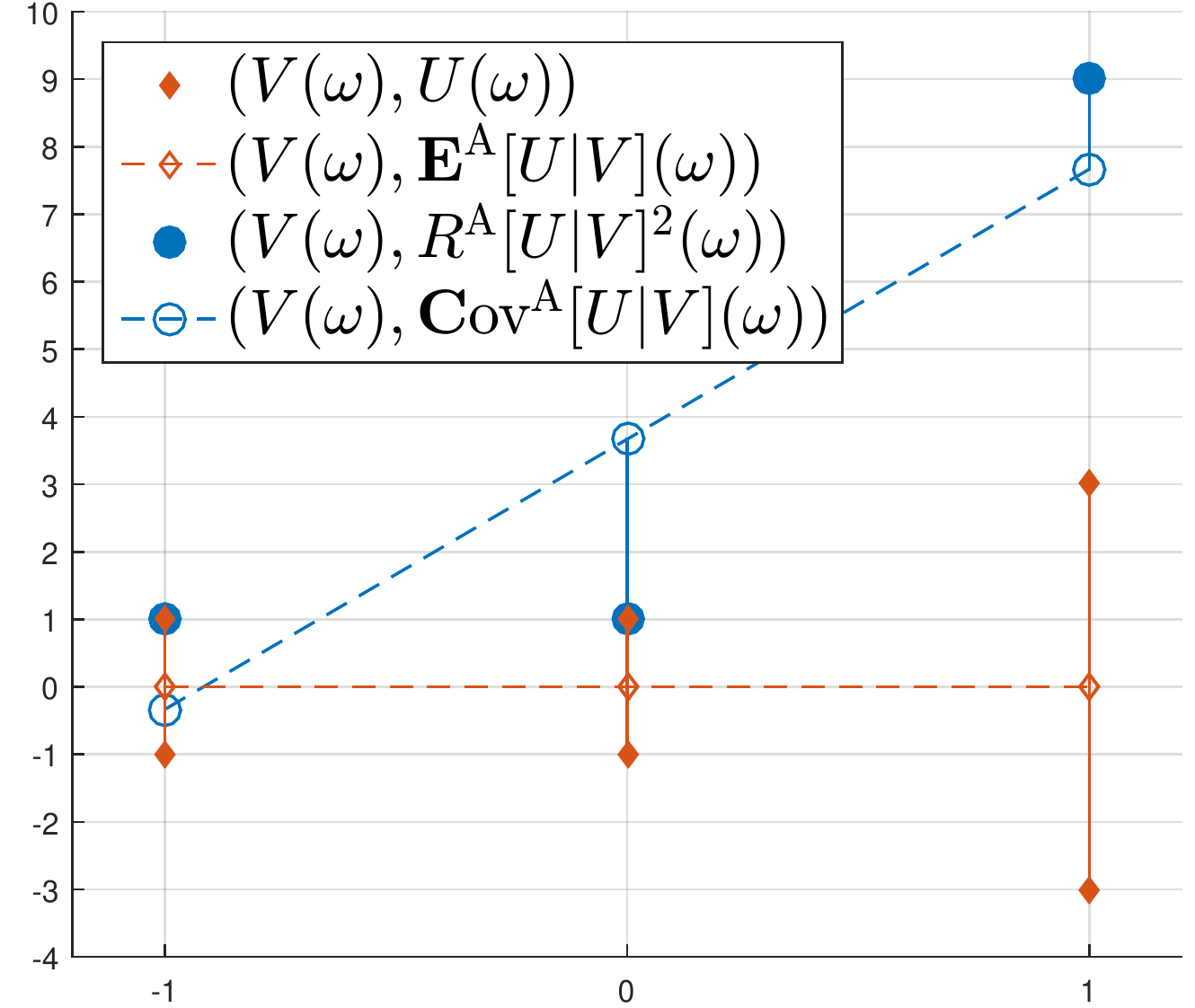}
		\caption{In contrast to the conditional covariance $\Cov[U|V]$, the LCC $\Cov^{\ABLin}[U|V]$ can take on negative values, while its expected value $\Cov_{V}^{\ABLin}[U]$ is guaranteed to be non-negative.}
		\label{fig:NegativeConditionalCovariance}
	\end{figure}
\end{proof}


\begin{proof}[Proof of \Cref{theorem:LinearConditionalMeanUnderRangeAssumption}]
	First note that, by \Cref{theorem:DouglasExistenceQ}, $C_V^\dagger C_{VU}\in\BLin(\cG;\cH)$ is well-defined and bounded and that $C_V (C_V^\dagger C_{VU}) = C_{VU}$, which implies $\overline{\gamma}_{U|V}^{\ABLin} C_{V} = (C_V^\dagger C_{VU})^\ast C_V = C_{UV}$.
	We have to show that $U - \gamma_{U|V}^{\ABLin}\circ V$ is $L^2(\bP;\cG)$-perpendicular to $\gamma\circ V$ for any other $\gamma\in \ABLin(\cH;\cG)$.
	Since $\bE[\gamma_{U|V}^{\ABLin}\circ V] = \mu_{U} = \bE[U]$, it follows by \Cref{lemma:L2InnerProductTraceCovariance} that
	\[
		\innerprod{ U - \gamma_{U|V}^{\ABLin}\circ V }{ \gamma \circ V }_{L^2(\bP;\cG)}
		=
		\trace \Bigl( \Cov \bigl[ U - \gamma_{U|V}^{\ABLin}\circ V , V \bigr]\, \overline{\gamma}^{\ast} \Bigr)
		=
		\trace \Bigl( \bigl( C_{UV} - \overline{\gamma}_{U|V}^{\ABLin}\, C_{V} \bigr)\, \overline{\gamma}^{\ast} \Bigr)
		=
		0 ,
	\]
	as required.
\end{proof}


\begin{proof}[Proof of \Cref{lemma:MartingaleLemmaEqualityOfSpaces}]
	The Karhunen--Lo\`eve expansion of $V$ takes the form
	\[
	V = \mu_V + \sum_{i\in\bN} Z_{i}\, h_{i},
	\]
	where the $Z_{i}$ are uncorrelated real-valued random variables over $(\Omega,\Sigma,\bP)$ with $\bE[Z_{i}]=0$ and $\bV[Z_{i}]=\sigma_{i}^{2}$ for all $i\in\bN$.
	Observe that $\sigma(V^{(n)}) = \sigma(Z^{(n)})$, where $Z^{(n)}\defeq (Z_1,\dots,Z_n)$.
	
	Now let $W \in \overline{\ABLin_{\cG}\circ V} \cap L^{2}(\Omega,\sigma(V^{(n)}) ; \cG)$.
	Since $W\in \overline{\ABLin_{\cG}\circ V}$, there exists a sequence $(\gamma_k)_{k\in\bN}$ in $\ABLin(\cH ; \cG)$ such that $\norm{\gamma_k \circ V - W}_{ L^2(\bP;\cG)} \xrightarrow[k\to\infty]{}0$.
	In order to show that $W \in \overline{\ABLin_{\cG}\circ V^{(n)}}$, we will find a sequence $(\tilde \gamma_k)_{k\in\bN}$ in $\ABLin(\cH ; \cG)$ such that $\norm{\tilde\gamma_k \circ V^{(n)} - W}_{ L^2(\bP;\cG)} \xrightarrow[k\to\infty]{}0$.
	To this end, we shift each $\gamma_k$ by a constant, choosing $\tilde \gamma_{k}(v) \defeq \gamma_{k}(v) + \gamma_{k}(\mu_{V} - \mu_{V}^{(n)})$, where $\mu_{V}^{(n)} \defeq \bE[V^{(n)}] = P_{\cH^{(n)}}\mu_{V}$.
	In order to prove above convergence observe that $\bE[W|V^{(n)}] = W$, since $W$ is $\sigma(V^{(n)})$-measurable, and that
	\[
	\bE[\gamma_k(V - \mu_{V}) | V^{(n)}]
	=
	\gamma_k \biggl( \bE \biggl[ \sum_{i\in\bN} Z_{i}\, h_{i} \bigg| Z^{(n)} \biggr] \biggr)
	=
	\gamma_k \biggl( \sum_{i=1}^{n} Z_{i}\, h_{i} \biggr)
	=
	\gamma_k \bigl( V^{(n)} - \mu_{V}^{(n)} \bigr),
	\]
	as the random variables $Z_{i}$ are uncorrelated.
	Since conditional expectations are $L^2(\bP;\cG)$-contractive projections, it follows that
	\begin{align*}
	\norm{\tilde\gamma_k \circ V^{(n)} - W}_{ L^2(\bP;\cG)}
	&=
	\bignorm{\gamma_k \bigl(V^{(n)} - \mu_{V}^{(n)}\bigr) - W + \gamma_k (\mu_{V})}_{ L^2(\bP;\cG)}
	\\
	&=
	\bignorm{\bE\bigl[\gamma_k(V - \mu_{V}) \big| V^{(n)}\bigr] - \bE\bigl[W \big| V^{(n)}\bigr] + \gamma_k (\mu_{V})}_{ L^2(\bP;\cG)}
	\\
	&=
	\bignorm{\bE[\gamma_k(V) - W | V^{(n)}]}_{L^2(\bP;\cG)}
	\\
	&\le
	\norm{\gamma_k \circ V - W }_{ L^2(\bP;\cG)}
	\\
	&\xrightarrow[k\to\infty]{}
	0.
	\end{align*}
	The second inclusion in \eqref{equ:TechnicalIdentityOfSpaces} is trivial.	
\end{proof}

\begin{proof}[Proof of \Cref{thm:MartingaleProperty}]
	Claim \ref{item:MartingaleProperty} follows from \Cref{lemma:MartingaleLemmaEqualityOfSpaces} via
	\[
	\bE\bigl[ \bE^{\ABLin}[U|V] \big| V^{(n)} \bigr]
	=
	P_{ L^{2}(\Omega,\sigma(V^{(n)}) ; \cG)} \, P_{\overline{\ABLin_{\cG}\circ V}}\, U
	=
	P_{\overline{\ABLin_{\cG}\circ V^{(n)}}}\, U
	=
	\bE^{\ABLin}[U|V^{(n)}],
	\]
	where all orthogonal projections are taken with respect to the $ L^{2}(\Omega,\Sigma,\bP ; \cG)$ inner product.

	Since $\bE^{\ABLin}[U|V] = \bE\bigl[ \bE^{\ABLin}[U|V] \big| V \bigr]$ by \Cref{theorem:BasicPropertiesLCM}\ref{item:ConditionalExpectationLCM} and using \eqref{equ:MartingaleProperty}, claim \ref{item:MartingaleConvergenceTheorem} follows directly from \citet[Theorems~2 and 6]{chatterji1960martingales} or \citet[Theorems~11.7 and 11.10]{klenke2013wahrscheinlichkeitstheorie}.
\end{proof}


\begin{proof}[Proof of \Cref{theorem:LinearConditionalMeanIncompatibleCase}]
	Since $C_{V}^{(n)}$ has finite rank, \Cref{theorem:BakerDecompositionCrosscovarianceOperator} yields $\ran C_{VU}^{(n)}\subseteq \ran C_{V}^{(n)}$ and \Cref{theorem:LinearConditionalMeanUnderRangeAssumption} implies
	\[
	\bE^{\ABLin}[U | V^{(n)}]
	=
	P_{\overline{\ABLin_{\cG}\circ V^{(n)}}} U
	=
	\gamma_{U|V}^{(n)} \circ V^{(n)} = \gamma_{U|V}^{(n)} \circ V.
	\]
	The statements follow directly from \Cref{thm:MartingaleProperty}.
\end{proof}


\begin{proof}[Proof of \Cref{theorem:LinearConditionalMeanRegularized}]
	First note that $\gamma_{\varepsilon}^{\AHS}\in \AHS (\cH;\cG)$, since it is a shifted composition of the Hilbert--Schmidt operator $C_{UV}$ and the bounded operator $(C_{V} + \varepsilon \Id_{\cH})^{-1}$.
	Now let $\gamma \in \AHS (\cH;\cG)$ and $\delta \defeq \gamma - \gamma_{\varepsilon}^{\AHS}\in \AHS (\cH;\cG)$.
	Since $\bE[\gamma_{\varepsilon}^{\AHS} \circ V] = \mu_{U}$, \Cref{lemma:L2InnerProductTraceCovariance} implies that
	\begin{align*}
		\bE[ \innerprod{U- \gamma_{\varepsilon}^{\AHS} \circ V}{\delta \circ V}_{\cG} ]
		&=
		\trace ( C_{UV}\overline{\delta}^{\ast} - \overline{\gamma}_{\varepsilon}^{\AHS} C_V \overline{\delta}^{\ast} )
		\\
		&=
		\trace \bigl( C_{UV} (C_{V} + \varepsilon \Id_{\cH})^{-1} ( C_{V} + \varepsilon \Id_{\cH} - C_{V} ) \overline{\delta}^{\ast} \bigr)
		\\
		&=
		\varepsilon \trace \bigl( \overline{\gamma}_{\varepsilon}^{\AHS}\overline{\delta}^{\ast} \bigr)
		\\
		&=
		\varepsilon\, \innerprod{\gamma_{\varepsilon}^{\AHS}}{\delta}_{\HS}.
	\end{align*}
	Hence,
	\begin{align*}
		\cE_{U|V}^{\textup{reg}}(\gamma)
		&=
		\cE_{U|V}^{\textup{reg}}(\gamma_{\varepsilon}^{\AHS}) +
		\bE[ \norm{\delta(V)}_{\cG}^{2} ] +
		\varepsilon \norm{\delta}_{\HS}^{2}
		\underbrace{
			- 2 \, \bE[ \innerprod{U-\gamma_{\varepsilon}^{\AHS}(V)}{\delta(V)}_{\cG} ]
			+
			2\, \varepsilon\, \innerprod{\gamma_{\varepsilon}^{\AHS}}{\delta}_{\HS}
		}_{=\, 0}
		\\
		&\ge
		\cE_{U|V}^{\textup{reg}}(\gamma_{\varepsilon}^{\AHS}),
	\end{align*}
	proving the claim.
\end{proof}


\begin{proof}[Proof of \Cref{thm:PropertiesLinearConditionalCovariance}]
	By the law of total linear expectation in 	\Cref{theorem:BasicPropertiesLCM}\ref{item:LawOfTotalExpectationLCM},
	\begin{align*}
	\bE\bigl[\Cov^{\ABLin}[U,W|V]\bigr]
	&=
	\bE\Bigl[ \bE^{\ABLin} \bigl[ R^{\ABLin}[U|V] \otimes R^{\ABLin}[W|V] \, \big| \, V\bigr] \Bigr]
	\\
	&=
	\bE\bigl[R^{\ABLin}[U|V] \otimes R^{\ABLin}[W|V]\bigr]
	\\
	&=
	\Cov^{\ABLin}_{V}[U,W],
	\end{align*}
	proving \ref{item:AverageConditionalCovariance}.
	By the law of total covariance and its linear version in \Cref{theorem:BasicPropertiesLCM}\ref{item:LawOfTotalCovarianceLCM}, we obtain
	\[
	\bE\bigl[\Cov[U|V]\bigr]
	=
	\Cov[U] - \Cov\bigl[\bE[U|V]\bigr]
	\le
	\Cov[U] - \Cov\bigl[\bE^{\ABLin}[U|V]\bigr]
	=
	\bE\bigl[\Cov^{\ABLin}[U|V]\bigr],
	\]
	proving \ref{item:UpperBoundExpectedConditionalCovariance} (the equality in \ref{item:UpperBoundExpectedConditionalCovariance} follows directly from \ref{item:AverageConditionalCovariance}).
	In order to prove \ref{item:GaussianFormulaForConditionalCovariance}, first note that, by \Cref{theorem:BakerDecompositionCrosscovarianceOperator} and using \Cref{notation:NotationForIncompatibleLCM},
	\[
	C_{V}^{1/2} (\overline{\gamma}_{U|V}^{(n)})^{\ast}
	=
	C_{V}^{1/2} C_{V}^{(n)\dagger} C_{VU}^{(n)}
	=
	P_{\cH^{(n)}} R_{VU} C_{U}^{1/2}
	\xrightarrow[n\to\infty]{}
	R_{VU} C_{U}^{1/2}
	=
	M_{VU}.
	\]
	Hence, by \eqref{equ:L2ConvergenceToLCE} and \Cref{lemma:L2InnerProductTraceCovariance},
	\begin{align*}
	\Cov\bigl[ \bE^{\ABLin} [U|V] , \bE^{\ABLin} [W|V] \bigr]
	&=
	\lim_{n\to\infty}
	\Cov\bigl[ \gamma_{U|V}^{(n)} \circ V , \gamma_{W|V}^{(n)} \circ V \bigr]
	\\
	&=
	\lim_{n\to\infty}
	\overline{\gamma}_{U|V}^{(n)} \, C_{V} \, (\overline{\gamma}_{W|V}^{(n)})^{\ast}
	\\
	&=
	\lim_{n\to\infty}
	\bigl( C_{V}^{1/2} (\overline{\gamma}_{U|V}^{(n)})^{\ast}\bigr)^{\ast}\,
	\bigl( C_{V}^{1/2} (\overline{\gamma}_{W|V}^{(n)})^{\ast} \bigr)
	\\
	&=
	M_{VU}^{\ast} M_{VW}.
	\end{align*}
	By \ref{item:AverageConditionalCovariance} and the law of total linear covariance in \Cref{theorem:BasicPropertiesLCM}\ref{item:LawOfTotalCovarianceLCM}, we obtain
	\[
	\Cov^{\ABLin}_{V}[U,W]
	=
	\Cov[U,W] - \Cov\bigl[ \bE^{\ABLin} [U|V] , \bE^{\ABLin} [W|V] \bigr]
	=
	C_{UW} - M_{VU}^{\ast} M_{VW},
	\]
	thus completing the proof.
\end{proof}


\begin{proof}[Proof of \Cref{corollary:FormulaForLCC}]
	Noting that $\mu_{Z} = \Cov^{\ABLin}_{V}[U,W]$, the claim follows directly from \Cref{theorem:LinearConditionalMeanUnderRangeAssumption,theorem:LinearConditionalMeanIncompatibleCase,thm:PropertiesLinearConditionalCovariance}.
\end{proof}


\begin{proof}[Proof of \Cref{proposition:AssumptionHierarchyCMEimpliesOldVersion}]	
	In this proof $\fm$ will be viewed as an element of $\HS(\cG;L^2(\bP_{X}))$, which is isometrically isomorphic to $\cG \otimes L^2(\bP_{X}) \cong L^2(\bP_{X};\cG)$ (see \Cref{remark:ViewTensorProductAsHS}).
	In this case $f_g = \fm(g)$ by \eqref{equ:SomeTensorProductIdentities}.
	
	If $\fm \in \HS(\cG;\cH)$, then clearly $f_g = \fm(g) \in \cH$ and this shows that \ref{assump:strongCME} $\implies$ \ref{assump:strongCMEold}.
	
	Now let $[\fm] \in (\cG\otimes \cH)_{\cC}$.
	Then there exist $\fh\in\cG\otimes \cH$ and $c\in\cG$ such that $\fh(x) + c = \fm(x)$ for $\bP_{X}$-a.e.\ $x\in\cX$, which implies $f_g = \fm(g) = \fh(g) + \innerprod{c}{g}_{\cG}$ $\bP_{X}$-a.e.\ in $\cX$.
	Since $\fh(g)\in\cH$ and $\innerprod{c}{g}_{\cG}\in\bR$ for each $g\in\cG$, this shows that \ref{assump:weakerCME} $\implies$ \ref{assump:weakerCMEold}.
	
	Let $\fh\in \cG\otimes \cH$ be such that $[\fh] = P_{\overline{(\cG\otimes \cH)_{\cC}}^{L_{\cC}^{2}(\bP_{X};\cG)}} [\fm]$.
	Letting $c\defeq \bE[(\fm - \fh)(X)] \in \cG$ and denoting the unit constant function by $\mathds{1}\in L^2(\bP_{X})$, it follows that, for each $h\in\cH$ and $g\in\cG$,
	\begin{align*}
	0
	&=
	\innerprod{[\fh] - [\fm]}{[g\otimes h]}_{L_{\cC}^{2}(\bP_{X};\cG)}
	\\
	&=
	\innerprod{\fh + c \otimes \mathds{1} - \fm}{g\otimes h - g\otimes \bE[h(X)]}_{\cG \otimes L^2(\bP_{X})}
	\\
	&=
	\innerprod{\fh(g) + \innerprod{c}{g}_{\cG} \mathds{1} - \fm(g)}{h - \bE[h(X)]}_{L^2(\bP_{X})}
	\\
	&=
	\innerprod{[\fh(g)] - [\fm(g)]}{[h]}_{L_{\cC}^{2}(\bP_{X})},
	\end{align*}
	where we used \eqref{equ:SomeTensorProductIdentities} and $\innerprod{c}{g}_{\cG} = \bE[(\fm(g) - \fh(g))(X)]$.
	Since $\fh(g)\in\cH$, this shows that \ref{assump:weakCME} $\implies$ \ref{assump:weakCMEold}.
	
	If $\fh \defeq P_{\overline{\cG\otimes \cH}^{L^{2}(\bP_{X};\cG)}} \fm \in \cG\otimes \cH$, then, for each $h\in\cH$ and $g\in\cG$,
	\[
	0
	=
	\innerprod{\fh - \fm}{g\otimes h}_{\cG \otimes L^2(\bP_{X})}
	=
	\innerprod{\fh(g) - \fm(g)}{h}_{L^2(\bP_{X})},
	\]
	where we used \eqref{equ:SomeTensorProductIdentities}.
	Since $\fh(g)\in\cH$, this shows that \ref{assump:weakCMEuncentred} $\implies$ \ref{assump:weakCMEuncentredold}.
	
	If $\fm \in \overline{\cG\otimes \cH}^{L^{2}(\bP_{X};\cG)}$, then there exists a sequence $(\fh_{n})_{n\in\bN}$ in $\cG\otimes \cH$ such that $\norm{\fh_{n} - \fm}_{\HS(\cG;L^2(\bP_{X}))} \to 0$ as $n\to\infty$.
	Let $g\in\cG$ and $h_n \defeq \fh_{n}(g)\in\cH$, $n\in\bN$.
	Then
	\[
	\norm{h_{n} - f_g}_{L^2(\bP_{X})}
	=
	\norm{\fh_{n}(g) - \fm(g)}_{L^2(\bP_{X})}
	\le
	\norm{\fh_{n} - \fm}_{\HS(\cG;L^2(\bP_{X}))}\, \norm{g}_{\cG}
	\xrightarrow[n\to\infty]{}
	0,
	\]
	which proves \ref{assump:strongCMElimit} $\implies$ \ref{assump:strongCMElimitold}.
	
	Finally, let $[\fm] \in \overline{(\cG\otimes \cH)_{\cC}}^{L_{\cC}^{2}(\bP_{X};\cG)}$.
	Then there exists a sequence $(\fh_{n})_{n\in\bN}$ in $\cG\otimes \cH$ such that $\norm{[\fh_{n}] - [\fm]}_{\cG \otimes L_{\cC}^2(\bP_{X})}\to 0$ as $n\to\infty$, i.e.\ $\norm{\fh_{n} + c_n \otimes \mathds{1} - \fm}_{\cG \otimes L^2(\bP_{X})}\to 0$ as $n\to\infty$ where $c_n \defeq \bE[(\fm-\fh_{n})(X)]\in\cG$ and $\mathds{1}\in L^2(\bP_{X})$ is the unit constant function.
	Let $g\in\cG$, $h_n \defeq \fh_{n}(g)\in\cH$ and $r_n \defeq \innerprod{c_n}{g}_{\cG} \mathds{1}$, $n\in\bN$.
	Then, as $n\to\infty$,
	\[
	\norm{h_{n} + r_n - f_g}_{L^2(\bP_{X})}
	=
	\norm{\fh_{n}(g) + \innerprod{c_n}{g}_{\cG} \mathds{1} - \fm(g)}_{L^2(\bP_{X})}
	\le
	\norm{\fh_{n} + c_n \otimes \mathds{1} - \fm}_{\HS(\cG;L^2(\bP_{X}))} \norm{g}_{\cG}
	\to 0,
	\]
	which proves \ref{assump:weakerCMElimit} $\implies$ \ref{assump:weakerCMElimitold}.
	
	The last two implications follow from the above and \citet[Theorems~4.1 and 5.1]{klebanov2019rigorous}.
\end{proof}


\begin{proof}[Proof of \Cref{lemma:CharacteristicImpliesBstar}]
	Suppose that $(\cG\otimes \cH)_{\cC}$ is not dense in $L_{\cC}^{2}(\bP_{X};\cG)$.
	Then there exists $\ff\in L^2(\bP_{X};\cG)$ that is not $\bP_{X}$-a.e.\ constant (i.e.\ there is no $c\in\cG$ such that $\ff(x)=c$ for $\bP_{X}$-a.e.\ $x\in\cX$) such that $[\ff]\perp_{ L_{\cC}^{2}(\bP_{X};\cG)} (\cG\otimes \cH)_{\cC}$.
	Let $\tilde{\ff} \coloneqq \ff-\bE[\ff(X)]$ and $\bP_{\tilde{\ff}} = \tilde{\ff}_{\#}\bP_{X}$ denote the pushforward measure of $\bP_{X}$ under $\tilde{\ff}$.
	
	Then there exists $g_{\ast}\in\supp (\bP_{\tilde{\ff}}) \subseteq \cG$ such that $g_{\ast}\neq 0$ (otherwise $\bP_{\tilde{\ff}}(\cG\setminus \{0\}) = 0$ and therefore $\tilde{\ff} = 0$ $\bP_{X}$-a.e.).
	Hence, after proper normalisation of $\tilde{\ff}$, we can define the following two \emph{distinct} probability measures on $\cX$:	
	\[
	Q_1(E)
	\defeq
	\int_{E} \absval{ \innerprod{\tilde{\ff}(x)}{g_\ast}_{\cG}}\, \rd \bP_{X}(x),
	\qquad
	Q_2(E)
	\defeq
	\int_{E} \absval{ \innerprod{\tilde{\ff}(x)}{g_\ast}_{\cG}} - \innerprod{\tilde{\ff}(x)}{g_\ast}_{\cG}\, \rd \bP_{X}(x)
	\]
	for every measurable subset $E\subseteq\cX$.	
	Indeed, for $\varepsilon \defeq \norm{g_{\ast}}_{\cG}/2 > 0$ and any $g = g_{\ast} + w \in B_{\varepsilon}(g_{\ast}) \defeq \{ g_{\ast} + w \mid w \in\cG,\, \norm{w}_{\cG}<\varepsilon \}$, the reverse triangle inequality implies $\innerprod{g}{g_\ast}_{\cG} \ge \norm{g_{\ast}}_{\cG}^{2} - \norm{w}_{\cG}\norm{g_{\ast}}_{\cG} > 2\varepsilon^{2}$.
	Hence, since $g_{\ast}\in \supp (\bP_{\tilde{\ff}})$, it follows that, for $E = \tilde{\ff}^{-1}(B_{\varepsilon}(g_{\ast}))$,
	\[
	Q_1(E) - Q_2(E)
	=
	\int_{E} \innerprod{\tilde{\ff}(x)}{g_\ast}_{\cG}\, \rd \bP_{X}(x)
	\ge
	2 \varepsilon^{2}\, \bP_{X}(E)
	=
	2 \varepsilon^{2}\, \bP_{\tilde{\ff}}(B_{\varepsilon}(g_{\ast}) )
	>
	0.
	\]
	Since, for every $\fh \in \cG\otimes \cH$,
	\[
	\innerprod{ \tilde{\ff} }{ \fh }_{L^{2}(\bP_{X};\cG)}
	=
	\innerprod{ \ff-\bE[\ff(X)] }{ \fh }_{L^{2}(\bP_{X};\cG)}
	\stackrel{[\ff]\perp (\cG\otimes \cH)_{\cC} }{=}
	\innerprod{ \ff-\bE[\ff(X)] }{ \bE[\fh(X)] }_{L^{2}(\bP_{X};\cG)}
	=
	0,
	\]
	it follows that $\tilde{\ff} \perp_{L^{2}(\bP_{X};\cG)} \cG\otimes \cH$.
	Let $Z_1\sim Q_1$, $Z_2\sim Q_2$ and $x\in\cX$.
	Since $g_{\ast} \otimes \varphi(x) \in \cG\otimes \cH$,
	\[
	\bigl( \bE[\varphi(Z_1)] - \bE[\varphi(Z_2)] \bigr)(x)
	=
	\int_{\cX} k(x,x')\, \innerprod{\tilde{\ff}(x')}{g_\ast}_{\cG}\, \rd \bP_{X}(x')
	=
	\innerprod{ g_{\ast} \otimes \varphi(x) }{ \tilde{\ff} }_{ L^2(\bP_{X};\cG)}
	=
	0,
	\]
	where we used \eqref{equ:IdentityFromViewingOtimesAsHS}, contradicting the assumption of $k$ being characteristic.	
\end{proof}


\begin{proof}[Proof of \Cref{thm:CMEunderB}]
	By \Cref{assumption:AssumptionHierarchyCMEalmosteverywhere}\ref{assump:weakerCME}, $\fm(x) = \ff(x) + c$ with $\ff\in \cG\otimes \cH$ and $c \in \cG$.
	As discussed in \Cref{remark:ViewTensorProductAsHS}, we can view $\ff\in \cG\otimes \cH$ as an element of both $L^2(\bP_{X};\cG)$ and $\HS(\cH;\cG)$, and thereby $\fm$ as an element of $\AHS(\cH;\cG)$.
	Hence, \eqref{equ:IdentityFromViewingOtimesAsHS} and the injectivity of $\varphi$ imply that
	\[
		\bE[U|V]
		=
		\bE[\psi(Y)|X]
		=
		\fm(X)
		=
		\ff(X) + c
		=
		\ff(\varphi(X)) + c
		=
		\fm\circ V
		\quad
		\text{a.s.}
	\]
	Since $\fm \in \AHS(\cH;\cG)\subseteq \ABLin(\cH;\cG)$, the statements follow from \Cref{theorem:LinearConditionalMeanUnderRangeAssumption};
	the inclusion $\ran C_{VU}\subseteq \ran C_{V}$ follows from \Cref{assumption:AssumptionHierarchyCMEalmosteverywhere}\ref{assump:weakerCME} by \Cref{proposition:AssumptionHierarchyCMEimpliesOldVersion}, cf.\ \Cref{fig:HierarchyOfAssumptions}.
\end{proof}


\begin{proof}[Proof of \Cref{thm:CMEunderBstar}]
	By \Cref{assumption:AssumptionHierarchyCMEalmosteverywhere}\ref{assump:weakerCMElimit}, there exists a sequence $\ff^{(n)}\in \cG\otimes \cH$, $n\in\bN$, such that
	\[
		\bignorm{[\ff^{(n)}] - [\fm]}_{L_{C}^{2}(\bP_{X};\cG)} \xrightarrow[n\to\infty]{} 0 .
	\]
	Therefore, denoting $c^{(n)} \defeq \bE[\fm(X) - \ff^{(n)}(X)] \in\cG$,
	\[
		\bignorm{\ff^{(n)}(X) + c^{(n)} - \fm(X)}_{L^{2}(\bP;\cG)} \xrightarrow[n\to\infty]{} 0.
	\]
	As discussed in \Cref{remark:ViewTensorProductAsHS}, $\ff^{(n)}\in \cG\otimes \cH$ can be seen as an element of both $L^2(\bP_{X};\cG)$ and $\HS(\cH;\cG)$.
	Hence, \eqref{equ:IdentityFromViewingOtimesAsHS} and the injectivity of $\varphi$ imply
	\begin{align*}
		\bE[U|V]
		& =
		\bE[\psi(Y)|X]
		=
		\fm(X)
		=
		\lim_{n\to\infty} \ff^{(n)}(X) + c^{(n)} \\
		& =
		\lim_{n\to\infty} \ff^{(n)}(\varphi(X)) + c^{(n)}
		=
		\lim_{n\to\infty} \gamma^{(n)}\circ V
		\quad
		\text{a.s.,}
	\end{align*}
	where the limits are in $L^{2}(\bP;\cG)$ and $\gamma^{(n)}(h) \defeq \ff^{(n)}(h) + c^{(n)}$.
	Since $\gamma^{(n)} \in \AHS(\cH;\cG)\subseteq \ABLin(\cH;\cG)$,
	this implies that $\bE[U|V] \in \overline{\ABLin_{\cG}\circ V}$ and thereby $\bE[U|V] = \bE^{\ABLin}[U|V]$.
	The claim now follows from \Cref{theorem:LinearConditionalMeanIncompatibleCase}.
\end{proof}

%% file: Section_AppendixTechnicalResults.tex

\section{Technical Results}
\label{section:TechnicalResults}

The following well-known result due to \citet[Theorem~1]{douglas1966majorization} (see also \citealp[Theorem~2.1]{fillmore1971operator}) is used several times.


\begin{theorem}
	\label{theorem:DouglasExistenceQ}
	Let $\cH$, $\cH_1$, and $\cH_2$ be Hilbert spaces and let $A\colon \cH_1\to \cH$ and $B\colon \cH_2\to \cH$ be bounded linear operators with $\ran A\subseteq \ran B$.
	Then $Q \defeq B^\dagger A\colon \cH_1 \to \cH_2$ is a well-defined and bounded linear operator, where $B^\dagger$ denotes the Moore--Penrose pseudo-inverse of $B$.
	It is the unique operator that satisfies the conditions
	\begin{equation}
		A = BQ,
		\qquad
		\ker Q = \ker A,
		\qquad
		\ran Q \subseteq \overline{\ran B^{\ast}}.
	\end{equation}
\end{theorem}


\begin{remark}
	In the original work of \citet{douglas1966majorization} only the existence of a bounded operator $Q$ such that $A = BQ$ was shown.
	However, the construction of $Q$ in the proof is identical to that of $B^\dagger$ (multiplied by $A$).
	This connection has been observed before by \citet[Corollary~2.2 and Remark~2.3]{arias2008gi_douglas}, where it was proven in the case of closed range operators, leaving the proof of the general case to the reader.
	Moreover, \citet[Theorem~1]{douglas1966majorization} only treats the case $\cH = \cH_1 = \cH_2$; the general case is mentioned as a remark at the end of his paper.
\end{remark}

Further, we are going to use the following characterisations of cross-covariance operators due to \citet[Theorem~1]{baker1973joint}.


\begin{theorem}
	\label{theorem:BakerDecompositionCrosscovarianceOperator}
		Under the notation of \Cref{section:GeneralSetupAndNotation}, there exists a unique bounded linear operator $R_{VU}\colon \cG\to\cH$ with operator norm $\norm{R_{VU}}\le 1$ such that
		\begin{equation}
			\label{equ:BakerDecomposition}
			C_{VU} = C_V^{1/2} R_{VU} C_U^{1/2},
			\qquad
			R_{VU}
			=
			P_{(\ker C_V)^\perp} R_{VU} P_{(\ker C_U)^\perp}.
		\end{equation}
\end{theorem}


\begin{remark}
	If $\cH = \cG = \bR$, then $R_{VU}$ coincides with the Pearson correlation coefficient.
\end{remark}


This paper makes extensive use of the following two basic results.

\begin{lemma}
	\label{lemma:TraceOfProductImpliesZero}
	Let $A \colon \cH\to\cG$ be a trace-class operator such that $\trace(AB^{\ast}) = 0$ for any bounded operator $B\in \BLin(\cH;\cG)$.
	Then $A = 0$.
\end{lemma}


\begin{proof}
	Choosing $B = A$ yields $\norm{A}_{\HS} = \trace(A A^{\ast})^{1/2} = 0$, hence $A=0$.
\end{proof}



\begin{lemma}
	\label{lemma:L2InnerProductTraceCovariance}
	With the notation of \Cref{section:GeneralSetupAndNotation}, let $U'\in L^2(\Omega, \Sigma, \bP;\cG)$ and $\gamma \in \ALin_V(\cH;\cG)$.
	Then
	\begin{enumerate}[label=(\alph*)]
		\item
		\label{item:L2InnerProductTraceCovariance}
		$\displaystyle
		\innerprod{U-\mu_{U}}{U'}_{L^{2}(\bP;\cG)}
		=
		\trace\bigl( \Cov[U,U'] \bigr)$;
		\item
		\label{item:PullOutLinearOperatorCovariance}
		$\displaystyle
		\Cov[\gamma\circ V, U] = \overline{\gamma} \, C_{VU}$
		and
		$\displaystyle
		\Cov[U,\gamma\circ V] = (\overline{\gamma} \, C_{VU})^{\ast}$.
	\end{enumerate}
	If $\gamma\in \ABLin(\cH;\cG)$, then the last equation can be simplified to $\Cov[U,\gamma\circ V] = C_{UV} \, \overline{\gamma}^{\ast}$.
\end{lemma}


\begin{proof}
	Let $(e_j)_{j\in \cJ}$, $\cJ\subseteq\bN$ be an orthonormal basis of $\cG$. Then
	\begin{align*}
		\innerprod{U - \mu_{U}}{U'}_{L^{2}(\bP;\cG)}
		&=
		\innerprod{U - \mu_{U}}{U' - \mu_{U'}}_{L^{2}(\bP;\cG)}
		\\
		&=
		\bE\bigl[ \innerprod{U - \mu_{U}}{U' - \mu_{U'}}_{\cG} \bigr]
		\\
		&=
		\sum_{j\in\cJ} \bE\bigl[ \innerprod{e_{j}}{U - \mu_{U}}_{\cG}\, \innerprod{U' - \mu_{U'}}{e_{j}}_{\cG} \bigr]
		\\
		&=
		\sum_{j\in\cJ} \innerprod{e_{j}}{C_{UU'} e_{j}}_{\cG}
		\\
		&=
		\trace \bigl[\Cov[U,U']\bigr],
	\end{align*}
	proving \ref{item:L2InnerProductTraceCovariance}.
	Since $V\in L^2(\bP;\cH)$ and $\gamma \circ V\in L^2(\bP;\cG)$, all covariance operators are well defined and so, for $g\in\cG$,
	\[
		\Cov[\gamma V , U](g)
		=
		\bE[ \gamma (V-\mu_{V}) \innerprod{U - \mu_{U}}{g}_{\cG} ]
		=
		\gamma \, \bE[ (V-\mu_{V}) \innerprod{U - \mu_{U}}{g}_{\cG} ]
		=
		\gamma \, \Cov[V , U](g),
	\]
	proving \ref{item:PullOutLinearOperatorCovariance}.
\end{proof}
